\documentclass[11pt,reqno]{amsart}
\usepackage[foot]{amsaddr}
\usepackage{amssymb,nicefrac,bm,upgreek,verbatim}
\usepackage[mathscr]{eucal}
\usepackage[numbers,sort]{natbib}
\usepackage{dsfont}
\usepackage{graphicx,subfigure}
\usepackage[margin=10pt,font=small,labelfont=bf]{caption}
\usepackage{wrapfig}
\usepackage{hyperref}

\usepackage[margin=1in]{geometry}
\allowdisplaybreaks

\usepackage{color}
\definecolor{dmagenta}{rgb}{.4,.1,.5}
\definecolor{dblue}{rgb}{.0,.0,.6}
\definecolor{ddblue}{rgb}{.0,.0,.4}
\definecolor{dred}{rgb}{.5,.0,.0}
\definecolor{dgreen}{rgb}{.0,.4,.0}
\definecolor{Eeom}{rgb}{.0,.0,.5}

\newcommand{\rd}[1]{{\color{red}#1}}
\renewcommand{\rd}[1]{{\color{black}#1}}

\newtheorem{lemma}{Lemma}[section]
\newtheorem{theorem}{Theorem}[section]
\newtheorem{proposition}{Proposition}[section]
\newtheorem{corollary}{Corollary}[section]

\theoremstyle{definition}
\newtheorem{definition}{Definition}[section]
\newtheorem{assumption}{Assumption}[section]

\theoremstyle{remark}
\newtheorem{remark}{Remark}[section]
\numberwithin{equation}{section}

\newcommand{\cA}{{\mathcal{A}}}  %used
\newcommand{\eom}{{\mathscr{G}}} %set of ergodic occupation measures
\newcommand{\sH}{{\mathscr{H}}}  %Half space for diffusion with constraints
\newcommand{\Lg}{\mathcal{L}}    %generator
\newcommand{\cL}{{\mathscr{L}}}  %generator
\newcommand{\cM}{{\mathcal{M}}}  %Invariant Measures
\newcommand{\cP}{{\mathcal{P}}}  %Probability measures
\newcommand{\sS}{{\mathscr{S}}}  %Integer Lattice of points
\newcommand{\sU}{{\mathscr{U}}}  %Set of Markov controls for diffusion
\newcommand{\Lyap}{{\mathcal{V}}}  %Lyapunov
\newcommand{\sV}{{\mathscr{V}}}  %used
\newcommand{\sX}{{\mathscr{X}}}  %JWC set
\newcommand{\cX}{{\mathcal{X}}}  % Polish Space
\newcommand{\cZn}{{\mathcal{Z}^n}} %Action Set
\newcommand{\bcZn}{{\Breve{\mathcal{Z}}^n}} %Action Set
\newcommand{\fZ}{{\mathfrak{Z}}} %Policies

\newcommand{\RR}{\mathds{R}}
\newcommand{\NN}{\mathds{N}}
\newcommand{\ZZ}{\mathds{Z}}
\newcommand{\Rd}{\mathds{R}^{d}}
\DeclareMathOperator{\Exp}{\mathbb{E}}
\DeclareMathOperator{\Prob}{\mathbb{P}}
\newcommand{\D}{\mathrm{d}}
\newcommand{\E}{\mathrm{e}}

\newcommand{\Act}{{\mathbb{U}}}
\newcommand{\Uadm}{\mathfrak{U}}
\newcommand{\Usm}{\mathfrak{U}_{\mathrm{SM}}}
\newcommand{\Ussm}{\mathfrak{U}_{\mathrm{SSM}}}

\newcommand{\sF}{\mathfrak{F}}  % sigma field
\newcommand{\Ind}{\mathds{1}}   % indicator function
\newcommand{\Cc}{\mathcal{C}}   %Continuous Functions

\newcommand{\abs}[1]{\lvert#1\rvert}
\newcommand{\norm}[1]{\lVert#1\rVert}
\newcommand{\babs}[1]{\bigl\lvert#1\bigr\rvert}

\newcommand{\babss}[1]{\biggl\lvert#1\biggr\rvert}

\newcommand{\transp}{^{\mathsf{T}}}
\newcommand{\df}{:=}
\DeclareMathOperator*{\osc}{osc}

\DeclareMathOperator*{\diag}{diag}

\newcommand{\order}{{\mathscr{O}}}
\newcommand{\sorder}{{\mathfrak{o}}}

\newcommand{\tc}{{\Breve\uptau}}

\newcommand{\qandq}{\quad\text{and}\quad}
\newcommand{\ttl}{\Large Infinite Horizon Average Optimality of the N-network\\[5pt]
Queueing Model in the Halfin--Whitt Regime}

\begin{document}
%%%%%%%%%%%%%%%%%%%%%%%%%%%%%%%%%%%%%%%%%%%%%%%%%%%%%%%%%%%%%%%%%%%%%
\allowdisplaybreaks

\title[Infinite Horizon Average Optimality of the N-network]
{\ttl}

\author{Ari Arapostathis$^\dag$}
\address{$^\dag$ Department of Electrical and Computer Engineering\\
The University of Texas at Austin\\
2501 Speedway St., EER 7.824\\
Austin, TX~~78712}
\email{ari@ece.utexas.edu}
\author{Guodong Pang$^\ddag$}
\address{$^\ddag$ The Harold and Inge Marcus Dept.
of Industrial and Manufacturing Eng.\\
College of Engineering\\
Pennsylvania State University\\
University Park, PA~~16802}
\email{gup3@psu.edu} 

%\date{\today}

\begin{abstract}
We study the infinite horizon optimal control problem for N-network queueing systems,
which consist of two customer classes and two server pools, under average (ergodic)
criteria in the Halfin--Whitt regime.  We consider three control objectives:
1) minimizing the queueing (and idleness) cost,
2) minimizing the queueing cost while imposing a constraint on idleness at each
server pool, and 3) minimizing the queueing cost while requiring fairness on idleness.
The running costs can be any nonnegative convex functions having at most
polynomial growth.

For all three problems we establish asymptotic optimality, namely, the convergence
of the value functions of the diffusion-scaled state process to the corresponding
values of the controlled diffusion limit.
We also present a simple state-dependent priority scheduling policy under
which the diffusion-scaled state process is geometrically ergodic in the
Halfin--Whitt regime, and some results on convergence of mean empirical measures
which facilitate the proofs. 
\end{abstract}



\maketitle

%%%%%%%%%%%%%%%%%%%%%%%%%%%%%%%%%%%%%%%%%%%%%%%%%%%%%%%%%%%%%%%%%%%%%
\section{Introduction}

Parallel server networks in the Halfin--Whitt regime have been very
actively studied
in recent years. Many important insights have been gained in their performance,
design and  control. 
One important question that has mostly remained open is optimal control under the
long-run average expected cost (ergodic) criterion. 
Since it is prohibitive to exactly solve the
discrete state Markov decision problem,
the plausible approach is to solve the control problem for the
limiting diffusion in the Halfin--Whitt regime and use this as an approximation. 
However, the results in the existing literature
for ergodic  control of diffusions
(see a good review in \citet{book}) cannot be directly
applied to the class of diffusion
models arising from the parallel server networks in the Halfin--Whitt regime. 
Recently, \citet{ABP14} and \citet{AP15} have developed the basic tools
needed to tackle this class of ergodic control problems. 

Given an optimal solution to the control problem for
the diffusion limit, the important task that remains is to show
it gives rise to a scheduling policy for the network and establish that
any sequence of such scheduling policies is asymptotically optimal in the
Halfin--Whitt regime.
Under the discounted cost criterion, this task has been accomplished in
\citet{atar-mandel-rei} for the multiclass \emph{V-model} (or \emph{V-network}),
 which consists of multiple customer classes
that are catered by servers in a single pool,
and in
\citet{Atar-05b} for multiclass multi-pool networks with certain tree topologies. 
Under the ergodic criterion, the problem becomes much more difficult because it is
intertwined with questions concerning the ergodicity of the diffusion-scaled
state process under the scheduling policies. 
This relates to various open questions on the stochastic stability
of parallel server networks in the Halfin--Whitt regime. 

\begin{wrapfigure}[14]{R}[0pt]{2.0in}
\centering
\includegraphics[width=2.0in]{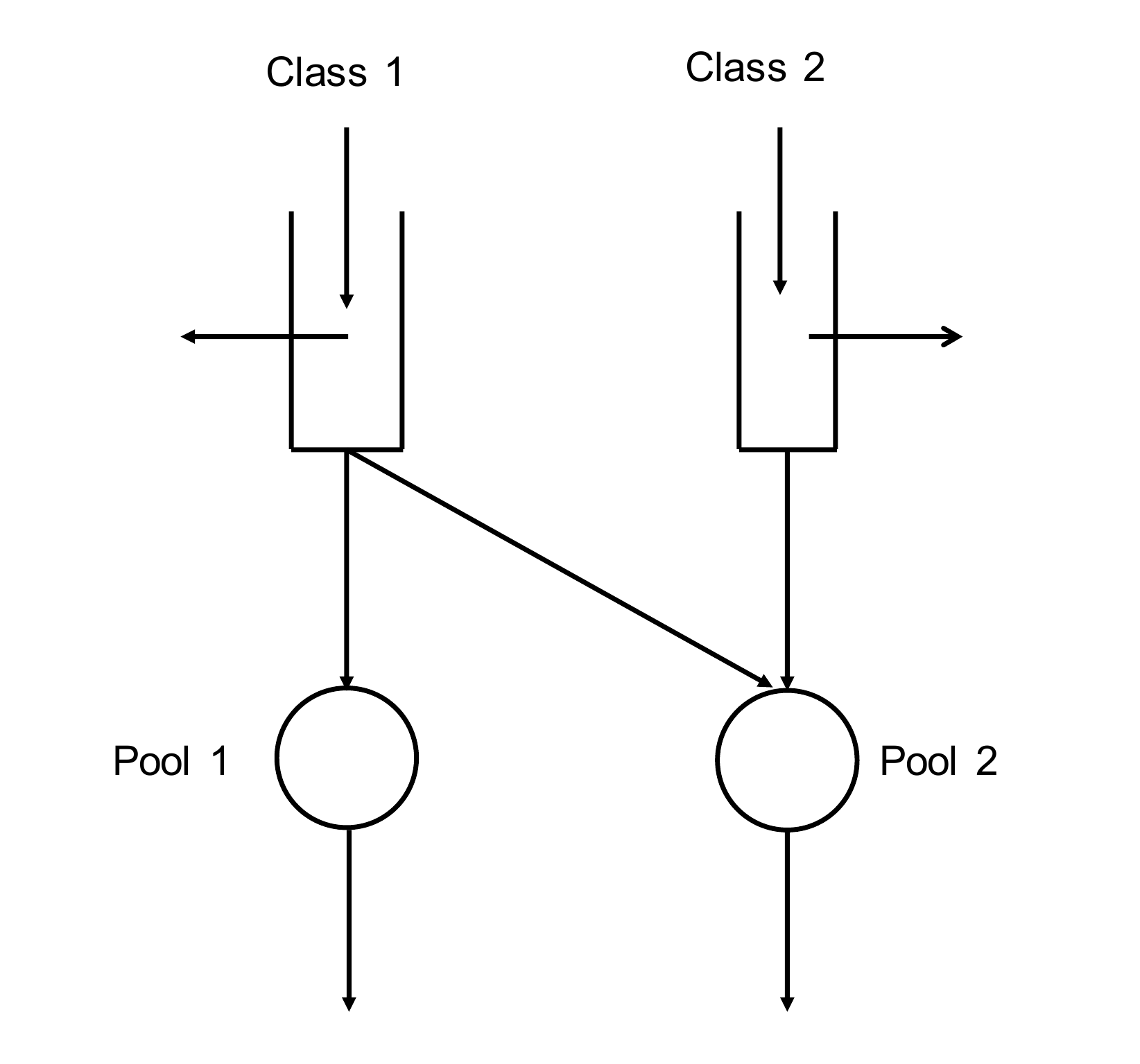}
\caption{The N-Network}
\label{N-model-plot}
\end{wrapfigure}
Stability of the multiclass V-model in the Halfin--Whitt regime is well treated
in \citet{gamarnik-stolyar}.
\citet{stolyar-15} has recently proved the tightness of the stationary distributions
of the diffusion-scaled state process for the so-called
\emph{N-network} (or \emph{N-model}), depicted in Figure~\ref{N-model-plot},
with no abandonment
under a static priority policy.  
For the V-network, \citet{ABP14} have shown that a sequence of scheduling
policies constructed from the optimal solution to the diffusion control problem
under the ergodic criterion is asymptotically optimal.
In this construction, the state space
is divided into a compact subset with radius in the order of the
square root of the number of servers around the steady state, and  its complement.
An approximation to the optimal control for the diffusion is used inside this
set, and a static priority policy is employed in its complement.
It follows from the results of \cite{ABP14} that under this sequence of
scheduling policies the state process is geometrically ergodic.
The proof of asymptotic optimality
takes advantage of the fact that, under the static priority scheduling policy,
the state process of the V-model in the Halfin--Whitt regime
is geometrically ergodic.
In fact, such a static priority policy for the V-model also corresponds to a
constant Markov control, under which the limiting diffusion is geometrically ergodic. 

However, for multiclass multi-pool networks, although the optimal
control problem for the limiting diffusion
has been thoroughly solved in \citet{AP15}, the lack of 
sufficient understanding of the stochastic stability properties of
the diffusion-scaled state process has been the critical obstacle to establishing
asymptotic optimality.
It is worth noting that this difficulty is
related to the so-called ``joint work conservation'' (JWC) condition which
plays a key role in the study of multiclass multi-pool networks as shown in
\citet{Atar-05a, Atar-05b}.
Although the JWC condition holds for the limiting diffusions over the entire
state space, it generally holds only in a bounded subset of the state space for
the diffusion-scaled process, whose
radius is in the order of the number of servers around the steady state. 
Thus, an optimal control derived from the limiting diffusion does not
translate well to a scheduling policy which is
compatible with the controlled dynamics of the network on the entire state space.
At the same time, although as shown in \cite{AP15} there exists a constant
Markov control under which the limiting diffusion of
multiclass multi-pool networks is geometrically ergodic, it
 is unclear if this is also the case for the diffusion-scaled
state processes under the corresponding static
priority scheduling policy. 
Therefore, the limiting diffusion does not offer much help in
the synthesis of a suitable scheduling policy on the part of the
state space where the JWC condition does not hold, and
as a result constructing stable policies for
multiclass multi-pool networks is quite a challenge. 

In this paper, we address these challenging problems for the
N-network.
We study three ergodic control problems: (P1) minimizing the queueing
(and idleness) cost, (P2) minimizing the queueing cost while imposing a 
constraint on the idleness of each server pool
(e.g., the long-run average idleness cannot exceed a specified
threshold), and (P3) minimizing the queueing cost while requiring fairness on
idleness (e.g., the average idleness of the two server pools satisfies a fixed
ratio condition).
The running cost can be any nontrivial nonnegative convex functions
having at most polynomial growth.
Under its usual parameterization, the control specifies
the number of customers from each class that are scheduled to each server
pool, and we refer to it as a ``scheduling'' policy.
However, the control can be also parameterized in a way so as to
specify which class of customers
should be scheduled to server pool~$2$ if it has any available servers
(``scheduling'' control), and which of the server pools
should class-$1$ customers be routed to, if both pools have available servers
 (``routing" control).
The optimal control problems for the limiting diffusion corresponding to
(P1)--(P3) are well-posed and in the case of (P1)--(P2) the
solutions can be fully characterized via HJB equations, following the methods
in \cite{AP15,ABP14}.
The dynamic programming characterization for (P3) is more difficult.
This is one of those rare examples in ergodic control where the running cost
is not bounded below or above, and there is no blanket stability property.
In this paper, we establish the existence of a solution to the HJB equation,
 and the usual characterization
of optimality for this problem. 

We first present a Markov scheduling policy, for the N-network under which the
diffusion-scaled state processes are geometrically ergodic in the Halfin--Whitt
regime (see Section~\ref{S3.2}). 
Unlike the V-model, this scheduling policy is a state-dependent priority 
(SDP) policy,
i.e., priorities change as the system state varies---yet it is simple to describe.
This result is significant 
since it indicates that the ergodic control problems for the diffusion-scaled
processes in the Halfin--Whitt regime have finite values.  
Moreover, it can be used as a scheduling policy outside a
bounded subset of the state space where the JWC property might fail to hold.
On the other hand, it follows from the theory in \citet{AP15} that the
controlled diffusion limit
is geometrically ergodic under some constant Markov control (see Theorem 4.2 in \cite{AP15}).
In this paper we show that a much stronger result applies for the N-network
(Lemma \ref{L4.1}):
as long as the scheduling control is a constant Markov control with pool~$2$
prioritizing class $2$ over $1$,  the controlled diffusion limit
is geometrically ergodic, uniformly over all routing controls
(e.g., class-$1$ customers prioritizing server pool $1$ over $2$,
or a state-dependent priority policy,
or even a non-stationary one). 

The main results of the paper center around the proof of
convergence of the value functions, which is accomplished by establishing
matching lower and upper bounds (see Theorems~\ref{T-lbound}--\ref{T-ubound}). 
To prove the lower bound, the key is to show that as long as the long-run
average first-order moment of the diffusion-scaled state process is finite, 
the associated mean empirical measures are tight
and converge to an ergodic occupation measure corresponding to a stationary
stable Markov control for the limiting diffusion (Lemma \ref{L7.1}).
In fact, we can show that for the N-network, under any  admissible
(work conserving) scheduling policy, the long-run average $m^{\rm th}$
($m\ge 1$) moment of the diffusion-scaled state process is bounded by the
 long-run average $m^{\rm th}$ moment
of the diffusion-scaled queue 
under that policy (Lemma \ref{L8.1}).
The lower bounds can then be deduced from  these observations.
It is worth noting that in order to establish
asymptotic optimality for the fairness problem (P3),
we must relax the equality in the constraint and show instead that the
constraint is asymptotically feasible.

In order to establish the upper bound, 
a Markov scheduling policy is synthesized which is the
concatenation of a Markov policy induced by the solution of
the ergodic control problem for the diffusion limit,
and which is applied on a bounded subset of the
state space where the JWC condition holds, and the SDP policy,
which is applied on the complement of this set.

The proof involves the following key components. 
First, we apply the spatial truncation approximation
technique developed in \citet{ABP14} and \citet{AP15} for the
ergodic control problem for the diffusion limit.
This provides us with an
$\epsilon$-optimal continuous precise control.
Second, we show that under the concatenation of the Markov scheduling policy
induced by
this $\epsilon$-optimal control and the SDP policy,
the diffusion-scaled state processes
are geometrically ergodic (Lemma \ref{L9.1}).
Then we prove that the mean empirical measures of the diffusion-scaled process
and control, converge to the ergodic occupation measure of the
diffusion limit associated with the $\epsilon$-optimal precise control
originally selected (Lemma \ref{L7.2}).
Uniform integrability implied by the geometric ergodicity takes care
of the rest.
 
\subsection{Literature review}
In a certain way, the N-network has been viewed as the benchmark of
multiclass multi-pool networks, mainly because it is simple to describe,
yet it has complicated enough dynamics.
There are several important studies on stochastic control of parallel server
networks, focusing on N-networks. 
\citet{XRS92} studied the Markovian single-server
N-network and showed that a threshold scheduling policy is optimal
under the expected discounted and long-run average linear holding cost,
utilizing a Markov decision process approach. 
In the conventional (single-server) heavy-traffic regime, the N-network
with two single severs,  was first studied in \citet{harrison-1998},
under the assumption of Poisson arrivals and deterministic services,
and a ``discrete-review" policy is shown to be
asymptotically optimal 
under an infinite horizon discounted linear queueing cost.
The N-model with renewal arrival processes and general service time distributions
was then studied in \citet{bell-williams-2001}, as a Brownian
control problem under an infinite horizon discounted linear queueing cost, and a
threshold policy is shown to be asymptotically optimal. 
\citet{GW13} studied the N-network with renewal arrival processes, general service time distributions and exponential patience times, and showed a two-threshold scheduling policy is asymptotically optimal via a Brownian control problem under an infinite horizon discounted linear queueing cost. 
Brownian control models for multiclass networks were pioneered in \citet{harrison-1988,harrison-2000}
and have been extended to many interesting networks; see \citet{williams-2016}
for an extensive review of that literature.

In the many-server Halfin--Whitt regime,  \citet{Atar-05a, Atar-05b} pioneered
the study of multiclass multi-pool networks with abandonment
(of a certain tree topology) via the corresponding control problems
for the diffusion limit under an
infinite-horizon discounted cost.
\citet{GW09, GW10} have studied queue-and-idleness-ratio controls for
multiclass multi-pool networks (including the N-network)
in the Halfin--Whitt regime by establishing a
State-Space-Collapse property, under certain assumptions on the network structure
and the system parameters. 
The N-network with many-server pools and abandonment has been recently studied
in \citet{TD-10}, where a static priority policy is shown to be asymptotically
optimal in the Halfin--Whitt regime under a finite-time horizon cost criterion.
In \citet{WA-13}, some blind fair routing policies are proposed for some
multiclass multi-pool networks (including the N-network), where the
control problems are formulated to minimize the average queueing cost under a
fairness constraint on the idleness.

On the other hand, most of the existing results on the stochastic control of
multiclass multi-pool networks in the Halfin--Whitt regime have only considered
either discounted cost criteria (\citet{Atar-05a, Atar-05b, Atar-09}) or
finite-time horizon cost criteria (\citet{dai-tezcan-08, dai-tezcan-11}).
There is only limited work of multiclass networks under
ergodic cost criteria.  
\citet{ABP14} have recently studied the multiclass V-model under ergodic cost
in the Halfin--Whitt regime. 
The inverted V-model is studied in  \citet{Armony-05}, and it  is shown that the
fastest-server-first policy is asymptotically optimal for minimizing the
steady-state expected queue length and waiting time.
For the same model, 
\citet{AW-10} showed that 
a threshold policy is asymptotically optimal for minimizing the steady-state
expected queue length and waiting time subject to a ``fairness" constraint
on the workload division. 
 \citet{Biswas-15} has recently studied a multiclass multi-pool network with
``help" under an ergodic cost criterion,
where each server pool has a dedicated stream
of a customer class,
and can help with other customer classes only when it has idle servers.
The N-network does not belong to the class of models considered
in \citet{Biswas-15}. 
For general multiclass multi-pool networks, \citet{AP15} have thoroughly
studied  ergodic control problems for the limiting diffusion.  
However, as mentioned earlier, asymptotic optimality has remained open.
This work makes a significant contribution in that direction,
by studying the N-network.
The fairness problem we study fills,
in some sense (our formulation is more general), the asymptotic optimality gap in
\citet{WA-13},
where  the associated approximate diffusion control problems are
studied via simulations.  

We also feel that this work contributes to the understanding of the
stability of multiclass multi-pool
networks in the Halfin--Whitt regime. 
In this topic, in addition to the stability studies of the V and N-networks in
\citet{gamarnik-stolyar} and \citet{stolyar-15},
it is worthwhile mentioning the following relevant work. 
\citet{stolyar-yudovina-12b} studied the stability of
multiclass multi-pool networks under a load balancing scheduling and routing
policy, ``longest-queue freest-server" (LQFS-LB).
They showed that the
fluid limit may be unstable in the vicinity of the equilibrium point for
certain network structures and system parameters, and that the sequence of
stationary
distributions of the diffusion-scaled processes may not be tight in both the
underloaded regime and the Halfin--Whitt regime.
They also provided positive answers to the stability and exchange-of-limit
results in the diffusion scale for one special class of networks. 
\citet{stolyar-yudovina-12a} proved the tightness of
the sequence of stationary distributions of multiclass multi-pool networks
under a \emph{leaf activity priority} policy (assigning static priorities
to the activities in the order of sequential ``elimination" of the tree leaves)
in the scale $n^{\nicefrac{1}{2}+\varepsilon}$ ($n$ is the scaling parameter) for
all $\varepsilon>0$, which was extended to the diffusion scale
$n^{\nicefrac{1}{2}}$ in \citet{stolyar-15}. 
The stability/recurrence properties for general multiclass multi-pool networks
under other scheduling policies remain open. 

As alluded above, the main challenge to establish asymptotic
optimality for general multiclass multi-pool networks is to understand
the stochastic stability/recurrence properties of the diffusion-scaled
state processes in the Halfin-Whitt regime.
Despite the recent development in
\cite{stolyar-yudovina-12a, stolyar-yudovina-12b, stolyar-15}, these are far
from being adequate for proving the asymptotic optimality for general multiclass
multi-pool networks. The stochastic stability/recurrence properties may depend
critically upon the network topology and/or parameter assumptions.
We believe that the methodology developed here for the N-network will provide
some important
insights on what stochastic stability properties are required and the roles
they may play in proving  asymptotic optimality.

\subsection{Organization of the paper}
The notation used in this paper is summarized in Section~\ref{S1.3}.
A detailed description of the N-network model is given in Section~\ref{sec-model}.
We define the control objectives in Section~\ref{S3.1} and present a
state-dependent priority policy that is geometrically stable in Section~\ref{S3.2}. 
We state the corresponding ergodic control problems for the limiting diffusion,
as well as the results on the
characterization of optimality in Section~\ref{S4}. 
The asymptotic optimality results are stated in Section~\ref{S5}.
We describe the system dynamics and an equivalent
control parameterization in Section~\ref{S6}. 
In Section~\ref{S7}, we establish convergence results for the mean empirical
measures for the diffusion-scaled state processes.
We then prove the lower and upper bounds in
Sections~\ref{S8} and~\ref{S9}, respectively.
The proof of geometric stability of the
SDP policy is given in Appendix~\ref{App1},
and Appendix~\ref{App2} is concerned with the proof of
Theorem~\ref{T4.3}.

\subsection{Notation}\label{S1.3}

The following notation is used in this paper.
The symbol $\RR$, denotes the field of real numbers,
and $\RR_{+}$, $\NN$, and $\ZZ$ denote the sets of nonnegative
real numbers, natural numbers, and integers, respectively.
Given two real numbers $a$ and $b$, the minimum (maximum) is denoted by $a\wedge b$ 
($a\vee b$), respectively.
Define $a^{+}\df a\vee 0$ and $a^{-}\df-(a\wedge 0)$. 
The integer part of a real number $a$ is denoted by $\lfloor a\rfloor$.
We also let $e\df (1,1)\transp$.

For a set $A\subset\Rd$, we use
$\Bar A$, $A^{c}$, and $\Ind_{A}$ to denote the closure,
the complement, and the indicator function of $A$, respectively.
A ball of radius $r>0$ in $\Rd$ around a point $x$ is denoted by $B_{r}(x)$,
or simply as $B_{r}$ if $x=0$.
The Euclidean norm on $\Rd$ is denoted by $\abs{\,\cdot\,}$,
$x\cdot y$
denotes the inner product of $x,y\in\RR^{d}$,
and $\norm{x}\df \sum_{i=1}^{d}\abs{x_{i}}$.

For a nonnegative function $g\in\Cc(\RR^{d})$ 
we let $\order(g)$ denote the space of functions
$f\in\Cc(\RR^{d})$ satisfying
$~\sup_{x\in\RR^{d}}\;\frac{\abs{f(x)}}{1\vee g(x)}<\infty$.
We also let $\sorder(g)$ denote the subspace of $\order(g)$ consisting
of those functions $f$ satisfying
$~\limsup_{\abs{x}\to\infty}\;\frac{\abs{f(x)}}{1\vee g(x)}\;=\;0\,.$
Abusing the notation, $\order(x)$ and $\sorder(x)$ occasionally denote
generic members of these sets.

We let $\Cc^{\infty}_{c}(\RR^{d})$ denote the set of smooth real-valued functions
on $\RR^d$ with compact support.
Given any Polish space $\cX$, we denote by $\cP(\cX)$ the set of
probability measures on $\cX$ and we endow $\cP(\cX)$ with the
Prokhorov metric.
For $\nu\in\cP(\cX)$ and a Borel measurable map $f\colon\cX\to\RR$,
we often use the abbreviated notation
$\nu(f)\df \int_{\cX} f\,\D{\nu}\,.$
The quadratic variation of a square integrable martingale is
denoted by $\langle\,\cdot\,,\cdot\,\rangle$.
For any path $X(\cdot)$ of a c\`adl\`ag process,
we use the notation $\Delta X(t)$ to denote the jump at time $t$.

%%%%%%%%%%%%%%%%%%%%%%%%%%%%%%%%%%%%%%%%%%%%%%%%%%%%%%%%%%%%%%%%%%%%%
\section{Model Description} \label{sec-model}

All stochastic variables introduced below are defined on a complete
probability space 
$(\Omega,\mathfrak{F},\Prob)$.  The expectation w.r.t. $\Prob$ is denoted by
$\Exp$. 

\subsection{The N-network model}

Consider an N-network with two classes of jobs (or customers) and two server pools,
as depicted in Figure~\ref{N-model-plot}. 
Jobs of each class arrive according to a Poisson process with rates
$\lambda_i^n$, $i=1,2$.
There are two server pools, each of which have multiple statistically identical
servers, and servers in pool~$1$ can only serve class-$1$ jobs, while servers
in pool~$2$ can serve both classes of jobs. Let $N_j^n$ be the number of
servers in pool~$j$, $j=1,2$. 
The service times of all jobs are exponentially distributed, where jobs of
class~$1$ are served at rates $\mu_{11}^n$ and $\mu_{12}^n$ by servers in
pools~$1$ and $2$, respectively, while jobs of class~$2$ are served at a rate
$\mu_{22}^n$ by servers in pool~$2$.
Throughout the paper we set $\mu_{21}^n\equiv 0$, and $\mu_{21} \equiv 0$.
Jobs may abandon while waiting in queue, with an exponential patience time with rate
$\gamma_i^n$ for $i=1,2$. 
We study a sequence of such networks indexed by an integer
$n$ which is the order of the number
of servers and let $n \to\infty$.

Throughout the paper we assume that the parameters satisfy the following conditions.

\begin{assumption} $($Halfin--Whitt Regime$)$ \label{as-para}
As $n\to\infty$, the following hold:
\begin{gather*}
\frac{\lambda^{n}_{i}}{n} \;\to\;\lambda_{i}\;>\;0\,,\qquad \frac{\lambda^{n}_{i}
- n \lambda_{i}}{\sqrt n} \;\to\;\Hat{\lambda}_{i}\,, \qquad 
\gamma_{i}^{n} \;\to\; \gamma_{i}\;\ge \;0\,, \qquad i=1,2\,,\\
\frac{N^{n}_{j}}{n} \;\to\;\nu_{j}\;>\;0\,,\qquad
 {\sqrt n}\,(n^{-1} N^{n}_{j} -  \nu_{j}) \;\to\; 0 \,,   \qquad j=1,2\,,\\
\mu_{ij}^{n} \;\to\;\mu_{ij}\;>\;0\,, \qquad
{\sqrt n}\,(\mu^{n}_{ij} - \mu_{ij}) \;\to\;\Hat{\mu}_{ij}\,, \qquad i,j =1,2\,.
\end{gather*}
We also have
\begin{equation}\label{HWpara4}
\lambda_1> \mu_{11} \nu_1\,, \qquad \frac{\lambda_1- \mu_{11}\nu_1}{\mu_{12} \nu_2}
+ \frac{\lambda_2}{\mu_{22} \nu_2} \;=\;1\,. 
\end{equation}
\end{assumption}

Note that \eqref{HWpara4} implies that class-$1$ jobs are
overloaded for server pool~$1$, class-$2$ jobs are underloaded for server
pool~$2$, and the overload
of class-$1$ jobs can be served by server pool~$2$ so that both server pools are
critically loaded. This assumption is referred to as the
\emph{complete resource pooling} condition (\citet{williams-2000,Atar-05b}). 

Let $\xi^*$ be a constant matrix
\begin{equation} \label{Exi*}
\xi^*\;\df\; \begin{bmatrix} 1 \quad&  \frac{\lambda_1
- \mu_{11}\nu_1}{\mu_{12} \nu_2} \\[5pt]
0 \quad & \frac{\lambda_2}{\mu_{22} \nu_2}  \end{bmatrix} \,.  
\end{equation}
The quantity $\xi_{ij}^*$  can be interpreted as the steady-state fraction of
service allocation of pool~$j$ to class-$i$ jobs in the fluid scale.
Define
 $x^*=(x^*_{i})_{i =1,2}$ and 
$z^* = (z^*_{ij})_{ i, j = 1,2}$ by
\begin{equation}\label{Ex*}
x^*_1 \;\df\; \xi^*_{11} \nu_1 + \xi^*_{12} \nu_2\,,
\qquad x^*_2 \;\df\; \xi^*_{22} \nu_2\,,  
\end{equation}
\begin{equation} \label{Ez*}
z^* \;=\; (z^*_{ij}) \;\df\; ( \xi^*_{ij} \nu_j) \;=\;
\begin{bmatrix} \nu_1 \quad&  \frac{\lambda_1- \mu_{11}\nu_1}{\mu_{12} } \\[5pt]
0 \quad & \frac{\lambda_2}{\mu_{22} }  \end{bmatrix}   \,.  
\end{equation}
Then $x_i^*$ can be interpreted as the steady-state total number of class-$i$ jobs,
and $z_{ij}^*$ can be interpreted as the steady-state number of class-$i$ jobs
receiving service in pool~$j$, in the fluid scale. 
It is easy to check that
$e\cdot x^* \;=\; e \cdot \nu$,
where $\nu \df (\nu_1, \nu_2)\transp$.

For each $i=1,2,$ let
$X^{n}_i = \{X^{n}_i(t): t\ge 0\}$ and $Q^{n}_i = \{Q^{n}_i(t): t\ge 0\}$
be the total number of class-$i$ jobs
in the system and in the queue, respectively. 
For each $j=1,2$, let $Y^{n}_j = \{Y^{n}_j(t): t\ge 0\}$ be the number of
idle servers in server pool~$j$.
For $i,j=1,2$, let $Z_{ij}^{n} = \{Z_{ij}^{n}(t): t\ge 0\}$ be the
number of class-$i$ jobs being served in server pool~$j$,
and note that $Z_{21}^n\equiv 0$. 
The following fundamental balance equations hold:
\begin{equation}\label{baleq}
\begin{split}
&\begin{aligned}
X^{n}_1(t)&\;=\;Q_1^{n}(t) +  Z_{11}^{n}(t) + Z_{12}^{n}(t)\,, \\[5pt]
X^{n}_2(t)&\;=\;Q_2^{n}(t) + Z_{22}^{n}(t)\,,
\end{aligned}
\qquad
\begin{aligned}
N_1^{n}&\;=\;Y_1^{n}(t) + Z_{11}^{n}(t)\,,  \\[5pt]
N_2^{n}&\;=\;Y_2^{n}(t) + Z_{12}^{n}(t) +  Z_{22}^{n}(t)\,,
\end{aligned}\\[5pt]
&\mspace{0mu}
X^{n}_i(t) \,\ge\, 0\,, \quad Q_i^{n}(t) \,\ge\, 0\,, \quad Y_j^{n}(t) \,\ge\, 0\,,
\quad Z_{ij}^{n}(t) \,\ge\, 0\,,\quad i,j =1,2\,,
\end{split}
\end{equation}
for each $t\ge0$.
We let $Z^{n} = (Z_{ij}^{n})_{i, j=1,2}$, $X^{n} = (X_i^{n})_{i=1,2}$,
and analogously define $Q^{n}$ and $Y^{n}$.

\subsection{Scheduling control}

We  only consider work conserving policies that are non-anticipative and
preemptive. 
Work conservation requires that the processes $Q^{n}$ and $Y^{n}$ satisfy
\begin{equation*}
Q_1^{n}(t) \wedge Y^{n}_j(t)\;=\;0\, \qquad \forall j = 1,2,
\qandq Q_2^{n}(t) \wedge Y^{n}_2(t)\;=\;0\,,  \quad\forall\, t \ge 0\,. 
\end{equation*}
In other words, no server will idle if there is any job in a queue
that the server can serve. 
Service preemption is allowed, that is, jobs in service at pool~$2$ can be
interrupted and resumed at a later time in order to serve jobs
from the other class.

Let
\begin{equation*}
\begin{aligned}
q_1(x,z) &\;\df\; x_1 - z_{11} - z_{12}\,,\qquad &
y_1^n(x,z) &\;\df\;   N_1^n - z_{11}\,,\\[5pt]
q_2(x,z) &\;\df\; x_2 - z_{22}\,,\qquad &
y_2^n(x,z) &\;\df\; N_2^n  - z_{12} - z_{22}\,.
\end{aligned}
\end{equation*}
We define the action set $\cZn(x)$ as
\begin{multline*}
\cZn(x)\;\df\; \bigl\{z \in \ZZ^{2 \times 2}_+\;\colon\, z_{21}=0\,,~
q_1(x,z)\wedge q_2(x,z)\wedge y_1^n(x,z)\wedge y_2^n(x,z)\ge0\,,\\
q_1(x,z) \wedge \bigl(y_1^n(x,z) +y_2^n(x,z)\bigr)=0\,,
\quad q_2(x,z) \wedge y_2^n(x,z)=0 \bigr\}\,. 
\end{multline*}

Define the $\sigma$-fields 
\begin{align*}
\mathcal{F}^{n}_t &\;\df\; \sigma \bigl\{ X^{n}(0),\; \Tilde{A}^{n}_i(s),\;
\Tilde{S}^{n}_{ij}(s),\; \Tilde{R}^{n}_i(s)\;\colon\, i, j = 1,2, 
\; 0 \le s \le t \bigr\} \vee \mathcal{N} \,,\\[5pt]
\mathcal{G}^{n}_t &\;\df\; \sigma \bigl\{ \delta\Tilde{A}^{n}_i(t, r),\;
\delta \Tilde{S}^{n}_{ij}(t, r),\; \delta\Tilde{R}^{n}_i(t, r)\;\colon\,
i, j =1,2, \; r \ge 0 \bigr\} \,,
\end{align*}
where  $\mathcal{N}$ is the collection of all $\Prob$-null sets, and
\begin{align*}
\Tilde{A}^{n}_i(t) &\;\df\; A^{n}_i(\lambda_i^{n} t),
&\quad \delta\Tilde{A}^{n}_i(t,r) &\;\df\;
\Tilde{A}^{n}_i(t+r) - \Tilde{A}^{n}_i(t) \,,\\[5pt]
\Tilde{S}^{n}_{ij}(t) &\;\df\; S^{n}_{ij} \left(\mu_{ij}^{n}
\int_0^t Z_{ij}^{n}(s)\,\D{s} \right), 
&\quad \delta \Tilde{S}^{n}_{ij}(t, r) &\;\df\; S^{n}_{ij}
\left( \mu_{ij}^{n}\int_0^t Z_{ij}^{n}(s)\,\D{s} + \mu_{ij}^{n} r \right)
- \Tilde{S}^{n}_{ij}(t) \,,\\[5pt]
\Tilde{R}^{n}_i(t) &\;\df\; R_i^{n}
\left(\gamma_i^{n} \int_0^t Q^{n}_i(s)\,\D{s} \right)\,, 
&\quad \delta \Tilde{R}^{n}_i(t, r) &\;\df\; R_i^{n}
\left(\gamma_i^{n} \int_0^t Q^{n}_i(s)\, \D{s}
+ \gamma_i^{n} r \right) -  \Tilde{R}^{n}_i(t) \,.
\end{align*}
The processes $A^{n}_i$, $S^{n}_{ij}$ and $R^{n}_i$ are all rate-1 Poisson
processes,  representing the arrival, service and abandonment quantities, respectively.
We assume that they are mutually independent, and also independent of the initial
condition $X^{n}_i(0)$. 
Note that quantities with subscript $i=2$, $j=1$ are all equal to zero.  
The filtration $ \mathbf{F}^{n}\df\{\mathcal{F}^{n}_t: t \ge 0\}$ represents
the information available up to time $t$, and the filtration 
$\mathbf{G}^{n}\df\{\mathcal{G}^{n}_t: t \ge 0\}$ contains the information
about future increments of the processes. 
We say that  a scheduling policy $Z^n$ is \emph{admissible} if 
\begin{enumerate}
\item[(i)] $Z^{n}(t)\in \cZn(X^n(t))$ for all $t\ge0$;
\smallskip
\item[(ii)] $Z^{n}(t)$ is adapted to $\mathcal{F}^{n}_t$;
\smallskip
\item[(iii)] $\mathcal{F}^{n}_t$ is independent of $\mathcal{G}^{n}_t$ at each time
$t\ge 0$;
\smallskip
\item[(iv)] for each $i,j\in\{1,2\}$, and for each
$t\ge 0$, the process  $\delta \Tilde{S}^{n}_{ij}(t, \cdot)$ agrees in law
with $S^{n}_{ij}(\mu_{ij}^{n}\,\cdot)$, and the process
$\delta \Tilde{R}^{n}_i(t, \cdot)$
agrees in law with 
$R^{n}_i (\gamma_i^{n} \cdot)$.  
\end{enumerate}
We denote the set of all admissible scheduling policies
$(Z^{n}, \mathbf{F}^{n}, \mathbf{G}^{n})$ by $\mathfrak{Z}^{n}$.
Abusing the notation we sometimes denote this as $Z^{n}\in\mathfrak{Z}^{n}$.

Following \citet{Atar-05b}, we also consider a stronger condition,
\emph{joint work conservation} (JWC), for preemptive scheduling policies. 
Namely, for each $x\in\ZZ^2_+$, there exists a rearrangement $z\in\cZn(x)$ of jobs in service
such that there is either no job in queue or no idling server in the system, satisfying
\begin{equation} \label{E-JWC-def}
e\cdot q(x,z) \wedge e \cdot y^n(x,z)
\;=\;0\,. 
\end{equation}
We let $\sX^{n}$ denote the set of all possible values of
$\ZZ^2_+$ for which the JWC condition \eqref{E-JWC-def} holds, i.e., 
\begin{equation*}
\sX^{n}\;\df\;\bigl\{x\in\ZZ^2_+\,\colon\, \text{\eqref{E-JWC-def} holds for some}~ z\in\cZn(x)\bigr\}\,. 
\end{equation*} 
Note that the set $\sX^{n}$ may not include all possible scenarios
of the system state $X^{n}(t)$ for finite $n$ at each time $t\ge 0$.

%Specifically, we let $\sX^{n}$ denote the set of all possible values of
%$\ZZ^2_+$ for which there exists a rearrangement $z$ of jobs in service
%such that there is either no job in queue or no idling server in the system,
%i.e.,
%\begin{equation*}
%\sX^{n}\;\df\;\bigl\{x\in\ZZ^2_+\,\colon\, ~e\cdot q(x,z) \wedge e \cdot y^n(x,z)
%\;=\;0 ~\text{for some}~ z\in\cZn(x)\bigr\}\,. 
%\end{equation*} 

We quote a result from \citet{Atar-05b}, which is used later.

\begin{lemma}[Lemma~3 in \citet{Atar-05b}]\label{L-JWC}
There exists a constant $c_{0}>0$ such that, the collection
of sets $\Breve{\sX}^{n}$ defined by
\begin{equation*}
\Breve{\sX}^{n}\;\df\;
\bigl\{x \in \ZZ^{2}_{+}\;\colon\,  \norm{x - n x^*}\le c_{0}\, n  \bigr\}\,,
\end{equation*} 
satisfies $\Breve{\sX}^{n}\subset\sX^n$ for all $n\in\NN$.
Moreover, for any $x,\, q,\, y \in\ZZ^2_+$ satisfying $e\cdot q\wedge e\cdot y=0$
and $e\cdot (x-q)= e\cdot (N^n-y)\ge0$, we have
\begin{equation}\label{L2.1B}
\begin{bmatrix}
N_1^n - y_1 \quad & x_1 - q_1-  (N_1^n - y_1) \\[3pt]
0 \quad & x_2 - q_2
\end{bmatrix} \;\in\; \cZn(x)\,.
\end{equation}
\end{lemma}

We need the following definition.

\begin{definition}\label{D2.1}
We fix some open ball $\Breve{B}$ centered at the origin,
such that $n (\Breve{B}+x^*)\subset \Breve{\sX}^n$ for all $n\in\NN$.
The \emph{jointly work conserving
action set $\bcZn(x)$ at $x$} is defined as the subset of $\cZn(x)$,
which satisfies
\begin{equation*}
\bcZn(x) \;\df\; \begin{cases}
\bigl\{ z\in\cZn(x)\,\colon\,~
e\cdot q(x,z) \wedge e \cdot y^n(x,z)\;=\;0\bigr\}
&\text{if\ \ } x\in n (\Breve{B}+x^*)\,,\\[3pt]
\cZn(x)&\text{otherwise.}
\end{cases}
\end{equation*}
We also define the associated admissible policies by
\begin{equation*}
\begin{split}
\Breve\fZ^n &\;\df\;
\bigl\{Z^{n}\in\fZ^n \,\colon\, Z^n(t)\in\bcZn\bigl(X^n(t)\bigr)
\,,\ \ \forall\, t\ge0\bigr\}\,,\\[5pt]
\boldsymbol\fZ &\;\df\;\{Z^{n}\in\Breve\fZ^n\,\colon\,n\in\NN\}\,.
\end{split}
\end{equation*}
We refer to the policies in $\boldsymbol\fZ$ as
\emph{eventually jointly work conserving} (EJWC).
\end{definition}

\begin{remark}
The ball $\Breve{B}$ is fixed in Definition~\ref{D2.1} only for convenience.
We could instead adopt a more general definition of $\boldsymbol\fZ$, without
affecting the results of the paper.
Let $\{D_n\,,\, n\in\NN\}$ be a collection of domains which
covers $\RR^2$ and satisfies
$D_n\subset D_{n+1}$,
and $\sqrt{n} D_n+ n x^*\subset\Breve{\sX}^n$ for all $n\in\NN$.
Then we redefine $\bcZn$ using Definition~\ref{D2.1} and replacing
$n (\Breve{B}+x^*)$ with $\sqrt{n} D_n+ n x^*$ and define $\boldsymbol\fZ$ analogously.
If $\{Z^n\}\subset\boldsymbol\fZ$, then, in the diffusion scale, JWC holds
on an expanding sequence of domains which cover $\RR^2$.
This is the reason behind the terminology EJWC. 
The EJWC condition plays a crucial role in the derivation of the
controlled diffusion limit.
Therefore, convergence of mean empirical measures of the diffusion-scaled state
process and control, and thus, also the lower and upper bounds for
asymptotic optimality are established for sequences
$\{Z^n,\,n\in\NN\}\subset\boldsymbol\fZ$. 
\end{remark}

%%%%%%%%%%%%%%%%%%%%%%%%%%%%%%%%%%%%%%%%%%%%%%%%%%%%%%%%%%%%%%%%%%%%%
\section{Ergodic Control Problems} %\label{S3}

We define the diffusion-scaled processes
$\Hat{Z}^{n}\,=\,(\Hat{Z}^{n}_{ij})_{i,j \in \{1,2\} }\,$,
$\Hat{X}^{n}\,=\,(\Hat{X}^{n}_1, \Hat{X}^{n}_2)\transp\,$,
and analogously for  $\Hat{Q}^{n}$ and $\Hat{Y}^{n}$, by
\begin{equation} \label{DiffDef}
\begin{aligned}
\Hat{X}^{n}_i(t) &\;\df\; \frac{1}{\sqrt{n}} (X_i^{n}(t) - n x^*_{i}) \,,  \\[5pt]
\Hat{Q}^{n}_i(t) &\;\df\; \frac{1}{\sqrt{n}} Q_i^{n}(t) \,,
\end{aligned}
\qquad
\begin{aligned}
\Hat{Z}^{n}_{ij}(t) &\;\df\; \frac{1}{\sqrt{n}} (Z_{ij}^{n}(t) - n z^*_{ij})\,,\\[5pt]
\Hat{Y}^{n}_j(t) &\;\df\; \frac{1}{\sqrt{n}} Y_j^{n}(t) \,,
\end{aligned}
\end{equation}
where $x^*$ and $z^*$ are defined in \eqref{Ex*}--\eqref{Ez*}.

\subsection{Control objectives} \label{S3.1}

We consider three control objectives, which address the queueing (delay)
and/or idleness costs in the system:  (i) \emph{unconstrained problem},
minimizing the queueing (and idleness) cost and
(ii) \emph{constrained problem}, minimizing the queueing cost while
imposing a constraint on idleness, and
(iii) \emph{fairness problem},  minimizing the queueing cost while imposing a constraint
on the idleness ratio between the two server pools.
The running cost is a function of the diffusion-scaled processes,
which are related to the unscaled ones by \eqref{DiffDef}.
For simplicity, in all three cost minimization problems, 
we assume that the initial condition $X^{n}(0)$ is deterministic
and $\Hat{X}^{n}(0) \to x \in \RR^2$ as $n \to \infty$.  
Let $\Hat{r}: \RR^2_{+}\times \RR^2_+ \to \RR_+$ be defined by
\begin{equation} \label{Erhat}
\Hat{r}(q,y)\;\df\;\sum_{i=1}^2 \xi_i q_i^m + \sum_{j=1}^2 \zeta_j y_j^m\,,
\quad q \in \RR^{2}_{+}\,, \;y \in \RR^2_+\,, \quad\text{for some~} m\ge 1\,, 
\end{equation}
where 
$\xi=(\xi_1,\xi_2)\transp$ is a strictly positive vector and
$\zeta=(\zeta_1,\zeta_2)\transp$
is a nonnegative vector. In the case $\zeta\equiv 0$,
only the queueing cost is minimized.
In (P1) below, idleness may be added as a penalty in the objective.
We denote by $\Exp^{Z^n}$ the expectation operator under
an admissible policy $Z^n$.

\begin{itemize}
\item[\textbf{(P1)}] (\emph{unconstrained problem})
The running cost penalizes the queueing (and idleness).
Let $\Hat{r}(q,y)$ be the running cost function as defined in \eqref{Erhat}. 
Given an initial state $X^{n}(0)$, and an admissible scheduling policy
$Z^{n} \in \mathfrak{Z}^{n}$, we define the diffusion-scaled cost criterion by
\begin{equation} \label{cost-ds}
J\bigl(\Hat{X}^{n}(0), Z^{n}\bigr) \;\df\;
\limsup_{T \to \infty}\;\frac{1}{T}\;
\Exp^{Z^n} \left[\int_{0}^{T}
\Hat{r}\bigl(\Hat{Q}^{n}(s),\Hat{Y}^{n}(s)\bigr)\,\D{s}\right]\,.
\end{equation} 
The associated cost minimization problem becomes 
\begin{equation*} 
\Hat{V}^{n}(\Hat{X}^{n}(0)) \;\df\;
\inf_{Z^{n} \in \mathfrak{Z}^{n} } J \bigl(\Hat{X}^{n}(0), Z^{n}\bigr)\,.  
\end{equation*} 

\item[\textbf{(P2)}] (\emph{constrained problem})
The objective here is to minimize the queueing cost while
imposing idleness constraints on the two server pools.
Let $\Hat{r}_{\mathsf{o}}(q)$ be the
running cost function corresponding to  $\Hat{r}$ in \eqref{Erhat} with
$\zeta \equiv 0$.
The diffusion-scaled cost criterion
$J_{\mathsf{o}}\bigl(\Hat{X}^{n}(0), Z^{n}\bigr)$ is defined
analogously to \eqref{cost-ds} with running cost
$\Hat{r}_{\mathsf{o}}(\Hat{Q}^{n}(s))$, that is, 
\begin{align*} 
J_{\mathsf{o}}\bigl(\Hat{X}^{n}(0), Z^{n}\bigr)
&\;\df\;\limsup_{T \to \infty}\;\frac{1}{T}\;
\Exp^{Z^n} \left[\int_{0}^{T}
\Hat{r}_{\mathsf{o}}\bigl(\Hat{Q}^{n}(s)\bigr)\,\D{s}\right]\,.\\
\intertext{Also define}
J_{\mathsf{c},j}\bigl(\Hat{X}^{n}(0), Z^{n}\bigr)
&\;\df\;\limsup_{T \to \infty}\;\frac{1}{T}\;
\Exp^{Z^n} \left[\int_{0}^{T}
\bigl(\Hat{Y}^{n}_j(s)\bigr)^{\Tilde{m}}\,\D{s}\right]\,,\qquad j=1,2\,,
\end{align*}
with $\Tilde{m}\ge1$.
The associated cost minimization problem becomes 
\begin{align} 
\Hat{V}^{n}_{\mathsf{c}}(\Hat{X}^{n}(0)) &\;\df\;
\inf_{Z^{n} \in \mathfrak{Z}^{n}} J_{\mathsf{o}}\bigl(\Hat{X}^{n}(0), Z^{n}\bigr)\,,
\nonumber\\[5pt]
\text{subject to}\quad
J_{\mathsf{c},j}\bigl(\Hat{X}^{n}(0), Z^{n}\bigr)
&\;\le\; \updelta_j \,, \quad j =1,2 \,,\label{constraint-ds}
\end{align}
where $\updelta = (\updelta_1, \updelta_2)\transp$ is a positive vector.

\item[\textbf{(P3)}] (\emph{fairness})
Here we minimize the queueing cost while keeping the average idleness
of the two server pools balanced.
Let $\uptheta$ be a positive constant and let $1\le\Tilde{m}< m$.
The associated cost minimization problem becomes 
\begin{align*} 
\Hat{V}^{n}_{\mathsf{f}}(\Hat{X}^{n}(0)) &\;\df\;
\inf_{Z^{n} \in \mathfrak{Z}^{n}} J_{\mathsf{o}}\bigl(\Hat{X}^{n}(0), Z^{n}\bigr)\,,
\nonumber\\[5pt]
\text{subject to}\quad
J_{\mathsf{c},1}\bigl(\Hat{X}^{n}(0), Z^{n}\bigr)
&\;=\;\uptheta
J_{\mathsf{c},2}\bigl(\Hat{X}^{n}(0), Z^{n}\bigr)\,. 
\end{align*}
\end{itemize}

We refer to $\Hat{V}^{n}(\Hat{X}^{n}(0))$,
$\Hat{V}^{n}_{\mathsf{c}}(\Hat{X}^{n}(0))$ and
$\Hat{V}^{n}_{\mathsf{f}}(\Hat{X}^{n}(0))$ as the diffusion-scaled optimal values
for the $n^{\rm th}$ system given the initial state $X^{n}(0)$,
for (P1), (P2) and (P3), respectively.

\begin{remark}
We choose running costs of the form \eqref{cost-ds}
mainly to simplify the exposition.
However, all the results of this paper still hold
for more general classes of functions.
Let $h_{\mathsf{o}}\colon\RR^2\to\RR_+$ be a convex function satisfying
$h_{\mathsf{o}}(x)\ge c_{1}\abs{x}^m + c_{2}$ for some $m\ge1$ and constants
$c_1>0$ and $c_2\in\RR$, and
$h\colon\RR^2\to\RR_+$, $h_i\colon\RR\to\RR_+$, $i=1,2$, be convex functions
that have at most polynomial growth.
Then we can choose
$\Hat{r}(q,y)=h_{\mathsf{o}}(q) + h(y)$ for the unconstrained problem,
and $h_i(y_i)$ as the functions in the constraints in \eqref{constraint-ds}
(with $\Hat{r}_{\mathsf{o}}=h_{\mathsf{o}}$).
For the problem (P3) we require in addition that
$h_1=h_2\ne0$ and they are in $\sorder(\abs{x}^m)$.
The analogous running costs can of course be used in the corresponding
control problems for the limiting diffusion, which are presented later
in Section~\ref{S4.2}.
\end{remark}

\subsection{A geometrically stable scheduling policy} \label{S3.2}
We introduce a Markov scheduling policy for the N-network that results in geometric
ergodicity for the diffusion-scaled state process, and also implies that
the diffusion-scaled cost in the ergodic control problem (P1) is bounded,
uniformly in $n\in\NN$.
Let $N^n_{12} \df \lfloor \xi^*_{12}N^n_2 \rfloor$ and
$N^n_{22} \df \lceil \xi^*_{22}N^n_2 \rceil$.
Note that $N^n_{12} + N^n_{22} = N^n_2$.

\begin{definition}\label{D3.1}
For each $n$, we define the scheduling policy
$\Check{z}^n=\Check{z}^n(x)$, $x\in \ZZ^2_+$, by
\begin{align*}
\Check{z}^n_{11}(x) &\;=\; x_1 \wedge N_1^n\,, \\[5pt]%\label{ge-Zn11}
\quad \Check{z}^n_{12}(x) &\;=\; \begin{cases}
(x_1 - N_1^n)^+ \wedge N_{12}^n &\text{if~} x_2\ge N_{22}^n\\[3pt]
(x_1 - N_1^n)^+ \wedge (N_2^n - x_2)&\text{otherwise,}
\end{cases} \\[5pt]%\label{ge-Zn12}
\quad \Check{z}^n_{22}(x) &\;=\; \begin{cases}
x_2 \wedge N_{22}^n &\text{if~} x_1\ge N_1^n+ N_{12}^n\\[3pt]
x_2\wedge \bigl(N_{2}^n - (x_1- N_{1}^n )^+\bigr)
&\text{otherwise.}
\end{cases} %\label{ge-Zn22}
\end{align*}

\end{definition}

Note that the scheduling policy $\Check{z}^n$ is state-dependent, and can be
interpreted as follows.
Class-$1$ jobs prioritize  server pool~$1$ over  $2$.
Server pool~$2$ prioritizes the two classes of jobs depending on the system state.
Whenever $x_1\ge N_1^n+N_{12}^n$, server pool~$2$ allocates no more than
$N_{22}^n$ servers to class-$2$ jobs, while whenever $x_2\ge N_{22}^n$,
it allocates no more than  $N_{12}^n$ servers to class-$1$ jobs.
It is easy to check that this policy  $\Check{z}^n$ is work conserving.
The resulting
queue length and idleness $\Check{q}^n$ and $\Check{y}^n$
can be obtained by the balance equations: for $x\in \ZZ^2_+$,
\begin{align*} 
\Check{q}^n_1(x) &= x_1 - \Check{z}^n_{11}(x) - \Check{z}^n_{12}(x)\,,
& \Check{q}^n_2(x) &= x_2 - \Check{z}^n_{22}(x)\,, 
\\[5pt]
\Check{y}^n_1(x) &= N^n_1 -  \Check{z}^n_{11}(x)\,, &
\Check{y}^n_2(x) &= N^n_2 -  \Check{z}^n_{12}(x) -   \Check{z}^n_{22}(x) \,. 
\end{align*}

\begin{definition}\label{D3.2}
For each $x\in\RR^{2}_{+}$, define 
\begin{equation}\label{Td-xn}
\begin{split}
\Tilde{x}^{n}(x)&\;\df\; \bigl( x_1 - n x^*_1\,, x_2 - n x^*_2 \bigr)\,,\\[5pt]
\Hat{x}^{n}(x)&\;\df\; \frac{\Tilde{x}^{n}(x)}{\sqrt n}\,.
\end{split}
\end{equation}
where $x^*$ is given in \eqref{Ex*}. 
Also define
\begin{equation*}
\sS^n \;\df\; \bigl\{\Hat{x}^{n}(x) \,\colon\,x\in\ZZ^2_{+}\bigr\}\,,
\qquad
\Breve\sS^n \;\df\; \bigl\{\Hat{x}^{n}(x) \,\colon\,x\in\Breve\sX^n\bigr\}\,.
\end{equation*}

For $k\ge2$ and $\beta>0$, we let
\begin{equation}\label{Lyapk}
\Lyap_{k,\beta}(x)\;\df\; \abs{x_1}^k + \beta\abs{x_2}^k\,,\quad x\in\RR^2\,.
\end{equation}
\end{definition}

The generator of the  state process $X^n$ under a scheduling
policy $z^n$ takes the form
\begin{multline} \label{E-gen} 
\cL_n^{z^n} f(x) \;\df\; \sum_{i=1}^2 \lambda^{n}_i \bigl(f(x+e_i) - f(x)\bigr)
+ ( \mu_{11}^{n} z^{n}_{11} + \mu_{12}^{n} z^{n}_{12}) \bigl(f(x-e_1)
- f(x)\bigr)  \\
+  \mu_{22}^{n} z^{n}_{22} \bigl(f(x-e_2) - f(x)\bigr)
+ \sum_{i=1}^2 \gamma_i^{n} q^{n}_i
\bigl(f(x-e_i)- f(x)\bigr)\,, \qquad   x \in \ZZ^{2}_{+}\,,
\end{multline}
for $f\in\Cc_b(\RR^2)$.
We can write the generator $\widehat\cL_n^{z^n}$ of the diffusion-scaled
state process $\Hat{X}^n$ using \eqref{E-gen} and the function
$\Hat{x}^n$ in Definition~\ref{D3.2} as 
\begin{equation} \label{E-gen2} 
\widehat\cL_n^{z^n} f(\Hat{x})
\;=\; \cL_n^{z^n} f\bigl(\Hat{x}^n(x)\bigr)\,.
\end{equation}

We have the following.

\begin{proposition}\label{P3.1}
Let $\Hat{X}^{n}$ denote the diffusion-scaled state process under the
scheduling policy $\Check{z}^{n}$
in Definition~\ref{D3.1}, and
$\widehat\cL_n^{\Check{z}^{n}}$ be its generator.  
For any $k\ge 2$, there exists $\beta_0>0$, such that
\begin{equation} \label{Lya-GE}
\widehat\cL_n^{\Check{z}^{n}} \Lyap_{k,\beta}(\Hat{x}) \;\le\; C_1  - C_2\, \Lyap_{k,\beta}(\Hat{x})
\qquad\forall\,\Hat{x}\in \sS^n \,,\quad\forall\,n\ge n_0\,,
\end{equation} 
for some positive constants $C_1$, $C_2$, and $n_0\in\NN$, which depend
on $\beta\ge\beta_0$ and $k$.  Namely, $\Hat{X}^n$ under the
scheduling policy $\Check{z}^{n}$  is geometrically ergodic.
As a consequence, for any $k>0$,  there exists $n_0 \in \NN$ such that
\begin{equation} \label{upper-p1}
\sup_{n\ge n_0} \limsup_{T \to \infty}\;\frac{1}{T}\;
\Exp^{\Check{z}^{n}} \biggl[ \int_{0}^{T} \babs{\Hat{X}^{n}(s)}^k \,\D{s} \biggr]
\;<\; \infty\,,
\end{equation}
and the same holds if we replace $\Hat{X}^n$ with
$\Hat{Q}^n$ or $\Hat{Y}^n$ in \eqref{upper-p1}.
In other words, the diffusion-scaled cost criterion $J(\Hat{X}^n(0), Z^n)$ is
finite for $n\ge n_0$.
\end{proposition}

\begin{proof}
See Appendix~\ref{App1}.
\end{proof}

\begin{remark}
We remark that given \eqref{upper-p1} for $\Hat{X}^n$, the same property may
not hold for $\Hat{Q}^n$ or $\Hat{Y}^n$. 
It always holds if a scheduling policy satisfies the JWC condition
(by the balance equation \eqref{E-sumbal}).
Otherwise, that property needs to be verified under the given scheduling policy.
It can easily checked that if the property holds for any two processes of
$\Hat{X}^n$, $\Hat{Q}^n$ and $\Hat{Y}^n$, then it also holds for the third. 
\end{remark}

%%%%%%%%%%%%%%%%%%%%%%%%%%%%%%%%%%%%%%%%%%%%%%%%%%%%%%%%%%%%%%%%%%%%%
\section{Ergodic Control of the Limiting Diffusion}\label{S4}

\subsection{The controlled diffusion limit} %\label{S4.1}

If the action space is $\bcZn$, or equivalently
$Z^n\in\Breve\fZ^n$, the convergence in distribution of the
diffusion-scaled processes $\Hat{X}^n$ to the
limiting diffusion $X$ in \eqref{sde} is shown
in Proposition~3 in \citet{Atar-05b}.
For the class of multiclass multi-pool networks, the drift of the limiting diffusion
is given implicitly via a linear map in Proposition~3 of \citet{Atar-05b}.
For the N-network, the drift can be explicitly expressed as we show below
in \eqref{diff-drift}. In \citet{AP15}, a leaf elimination algorithm has been
developed to provide an
explicit expression for the drift of the limiting diffusion of general
multiclass multi-pool networks.
In the case of the N-network, the limit process $X$ is an $2$-dimensional diffusion satisfying
the It\^o equation 
\begin{equation} \label{sde}
d X_t\;=\;b(X_t, U_t)\,\D{t} + \Sigma \, \D W_t\,, 
\end{equation}
with initial condition $X_0 = x$ and the control $U_t \in \Act$, where
\begin{equation}\label{E-Act}
\Act \;\df\; \bigl\{ u= (u^c, u^s) \in \RR^{2}_+ \times \RR^2_+
\,\colon\, e\cdot u^c = e\cdot u^s=1\bigr\}\,.
\end{equation} 
In \eqref{sde},
the process $W$ is a $2$-dimensional standard Wiener process independent of the
initial condition $X_0=x$. 

Following the leaf elimination algorithm
for the N-network, the drift of the diffusion can be computed as follows.
Let
\begin{equation}\label{diff-map2}
\widehat{G}[u](x) \;\df\; \begin{pmatrix}
- (e\cdot x)^{-} u_1^s &  x_1- (e\cdot x)^{+} u^c_1+(e\cdot x)^{-} u_1^s\\[5pt]
0 & x_2- (e\cdot x)^{+} u^c_2
\end{pmatrix}\,,\qquad u\in\Act\,. 
\end{equation}
Then the drift $b: \RR^2 \times \Act \to \RR^2$ takes the form
\begin{equation*}
b(x,u)\;=\; \begin{pmatrix}
- \mu_{11} \widehat{G}_{11}[u](x)-\mu_{12} \widehat{G}_{12}[u](x)
- \gamma_i (e\cdot x)^{+} u^c_i + \ell_1\\[5pt]
\mu_{22} \widehat{G}_{22}[u](x)
- \gamma_2 (e\cdot x)^{+} u^c_2 + \ell_2\,,
\end{pmatrix}
\end{equation*}
which can also be written as (see Lemma~4.3 and Section~4.2 in \cite{AP15})
\begin{equation} \label{diff-drift}
b(x, u) \;=\; - B_1 (x - (e\cdot x)^+ u^c) + (e\cdot x)^{-} B_2u^s
- (e\cdot x)^+\Gamma u^c + \ell\,, 
\end{equation}
with 
\begin{equation*} 
B_1 \;\df\; \diag\{\mu_{12}, \mu_{22}\}\,, \quad B_2 \;\df\;
\diag\{\mu_{11} - \mu_{12}, 0\}\,, \quad \Gamma
\;\df\; \diag\{\gamma_1, \gamma_2\}\,.
\end{equation*}
Here,
$\ell\;\df\;(\ell_1,\ell_2)\transp$ is defined by
\begin{equation} \label{ell-limit}
\ell_1 \;\df\; \Hat\lambda_1 -\Hat{\mu}_{11} z^*_{11} - \Hat{\mu}_{12} z^*_{12}\,,
\qandq \ell_2 \;\df\; \Hat\lambda_2 - \Hat{\mu}_{22} z^*_{22}\,. 
\end{equation}
The covariance matrix is given by 
$\Sigma\;\df\;\diag\bigl(\sqrt{2 \lambda_1}, \ \sqrt{2 \lambda_2}\bigr)$. 
The control process $U$ lives in the compact set $\Act$ in \eqref{E-Act}, 
and $U_{t}(\omega)$ is jointly measurable in
$(t,\omega)\in[0,\infty)\times\Omega$.
Moreover, it is \emph{non-anticipative}:
for $s < t$, $W_{t} - W_{s}$ is independent of
\begin{equation*}
\sF_{s} \;\df\;\text{the completion of~} \sigma\{X_{0},U_{r},W_{r},\;r\le s\}
\text{~relative to~}(\sF,\Prob)\,.
\end{equation*}
Let $\Uadm$ be the set of all such controls, referred to as \emph{admissible}
controls. We refer the reader to Section \ref{sec-controlparam} on the
control parameterization.
A mere comparison of \eqref{diff-map2} with \eqref{D6.1D}
makes it clear how the control process $U$ relates to the control
process $U^{n}$ for the $n^{\rm th}$ system in Definition~\ref{D6.1}.

We remark that \eqref{sde} can be regarded as a
piecewise-linear controlled diffusion. 
Note that the matrix $B_1$ is an $M$-matrix. However, there is an additional term
$(e\cdot x)^{-} B_2u^s$ in the drift, which  differs from the class of
piecewise-linear controlled diffusions discussed in Section~3.3 of \citet{ABP14}.
We refer to \eqref{sde} as the \emph{limiting diffusion},
or the \emph{diffusion limit}.

The associated limit processes $Q$, $Y$, and $Z$ satisfy the
following balance equations:  
\begin{equation*}
\begin{split}
&\begin{aligned}
 X_1(t) &\;=\; Q_1(t) +  Z_{11}(t) + Z_{12}(t) \,,\\[5pt]
 X_2 (t) &\;=\; Q_2(t) + Z_{22}(t) \,,
\end{aligned}
\quad
\begin{aligned}
Y_1(t) + Z_{11}(t) &\;=\; 0 \,, \\[5pt]
Y_2(t) + Z_{12}(t) + Z_{22}(t) &\;=\; 0 \,,
\end{aligned}
\end{split}
\end{equation*}
with $Q_i (t) \ge 0$, $Y_j(t) \ge 0$, $i,j =1,2$.
Note that these `balance' conditions
imply that JWC always holds at the diffusion limit, i.e.,
\begin{equation*}
e\cdot Q(t) \;=\; \bigl(e\cdot X(t)\bigr)^{+}\,,
\qquad e \cdot Y(t) \;=\; \bigl(e\cdot X(t)\bigr)^{-} \qquad \forall\,t \ge 0 \,. 
\end{equation*}

\subsection{Control problems for the diffusion limit}\label{S4.2}

We state the three problems which correspond to (P1)--(P3)
in Section~\ref{S3.1}
for the controlled diffusion in \eqref{sde}.
Let $r: \RR^2 \times \Act \to \RR$ be defined by
\begin{equation*}
r(x, u)\;=\;r\bigl(x,(u^c, u^s)\bigr)\;\df\; \Hat{r}\bigl((e\cdot x)^{+}u^c,
(e\cdot x)^{-}u^s\bigr)\,,
\end{equation*}
with the same $\Hat{r}$ in \eqref{Erhat}, that is, 
\begin{equation}\label{E-cost}
r(x, u) \;=\; [(e\cdot x)^{+}]^m \sum_{i=1}^{2} \xi_i (u^c_i)^m
+  [(e\cdot x)^{-}]^m \sum_{j=1}^{2} \zeta_j (u^s_j)^m, \quad m\ge 1\,,
\end{equation}
for the given $\xi=(\xi_1, \xi_2)\transp$ 
and $\zeta=(\zeta_1,\zeta_2)\transp$ in \eqref{Erhat}. 
Let the ergodic cost associated with the controlled diffusion $X$ and the
running cost $r$ be defined as
\begin{equation*}
J_{x,U}[r] \;\df\; \limsup_{T \to \infty}\;\frac{1}{T}\;\Exp_x^U
\left[ \int_{0}^{T} r(X_t, U_t)\,\D{t} \right]\,, \quad U \in \Uadm\,. 
\end{equation*}
\begin{itemize}
\item[\textbf{(P1$\bm'$)}]
(\emph{unconstrained problem})
The running cost function $r(x,u)$ is as in
\eqref{E-cost}. The ergodic control problem is then defined as
\begin{equation}\label{diff-opt}
\varrho^*(x) \;=\; \inf_{U \in \Uadm} \;J_{x,U}[r] \,.
\end{equation}

\item[\textbf{(P2$\bm'$)}]
(\emph{constrained problem})
The running cost function $r_{\mathsf{o}}(x,u)$ is as in
\eqref{E-cost} with $\zeta\equiv 0$.
 Also define
\begin{equation}\label{E-rj}
r_j(x,u)\;\df\; [(e\cdot x)^{-}u^s_j]^{\Tilde{m}}\,,\quad j=1,2\,,
\end{equation}
with $\Tilde{m}\ge1$,
and let $\updelta=(\updelta_1, \updelta_2)$ be a positive vector.
The ergodic control problem under idleness constraints is defined as 
\begin{equation} \label{diff-opt-c}
\begin{split}
\varrho_{\mathsf{c}}^{*}(x) &\;=\; \inf_{U \in \Uadm} \;J_{x,U}[r_{\mathsf{o}}]\,,
\\[5pt]
\text{subject to}\quad
J_{x,U}[r_{j}] &\;\le\; \updelta_j\,,\quad j =1,2\,.
\end{split}
\end{equation}

\item[\textbf{(P3$\bm'$)}]
(\emph{fairness})
The running costs $r_{\mathsf{o}}$, $r_1$ and $r_2$ are as in (P2$'$).
Let $\uptheta$ be a positive constant, and $1\le\Tilde{m}< m$.
The ergodic control problem under idleness fairness is defined as
\begin{equation} \label{diff-opt-f}
\begin{split} 
\varrho_{\mathsf{f}}^{*}(x) &\;=\; \inf_{U \in \Uadm} \;J_{x,U}[r_{\mathsf{o}}]\,,
\\[5pt]
\text{subject to}\quad
J_{x,U}[r_{1}] &\;=\; \uptheta\, J_{x,U}[r_{2}]\,. 
\end{split}
\end{equation}
\end{itemize}

The last problem enforces fairness of idleness allocation among the
two server pools.
Also note that penalizing only the queueing cost in (P1),
raises a well-posedness question, which was resolved
in Corollaries~4.1--4.2 of \citet{AP15}.

The quantities $\varrho^*(x)$, $\varrho^*_{\mathsf{c}}(x)$ and 
$\varrho^*_{\mathsf{f}}(x)$ are called the optimal values of
the ergodic control problems (P1$'$), (P2$'$) and (P3$'$), respectively,
for the controlled diffusion process $X$ with initial state $x$.  
Note that as is shown in Section~3 of \citet{ABP14} and
Sections~3 and~5.4 of \citet{AP15},
the optimal values $\varrho^*(x)$, $\varrho^*_{\mathsf{c}}(x)$ and
$\varrho^*_{\mathsf{f}}(x)$ do not depend on $x\in\RR^{2}$,
and thus we remove their dependence on $x$ in the statements below.

Recall that a control is called \emph{Markov} if
$U_{t} = v(t,X_{t})$ for a measurable map $v\colon\RR_{+}\times\RR^{2}\to \Act$,
and it is called \emph{stationary Markov} if $v$ does not depend on
$t$, i.e., $v\colon\RR^{2}\to \Act$.
Let $\Usm$ denote the set of stationary Markov controls.
Recall also that a control $v\in\Usm$ is called \emph{stable}
if the controlled process is positive recurrent.
We denote the set of such controls by $\Ussm$,
and let $\mu_{v}$ denote the unique invariant probability
measure on $\RR^{2}$ for the diffusion under the control $v\in\Ussm$.
We also let $\cM\df\{\mu_{v}\,\colon\,v\in\Ussm\}$, and
$\eom$ denote the set of ergodic occupation measures corresponding to controls
in $\Ussm$, that is, 
\begin{equation*}\eom\;\df\;\biggl\{\uppi\in\cP(\RR^{2}\times\Act)\,\colon\,
\int_{\RR^{2}\times\Act}\Lg^{u} f(x)\,\uppi(\D{x},\D{u})=0\quad
\forall\,f\in\Cc^{\infty}_c(\RR^{2}) \biggr\}\,,\end{equation*}
where $\Lg^{u}f(x)$ is the controlled extended generator of the diffusion $X$, 
\begin{equation*}
\Lg^{u} f(x) \;\df\;\frac{1}{2} \sum_{i,j=1}^{2} a_{ij}\,\partial_{ij} f(x)
+ \sum_{i=1}^{2} b_{i}(x,u)\, \partial_{i} f(x)\,,\quad u\in\Act\,,
\end{equation*}
with $a \df\Sigma\Sigma\transp$ and
$\partial_{i}\df\tfrac{\partial~}{\partial{x}_{i}}$ and
$\partial_{ij}\df\tfrac{\partial^{2}~}{\partial{x}_{i}\partial{x}_{j}}$.  
The restriction of the ergodic control problem with running cost $r$ to
stable stationary Markov controls is equivalent
to minimizing
\begin{equation*}
\uppi(r)\;=\;\int_{\RR^{2}\times\Act} r(x,u)\,\uppi(\D{x},\D{u})
\end{equation*}
over all $\uppi\in\eom$.
If the infimum is attained in $\eom$, then we say that the ergodic control
problem is \emph{well posed}, and we refer to
any $\Bar\uppi\in\eom$ that attains this infimum
as an \emph{optimal ergodic occupation measure}.

We define the class of admissible controls
$\sU\df\{U=(U^c,U^s)\,\colon\,U^c=v^c(x) = (1,0)~\forall\,x\in\RR^2\}$,
and we also let
\begin{equation}\label{E-barbeta}
\Bar{\beta}_k \;\df\;\frac{(\gamma_1\vee\mu_{11}\vee\mu_{12})^{k+1}}
{\mu_{22}\,(\gamma_1\wedge\mu_{11}\wedge\mu_{12})^k}\,.
\end{equation}
We have the following lemma.

\begin{lemma}\label{L4.1}
Let $\Lyap_{k,\beta}$ be as in \eqref{Lyapk}.
There exist positive constants $C_1$ and $C_2$ depending only on $k$
and $\beta\ge\Bar{\beta}_k$,
such that
\begin{equation*}
\Lg^{U} \Lyap_{k,\beta} (x) \;\le\; C_1 - C_2\, \Lyap_{k,\beta} (x)
\qquad\forall\,U\in\sU\,,\quad
\forall\,x\in\RR^2\,.
\end{equation*}
\end{lemma}

\begin{proof}
By \eqref{diff-drift} we have
\begin{align*}
b_1\bigl(x,U\bigr) &\;=\;
\begin{cases} -\gamma_1 x_1 + (\mu_{12} - \gamma_1) x_2 +\ell_1&
\text{if~} (e\cdot x)^+ \ge 0\\[5pt]
-\bigl(\mu_{11} U_1^s + \mu_{12} U_2^s\bigr) x_1
- (\mu_{11}-\mu_{12})U_1^s\, x_2 + \ell_1&
\text{otherwise,}
\end{cases}\\[5pt]
b_2\bigl(x,U\bigr) &\;=\; - \mu_{22} x_2 + \ell_2\qquad
\forall\,x\in\RR^2\,.
\end{align*}
Therefore,
\begin{multline}\label{L4.1A}
\Lg^{U} \Lyap_{k,\beta} (x) \;\le\;
- k(\gamma_1\wedge\mu_{11}\wedge\mu_{12}) \abs{x_1}^k
+ k(\gamma_1\vee\mu_{11}\vee\mu_{12}) \abs{x_2}\abs{x_1}^{k-1}\\[3pt]
- \beta k\mu_{22} \abs{x_2}^k + k\ell_1\abs{x_1}^{k-1}
+\beta k\ell_2\abs{x_2}^{k-1}
+k(k-1)\bigl(\lambda_1\abs{x_1}^{k-2}+\lambda_2\beta\abs{x_2}^{k-2}\bigr)\,.
\end{multline}

Let
\begin{equation*}
\alpha\;\df\;\frac{\gamma_1\wedge\mu_{11}\wedge\mu_{12}}
{\gamma_1\vee\mu_{11}\vee\mu_{12}}\,.
\end{equation*}
Using \rd{Young's} inequality we write
\begin{equation*}
\abs{x_2}\abs{x_1}^{k-1}\;\le\;
(k-1)\frac{\alpha^{\frac{k}{k-1}}}{k}\abs{x_1}^k +
\frac{\alpha^{-k}}{k}\abs{x_2}^k
\;\le\;
(k-1)\frac{\alpha}{k}\abs{x_1}^k +
\frac{\alpha^{-k}}{k}\abs{x_2}^k\,. 
\end{equation*}
Thus, by \eqref{L4.1A}, we have
\begin{multline*}
\Lg^{U} \Lyap_{k,\beta} (x) \;\le\;
- (\gamma_1\wedge\mu_{11}\wedge\mu_{12}) \abs{x_1}^k
- (k\beta\mu_{22}-\Bar\beta_k)\abs{x_2}^k
+ k\ell_1\abs{x_1}^{k-1} +\beta k\ell_2\abs{x_2}^{k-1}\\[3pt]
+k(k-1)\bigl(\lambda_1\abs{x_1}^{k-2}+\lambda_2\beta\abs{x_2}^{k-2}\bigr)\,,
\end{multline*}
from which the result easily follows.
\end{proof}

As shown in Corollary~4.2 of \citet{AP15}, for any $k\ge1$,
there exists a constant $C=C(k)>0$
such that any solution $X_t$ of \eqref{sde} with $X_0=x_0\in\RR^2$ satisfies
\begin{equation}\label{E-apriori}
\Exp^U_x \biggl[\int_{0}^{T} \abs{X_t}^k\,\D{t}\biggr]
\;\le\; C\abs{x_{0}}^k+ CT +
C \Exp^U_x \biggl[\int_{0}^{T} \bigl((e\cdot X_t)^+\bigr)^k\,\D{t}\biggr]
\qquad \forall\, U\in\Uadm\,,\quad\forall\,T>0\,.
\end{equation}
This property plays a crucial role in solving (P1$'$)--(P3$'$).

\subsection{Optimal solutions to problems (P1\ensuremath{'})--(P3\ensuremath{'})}

The characterization of the optimal solutions to the ergodic
control problems (P1$'$)--(P3$'$)
has been thoroughly studied in \citet{ABP14} and \citet{AP15}.
We review some results that are used in the sections which follow to construct
asymptotically optimal scheduling policies and prove asymptotic optimality.
We first introduce some notation. 
Let 
\begin{equation} \label{E-H}
H_r(x,p)\;\df\;\min_{u\in\Act}\;\bigl[b(x,u)\cdot p + r(x,u)\bigr]\,
\quad \text{for\ \ }x, p \in \RR^2\,. 
\end{equation}
For
$\updelta=(\updelta_1,\updelta_2)\in\RR_{+}^2$, 
let
\begin{equation*}
\sH(\updelta)\;\df\; \bigl\{\uppi\in\eom\;\colon\,
\uppi(r_{j})\le \updelta_{j}\,,\; j=1, 2\bigr\}\,,\qquad
\sH^{\mathrm{o}}(\updelta)
\;\df\; \bigl\{\uppi\in\eom\;\colon\,
\uppi(r_{j})< \updelta_{j}\,,\; j=1,2 \bigr\}\,.
\end{equation*}

For $\updelta\in\RR_{+}^2$ and
$\uplambda=(\uplambda_{1}\,, \uplambda_2)\transp
\in\RR^{2}_{+}$
define the running cost $g_{\updelta,\uplambda}$ by
\begin{equation*}
g_{\updelta,\uplambda}(x,u) \;\df\; r_{\mathsf{o}}(x,u) + \sum_{j=1}^{2}
\uplambda_{j}\bigl(r_{j}(x,u)-\updelta_{j}\bigr)\,.
\end{equation*}

We say that the vector $\updelta\in (0,\infty)^{2}$
is \emph{feasible} (or that the constraints in \eqref{diff-opt-c} are feasible)
if  there exists $\uppi'\in\sH^{\mathrm{o}}(\updelta)$ such that
$\uppi'(r_{\mathrm{o}})<\infty$.
The following is contained in Theorem~5.2 of \cite{AP15}.

\begin{theorem}
For the ergodic control problem in \eqref{diff-opt}, there exists a unique
solution $V \in  \Cc^2(\RR^2)$, satisfying $V(0)=0$, to the associated HJB equation: 
\begin{equation*}
\min_{u\in\Act}\;\bigl[\Lg^{u}V(x)+r(x,u)\bigr]
\;=\;\varrho^{*}\,.
\end{equation*}
Moreover, a stationary Markov control $v \in \Ussm$ is optimal
if and only if it satisfies
\begin{equation*}
H_r\bigl(x,\nabla V(x)\bigr)
\;=\;b\bigl(x,v(x)\bigr)\cdot\nabla V(x)
+ r\bigl(x,v(x)\bigr)\quad\text{a.e.~in~} \RR^{2}\,.
\end{equation*}
\end{theorem}

The following is contained in Lemmas~3.3--3.5, and Theorems~3.1--3.2 of
\cite{AP15}.

\begin{theorem}\label{T4.2}
Suppose that $\updelta$ is feasible
for the ergodic control problem under constraints in \eqref{diff-opt-c},
i.e., there exists $\uppi'\in\sH^{\mathrm{o}}(\updelta)$ such that
$\uppi'(r_{\mathsf{o}})<\infty$.
Then the following hold.
\begin{itemize}
\item[\upshape{(}a\upshape{)}]
There exists 
$\uplambda^{*}\in\RR^{2}_{+}$ such that
\begin{equation*}
\inf_{\uppi\,\in\,\sH(\updelta)}\;\uppi(r_{\mathsf{o}}) \;=\;
\inf_{\uppi\,\in\,\eom}\;\uppi(g_{\updelta,\uplambda^{*}})
\;=\;\varrho^*_{\mathsf{c}}\,.
\end{equation*}
\item[\upshape{(}b\upshape{)}]
If $\uppi^{*}\in\sH(\updelta)$
attains the infimum of $\uppi\mapsto\uppi(r_{\mathsf{o}})$
in $\sH(\updelta)$, then
$\uppi^{*}(r_{\mathsf{o}})\;=\;\uppi^{*}(g_{\updelta,\uplambda^{*}})$,
and
\begin{equation*}
\uppi^{*}(g_{\updelta,\uplambda})\;\le\;\uppi^{*}(g_{\updelta,\uplambda^{*}})
\;\le\;
\uppi(g_{\updelta,\uplambda^{*}})
\qquad\forall\,(\uppi,\uplambda)\in\eom\times\RR^{2}_{+}\,.
\end{equation*}
\item[\upshape{(}c\upshape{)}]
There exists  $V_{\mathsf{c}}\in\Cc^{2}(\RR^2)$ satisfying 
\begin{equation*}
\min_{u\in\Act}\;\bigl[\Lg^{u}V_{\mathsf{c}}(x)
+ g_{\updelta,\uplambda^{*}}(x,u)\bigr]
\;=\; \uppi^{*}(g_{\updelta,\uplambda^{*}}) = \varrho^*_{\mathsf{c}}\,,\quad
x\in\RR^{2}\,.
\end{equation*}
\item[\upshape{(}d\upshape{)}]
A stationary Markov control $v_{\mathsf{c}}\in\Ussm$
is optimal if and only if it satisfies 
\begin{equation*}
H_{g_{\updelta,\uplambda^{*}}}\bigl(x,\nabla V_{\mathsf{c}}(x)\bigr) \;=\;
b\bigl(x, v_{\mathsf{c}}(x)\bigr)\cdot \nabla V_{\mathsf{c}}(x)
+ g_{\updelta,\uplambda^{*}}\bigl(x,v_{\mathsf{c}}(x)\bigr)
\quad\text{a.e.~in~} \RR^{2}\,,
\end{equation*}
where $H_{g_{\updelta,\uplambda^{*}}}$ is defined in \eqref{E-H} with $r$
replaced by $g_{\updelta,\uplambda^{*}}$.
\item[\upshape{(}e\upshape{)}]
The map
$\updelta\mapsto\inf_{\uppi\,\in\,\sH(\updelta)}\;\uppi(r_{\mathsf{o}})$
is continuous at any feasible point $\Hat\updelta$.
\end{itemize}
\end{theorem}

For uniqueness of the solutions $V_{\mathsf{c}}$ see Theorem~3.2
in \citet{AP15}.

We now turn to the constrained ergodic control problem in \eqref{diff-opt-f}.
Lemma~\ref{L4.1} implies that Assumption~5.1 in \cite{AP15} holds,
and consequently the solution of (P3$'$) follows by
Theorem~5.8 in the same paper.
However, the Lagrangian in (P3$'$) is not bounded below in $\RR^2$, and since
no details were provided in \cite{AP15} on the existence of solutions
to the HJB equation, we provide a proof in Appendix~\ref{App2}.

\begin{theorem}\label{T4.3}
For any $\uptheta>0$ the constraint in \eqref{diff-opt-f} is feasible.
All the conclusions of Theorem~\ref{T4.2} hold, provided that we replace $\sH(\updelta)$ and 
$g_{\updelta,\uplambda}$ with
\begin{equation} \label{E-sH-f-theta}
\sH_{\mathsf{f}}(\uptheta)
\;\df\; \bigl\{\uppi\in\eom\;\colon\, \uppi(r_{1})=\uptheta\,\uppi(r_{2})\bigr\}\,,
\end{equation}
and
\begin{equation*}
h_{\uptheta,\uplambda} (x,u)\;\df\; r_{\mathsf{o}}(x,u) + 
\uplambda\bigl(r_{1}(x,u)-\uptheta\, r_{2}(x,u)\bigr)\,,\quad
\uplambda\in\RR\,,
\end{equation*}
respectively. 
\end{theorem}

%%%%%%%%%%%%%%%%%%%%%%%%%%%%%%%%%%%%%%%%%%%%%%%%%%%%%%%%%%%%%%%%%%%%%
\section{Asymptotic Optimality} \label{S5}
In this section, we present the main results on asymptotic optimality. 
We show that the values of the three ergodic control
problems in the diffusion scale converge to the values of the corresponding
ergodic control problems for the limiting diffusion, respectively. 
The proofs of the lower and upper bounds are given in
Sections~\ref{S8} and \ref{S9}, respectively. 

Recall the definitions of $J$, $J_{\mathsf{o}}$, $\Hat{V}^{n}$,
$\Hat{V}^{n}_{\mathsf{c}}$, and
$\Hat{V}^{n}_{\mathsf{f}}$ in (P1)--(P3),
and the definitions of $\varrho^*$, $\varrho^{*}_{\mathsf{c}}$,
and $\varrho^{*}_{\mathsf{f}}$ in (P1$'$)--(P3$'$).

\begin{theorem} \label{T-lbound} $($lower bounds$)$
Let $\Hat{X}^{n}(0) \Rightarrow x \in \RR^2$ as $n\to \infty$.
The following hold:
\begin{itemize}
\item[\upshape{(}i\upshape{)}]
For any sequence $\{Z^{n},\;n\in\NN\}\subset\boldsymbol\fZ$
the diffusion-scaled cost in \eqref{cost-ds} satisfies
\begin{equation*}
\liminf_{n\to\infty}\;
J\bigl(\Hat{X}^{n}(0),\Hat{Z}^{n}\bigr) \;\ge\; \varrho^*\,.
\end{equation*}

\item[\upshape{(}ii\upshape{)}]
Suppose that under a sequence
$\{Z^{n},\;n\in\NN\}\subset\boldsymbol\fZ$ the constraint
in \eqref{constraint-ds} is satisfied
for all sufficiently large $n\in\NN$.
Then
\begin{equation*} 
\liminf_{n\to\infty}\; J_{\mathsf{o}}\bigl(\Hat{X}^{n}(0),\Hat{Z}^{n}\bigr)
\;\ge\;\varrho^{*}_{\mathsf{c}}\,,
\end{equation*}
and as a result we have that
$~\displaystyle
\liminf_{n\to\infty}\;\Hat{V}^{n}_{\mathsf{c}}(\Hat{X}^{n}(0))
\,\ge\, \varrho_{\mathsf{c}}^{*}\,.$

\smallskip
\item[\upshape{(}iii\upshape{)}]
There exists a positive constant $\Hat{C}$, such that if a sequence
$\{Z^{n},\;n\in\NN\}\subset\boldsymbol\fZ$ satisfies
\begin{equation}\label{ET5.1A}
\babss{ \frac{
J_{\mathsf{c},1}\bigl(\Hat{X}^{n}(0), Z^{n}\bigr)}
{ J_{\mathsf{c},2}\bigl(\Hat{X}^{n}(0), Z^{n}\bigr)}  - \uptheta } \;\le\;
\epsilon\,
\end{equation}
for some $\epsilon\in(0,\theta)$, and all sufficiently large $n\in\NN$, then
\begin{equation}\label{ET5.1B}
\liminf_{n\to\infty}\;J_{\mathsf{o}}(\Hat{X}^n(0),Z^n) \;\ge\;
\varrho^{*}_{\mathsf{f}} - \Hat{C}\epsilon\,.
\end{equation}
\end{itemize}
\end{theorem}

The proof of the theorem that follows relies on the fact that
$r$ and also $r_{j}$ for $i=0,1,2$, are convex functions of $u$.

\begin{theorem} \label{T-ubound} $($upper bounds$)$
Let $\Hat{X}^{n}(0) \Rightarrow x \in \RR^2$ as $n\to \infty$.
The following hold:
\begin{itemize}
\item[\upshape{(}i\upshape{)}]
$\displaystyle\limsup_{n\to\infty}\;\Hat{V}^{n}(\Hat{X}^{n}(0))
\;\le\; \varrho^*$\,.

\smallskip
\item[\upshape{(}ii\upshape{)}]
For any $\epsilon>0$, there exists a sequence
$\{Z^{n},\;n\in\NN\}\subset\boldsymbol\fZ$ such that the constraint
in \eqref{constraint-ds} is feasible for all sufficiently large $n$,
and
\begin{equation*}
\limsup_{n\to\infty}\; J_{\mathsf{o}}\bigl(\Hat{X}^{n}(0),\Hat{Z}^{n}\bigr)
\;\le\;\varrho^{*}_{\mathsf{c}}+\epsilon\,.
\end{equation*}
Consequently, we have that
$~\displaystyle
\limsup_{n\to\infty}\;\Hat{V}^{n}_{\mathsf{c}}(\Hat{X}^{n}(0))
\;\le\; \varrho_{\mathsf{c}}^{*}$\,.

\smallskip
\item[\upshape{(}iii\upshape{)}]
For any $\epsilon>0$, there exists a sequence
$\{Z^{n},\;n\in\NN\}\subset\boldsymbol\fZ$ such that
\eqref{ET5.1A} holds for all sufficiently large $n\in\NN$, and
\begin{equation*}
\limsup_{n\to\infty}\;J_{\mathsf{o}}(\Hat{X}^n(0),Z^n) \;\le\;
\varrho^{*}_{\mathsf{f}} + \epsilon\,.
\end{equation*}
\end{itemize}
\end{theorem}

\smallskip
%%%%%%%%%%%%%%%%%%%%%%%%%%%%%%%%%%%%%%%%%%%%%%%%%%%%%%%%%%%%%%%%%%%%%
\section{System dynamics and an equivalent control parameterization} \label{S6}

\subsection{Description of the system dynamics}

The processes $X^{n}$ can be represented via rate-$1$ Poisson processes:
for each $i=1,2$ and $t\ge 0$, we have
\begin{equation} \label{Xrep}
\begin{split}
X^{n}_1(t)&\;=\;X^{n}_1(0) + A^{n}_1(\lambda^{n}_1 t) - \sum_{j=1,2}
S^{n}_{1j} \left( \mu_{1j}^{n}\int_0^t Z_{1j}^{n}(s) \D{s} \right)
 - R_1^{n} \left(\gamma_1^{n} \int_0^t Q^{n}_1(s) \D{s} \right)\,, \\
X^{n}_2(t)&\;=\;X^{n}_2(0) + A^{n}_2(\lambda^{n}_2 t) - 
S^{n}_{22} \left( \mu_{22}^{n}\int_0^t Z_{22}^{n}(s) \D{s} \right)
- R_2^{n} \left(\gamma_2^{n} \int_0^t Q^{n}_2(s) \D{s} \right)\,.
\end{split}
\end{equation}
Recall that the processes $A^{n}_i$, $S^{n}_{ij}$ and $R^{n}_i$ are all rate-1 Poisson
processes and mutually independent, and independent of the initial
quantities $X^{n}_i(0)$.

By \eqref{DiffDef} and \eqref{Xrep}, we can write $\Hat{X}^{n}_1(t)$ and
$\Hat{X}^{n}_2(t)$  as
\begin{multline} \label{hatXn-1}
\Hat{X}^{n}_1(t) \;=\; \Hat{X}^{n}_1(0) + \ell_1^{n} t
-  \mu_{11}^{n} \int_0^t \Hat{Z}^{n}_{11}(s) \D{s}
-  \mu_{12}^{n} \int_0^t \Hat{Z}^{n}_{12}(s) \D{s}
- \gamma^{n}_1 \int_0^t \Hat{Q}^{n}_1(s) \D{s}  \\[5pt]
+ \Hat{M}^{n}_{A, 1} (t) - \Hat{M}^{n}_{S, 11}(t)  - \Hat{M}^{n}_{S, 12}(t)
- \Hat{M}^{n}_{R, 1}(t)\,, 
\end{multline}
\begin{equation} \label{hatXn-2}
\Hat{X}^{n}_2(t) \;=\; \Hat{X}^{n}_2(0) + \ell_2^{n} t
- \mu_{22}^{n}  \int_0^t \Hat{Z}^{n}_{22}(s) \D{s}
- \gamma^{n}_2 \int_0^t \Hat{Q}^{n}_2(s) \D{s}
+ \Hat{M}^{n}_{A, 2} (t) - \Hat{M}^{n}_{S, 22}(t) - \Hat{M}^{n}_{R, 2}(t)\,, 
\end{equation}
where for $i=1,2,$ and $j=1,2,$ 
\begin{align*}
\Hat{M}^{n}_{A, i}(t) &\;\df\;  \frac{1}{\sqrt{n}}(A^{n}_i(\lambda_i^{n} t)
- \lambda_i^{n} t), \\[5pt]
 \Hat{M}^{n}_{S, ij}(t) &\;\df\; \frac{1}{\sqrt{n}}\left( S^{n}_{ij}
 \left( \mu_{ij}^{n}\int_0^t Z_{ij}^{n}(s) \D{s} \right)
 - \mu_{ij}^{n}\int_0^t Z_{ij}^{n}(s) \D{s}\right) \,,\\[5pt]
\Hat{M}^{n}_{R, i}(t) &\;\df\;\frac{1}{\sqrt{n}}
\left(R_i^{n} \left(\gamma_i^{n} \int_0^t Q^{n}_i(s) \D{s} \right)
-\gamma_i^{n} \int_0^t Q^{n}_i(s) \D{s} \right)\,,
\end{align*}
and $\ell^{n} = (\ell^{n}_1, \ell^{n}_2)\transp$ is defined by
\begin{equation*}
\ell^{n}_1\;\df\; \frac{1}{\sqrt{n}} \left(\lambda^{n}_1
-   \mu_{11}^{n} z^*_{11} n -  \mu_{12}^{n} z^*_{12} n  \right) \,,
\quad \ell^{n}_2\;\df\; \frac{1}{\sqrt{n}} \left(\lambda^{n}_2
-  \mu_{22}^{n} z^*_{22} n \right) \,,
\end{equation*}
with $z_{ij}^*$ as in \eqref{Ez*}.
It is easy to see that under the assumptions on the parameters
in Assumption~\ref{as-para}, $\ell^{n} \to \ell$ as $n \to\infty$,
where $\ell$ is defined in \eqref{ell-limit}. 
The processes $\Hat{M}^{n}_{A, i}\df\{\Hat{M}^{n}_{A, i}(t): t\ge 0\}$,
$\Hat{M}^{n}_{S, ij}\df \{ \Hat{M}^{n}_{S, ij}(t): t\ge 0\}$, and
$\Hat{M}^{n}_{R, i}\df\{\Hat{M}^{n}_{R, i}(t): t\ge 0\}$ are 
square integrable martingales w.r.t. the filtration $\mathbf{F}^{n}$
with quadratic variations
\begin{equation*}
\langle \Hat{M}^{n}_{A, i} \rangle(t) \;\df\; \frac{\lambda_i^{n}}{n} t\,, \quad 
\langle \Hat{M}^{n}_{S, ij} \rangle (t) \;\df\; \frac{\mu_{ij}^{n}}{n}
\int_0^t Z_{ij}^{n}(s)\D{s}\,,\quad 
\langle \Hat{M}^{n}_{R, i} \rangle (t)\;\df\;\frac{\gamma_i^{n}}{n}
\int_0^t Q^{n}_i(s)\D{s} \,. 
\end{equation*}

By \eqref{Exi*}--\eqref{Ez*}, \eqref{baleq}, and \eqref{DiffDef},
we obtain the balance equations 
\begin{equation}\label{baleq-hat}
\begin{split}
\begin{aligned}
\Hat{X}^{n}_1(t) &\;=\; \Hat{Q}_1^{n}(t) +  \Hat{Z}^{n}_{11}(t)
+  \Hat{Z}^{n}_{12}(t)\,, \\[5pt]
\Hat{X}^{n}_2(t) &\;=\; \Hat{Q}_2^{n}(t) +  \Hat{Z}^{n}_{22}(t) \,,
\end{aligned}
\qquad
\begin{aligned}
\Hat{Y}^{n}_1(t) + \Hat{Z}^{n}_{11}(t) &\;=\; 0 \,, \\[5pt] 
\Hat{Y}^{n}_2(t) + \Hat{Z}^{n}_{12}(t) + \Hat{Z}^{n}_{22}(t) &\;=\; 0\,,
\end{aligned}
\end{split}
\end{equation}
for all $t\ge 0$.
The work conservation and JWC conditions
translate to the following:
\begin{equation*}
\Hat{Q}_1^{n}(t) \wedge \Hat{Y}^{n}_j(t)\;=\;0\,
\qquad \forall j = 1,2, \qandq \Hat{Q}_2^{n}(t) \wedge
\Hat{Y}^{n}_2(t)\;=\;0\,,  \quad\forall\, t \ge 0\,,
\end{equation*}
and 
$e\cdot \Hat{Q}^{n}(t) \wedge e \cdot \Hat{Y}^{n}(t)\;=\;0$, $t \ge 0$, 
respectively.

\subsection{Control parameterization} \label{sec-controlparam}
By \eqref{baleq-hat},
we obtain
\begin{equation} \label{E-sumbal}
e\cdot \Hat{X}^{n}(t) \;=\; e\cdot \Hat{Q}^{n}(t) - e \cdot \Hat{Y}^{n}(t)\,,
\end{equation}
and therefore the JWC condition is equivalent to
\begin{equation}\label{E-joint}
e\cdot \Hat{Q}^{n}(t) \;=\; \bigl(e\cdot \Hat{X}^{n}(t)\bigr)^{+}\,,
\qquad e \cdot \Hat{Y}^{n}(t) \;=\; \bigl(e\cdot \Hat{X}^{n}(t)\bigr)^{-}\,.
\end{equation}

\begin{definition}\label{D6.1}
We define the processes
$U^{c,n}\df (U^{c,n}_1, U^{c,n}_2)\transp$
and $U^{s,n}\df (U^{s,n}_1,U^{s,n}_2)\transp$, $t\ge0$, by
\begin{equation} \label{lb-Ucn}
U^{c,n}(t) \;\df\; \begin{cases}
\frac{\Hat{Q}^{n}(t)}{e\cdot \Hat{Q}^{n}(t)}
& \text{if~} e\cdot \Hat{Q}^{n}(t)>0\,,\\
e_1= (1,0) &\text{otherwise,}\end{cases}
\end{equation} 
and 
\begin{equation}\label{lb-Usn}
U^{s,n}(t) \;\df\; \begin{cases}
\frac{\Hat{Y}^{n}(t)}{e\cdot \Hat{Y}^{n}(t)}
& \text{if~} e\cdot \Hat{Y}^{n}(t)>0\,,\\
e_2 = (0,1) &\text{otherwise,}\end{cases}
\end{equation}
and let $U^{n}\df (U^{c,n}, U^{s,n})$.
\end{definition}

The process $U^{c,n}_i(t)$ represents the proportion of the total queue length
in the network at queue $i$ at time $t$, while $U^{s,n}_j(t)$ represents the
proportion of the total idle servers in the network at station $j$ at time $t$.
The control $U^{c,n}(t) = e_1 = (1,0)$ means that server pool~$2$
gives strict static priority to class-$2$ jobs, while the control
$U^{s,n}(t) = e_2=(0,1)$ means that class-$1$ jobs strictly prefer
service in pool~$1$.

Given $Z^{n} \in \mathfrak{Z}^{n}$, the process $U^{n}$ is uniquely determined
via \eqref{baleq-hat} and \eqref{lb-Ucn}--\eqref{lb-Usn}
and lives in the set $\Act$ in \eqref{E-Act}.
It follows by \eqref{baleq-hat} and \eqref{E-joint}
that, under JWC, we have that for each $t\ge 0$, 
\begin{equation}\label{lb-joint}
\Hat{Q}^{n}(t)\;=\; \bigl(e\cdot \Hat{X}^{n}(t)\bigr)^{+}\,U^{c,n}(t)\,,
\qquad
\Hat{Y}^{n}(t)\;=\;\bigl(e\cdot \Hat{X}^{n}(t)\bigr)^{-}\,U^{s,n}(t)\,.
\end{equation}
Also, by \eqref{lb-joint}, under the JWC condition, we have
\begin{equation}\label{D6.1D}
\Hat{Z}^{n} \;=\;
\begin{bmatrix}
- (e\cdot\Hat{X}^n)^{-} U^{s,n} _1 \quad &
\Hat{X}_1^n -  (e\cdot\Hat{X}^n)^{+} U^{c,n}_1 + (e\cdot\Hat{X}^n)^{-} U^{s,n}_1
\\[5pt]
0 \quad & \Hat{X}_2^n -  (e\cdot\Hat{X}^n)^{+} U^{c,n}_2 
\end{bmatrix} \,.
\end{equation}

%%%%%%%%%%%%%%%%%%%%%%%%%%%%%%%%%%%%%%%%%%%%%%%%%%%%%%%%%%%%%%%%%%%%%
\section{Convergence of mean empirical measures}\label{S7}

For the process $X^n$ under a scheduling policy
$Z^n$, and with $U^n$ as in Definition~\ref{D6.1},
we define the mean empirical measures 
\begin{equation} \label{E-emp}
\Phi^{Z^n}_{T}(A\times B)\;\df\;
\frac{1}{T}\,\Exp^{Z^n}\biggl[\int_{0}^{T}\Ind_{A\times B}
\bigl(\Hat{X}^{n}(t),U^{n}(t)\bigr)\,\D{t}\biggr]
\end{equation}
for Borel sets $A\subset\RR^{2}$ and $B\subset \Act$.
Recall Definition~\ref{D2.1}.
The lemma which follows provides a sufficient condition
under which the mean empirical measures $\Phi^{Z^n}_{T}$ are tight and converge
to an ergodic occupation measure corresponding to some stationary stable Markov
control for the limiting diffusion control problem.
The condition simply requires a finite long-run average first-order moment of
the diffusion-scaled state process under an EJWC scheduling policy.
This lemma is used in Section~\ref{S8}
to prove the lower bounds in Theorem~\ref{T-lbound}. 

\begin{lemma}\label{L7.1}
Suppose that under some sequence
$\{Z^{n},\;n\in\NN\}\subset\boldsymbol\fZ$ we have
\begin{equation}\label{E-moment}
\sup_{n}\;\limsup_{T\to\infty}\;\frac{1}{T}\,
\Exp^{Z^{n}}\biggl[\int_{0}^{T}\babs{\Hat{X}^{n}(s)}\,\D{s}\biggr]\;<\,\infty\,.
\end{equation}
Then any limit point $\uppi\in\cP(\RR^{2}\times\Act)$ of $\Phi^{Z^n}_{T}$,
defined in \eqref{E-emp},
as $(n,T)\to\infty$ satisfies
$\uppi\in\eom$.
\end{lemma}

\begin{proof}
Let $f \in \Cc^{\infty}_c(\RR^2)$, and define
\begin{multline}\label{EL7.1A}
\mathscr{D} f(\Hat{X}^n,s)\;\df\;
\Delta f(\Hat{X}^{n}(s))
- \sum_{i=1}^2 \partial_{i}f(\Hat{X}^{n}(s-)) \Delta \Hat{X}^{n}_i(s)\\
 - \frac{1}{2} \sum_{i, i'=1}^2 \partial_{ii'}f(\Hat{X}^{n}(s-))
\Delta \Hat{X}^{n}_i(s) \Delta \Hat{X}_{i'}^{n}(s)\,.
\end{multline}
By applying It{\^o}'s formula (see, e.g.,
Theorem~26.7 in \citet{kallenberg})
and using the definition of $\Phi^{Z^n}_{T}$ in
\eqref{E-emp} and $\Hat{X}^{n}$ in \eqref{hatXn-1}--\eqref{hatXn-2}, we obtain
\begin{multline}\label{EL7.1B}
\frac{\Exp\bigl[f(\Hat{X}^{n}(T))\bigr] -
 \Exp\bigl[f(\Hat{X}^{n}(0))\bigr]}{T}     
 \;=\; 
\int_{\RR^2 \times \Act} \mathcal{A}^{n}f(\hat{x},u)\,
\Phi^{Z^n}_{T} (\D{\Hat{x}}, \D{u})
+ \frac{1}{T}\;\Exp\Biggl[ \sum_{s\le T}\mathscr{D} f(\Hat{X}^n,s)\Biggr]\,,
\end{multline}
with $\Exp =\Exp^{Z^n}$.
Define
\begin{equation}\label{EL7.1C}
\begin{split}
\cA_{1,1}^{n}(\Hat{x},u) &\;\df\;  - \mu_{12}^{n}
(\Hat{x}_1 - (e\cdot \Hat{x})^+ u_1^c)
+ ( \mu_{11}^{n} - \mu_{12}^{n} ) (e\cdot \Hat{x})^{-} u_1^s
- \gamma^{n}_1 (e\cdot \Hat{x})^{+} u^c_1 + \ell^{n}_1\,,\\[5pt]
\cA_{2,1}^{n}(\Hat{x},u) &\;\df\; -  \mu_{22}^{n}
(\Hat{x}_2 - (e\cdot \Hat{x})^+ u_2^c)
- \gamma^{n}_2 (e\cdot \Hat{x})^{+} u^c_2 + \ell^{n}_2\,,\\[5pt]
\cA_{1,2}^{n}(\Hat{x},u) &\;\df\; \frac{1}{2}\biggl(\frac{\lambda^{n}_{1}}{n}
+ \mu_{11}^{n}z^*_{11} + \mu_{12}^{n}z^*_{12} 
+\frac{1}{\sqrt{n}}\,\mu_{12}^{n} \bigl(\Hat{x}_1 - (e\cdot \Hat{x})^+ u_1^c\bigr)\\
& \mspace{150mu}  + \frac{1}{\sqrt{n}}\,
(\mu_{11}^{n} - \mu_{12}^{n} ) (e\cdot \Hat{x})^{-} u_1^s
+ \frac{1}{\sqrt{n}}\,\gamma^{n}_1 (e\cdot \Hat{x})^{+} u^c_1 \biggr)\,, \\[5pt]
\cA_{2,2}^{n}(\Hat{x},u) &\;\df\; \frac{1}{2}\biggl(\frac{\lambda^{n}_{2}}{n}
+ \mu_{22}^{n}z^*_{22}
+\frac{1}{\sqrt{n}}\,\mu_{22}^{n} \bigl(\Hat{x}_2 - (e\cdot \Hat{x})^+ u_2^c \bigr)
+ \frac{1}{\sqrt{n}}\,\gamma^{n}_2 (e\cdot \Hat{x})^{+} u^c_2 \biggr)\,.
\end{split}
\end{equation}
Since $Z^{n}\in\Breve\fZ^n$,
the operator
$\mathcal{A}^{n}\colon\Cc^{\infty}_c(\sqrt{n}\Breve{B})\to
\Cc^{\infty}_c(\sqrt{n}\Breve{B}\times\Act)$
takes the form
\begin{equation*}
\mathcal{A}^{n}f(\hat{x},u) \;\df\; \sum_{i=1}^2
\Bigl(\cA_{i,1}^{n}(\Hat{x},u)\,\partial_{i}f(\Hat{x})
+  \cA_{i,2}^{n}(\Hat{x},u)\,\partial_{ii}f(\Hat{x}) \Bigr)\,.
\end{equation*}

Let
\begin{equation*}
\norm{f}_{\Cc^3} \df \sup_{x\in\RR^2}
\Bigl( \abs{f(x)} + \sum_{i=1,2} \abs{\partial_i f(x)}
+ \sum_{i,j=1}^{2} \abs{\partial_{ij} f(x)}
+ \sum_{i,j,k=1}^{2} \abs{\partial_{ijk} f(x)}\Bigr)\,.
\end{equation*}
By Taylor's formula, using also the fact that the jump size is 
$\frac{1}{\sqrt{n}}$, we obtain
\begin{align*}
\babs{\mathscr{D} f(\Hat{X}^n,s)}
&\;\le\;\kappa \norm{f}_{\Cc^{3}}
\sum_{i,j,k=1}^{2}
\babs{\Delta\Hat{X}^{n}_{i}(s)}\,\babs{\Delta \Hat{X}^{n}_{j}(s)}
\,\babs{\Delta \Hat{X}^{n}_{k}(s)}\\[5pt]
&\;\le\;\frac{\kappa' \norm{f}_{\Cc^{3}}}{\sqrt{n}}\sum_{i,i'=1}^{2}
\babs{\Delta\Hat{X}^{n}_{i}(s)\Delta \Hat{X}^{n}_{i'}(s)}\,,
\end{align*}
for some constants $\kappa$ and $\kappa'$ that do not depend on $n\in\NN$.
Let 
\begin{equation}\label{BarcXn}
\begin{split} 
\Bar{\mathcal{X}}^n_1(t) &\;\df\; \frac{\lambda^{n}_1}{n} + \frac{1}{n}
\mu^{n}_{11} Z^{n}_{11}(t) + \frac{1}{n}
\mu^{n}_{12} Z^{n}_{12}(t)
+ \frac{1}{n} \gamma^{n}_1 Q^{n}_1(t) \,,\\[5pt]
\Bar{\mathcal{X}}^n_2(t) &\;\df\; \frac{\lambda^{n}_2}{n} + \frac{1}{n}
\mu^{n}_{22} Z^{n}_{22}(t)
+ \frac{1}{n} \gamma^{n}_2 Q^{n}_2(t)
\end{split}
\end{equation}
for $t\ge0$. 
Since independent Poisson processes have no simultaneous jumps w.p.1.,  we have
\begin{equation*}
\frac{1}{T}\;\Exp\biggl[\int_{0}^{T}\sum_{i,i'=1}^{2}
\babs{\Delta\Hat{X}^{n}_{i}(s)\Delta \Hat{X}^{n}_{i'}(s)\,}\D{s}\biggr]
\;\le\;\frac{1}{T}\;\Exp\babss{\int_{0}^{T}
\bigl(\Bar{\mathcal{X}}^n_1(s) + \Bar{\mathcal{X}}^n_2(s) \bigr) \,\D{s}}\,,
\end{equation*}
and that the right hand side is uniformly bounded over $n\in\NN$
and $T>0$ by \eqref{E-moment}.
Thus, we have 
\begin{equation*}
\frac{1}{T}\;\Exp\Biggl[ \sum_{s\le T}\babs{\mathscr{D} f(\Hat{X}^n,s)}\Biggr]
\;\le\; \frac{\kappa' \norm{f}_{\Cc^{3}}}{T\sqrt{n}}
\Exp\biggl[\int_{0}^{T}\sum_{i,i'=1}^{2}
\babs{\Delta\Hat{X}^{n}_{i}(s)\Delta \Hat{X}^{n}_{i'}(s)\,}\D{s}\biggr] \;\to\; 0\,,
\end{equation*}
as $(n, T) \to \infty$. 
Therefore, taking limits in \eqref{EL7.1B}, we obtain 
\begin{equation*}
\limsup_{(n,T)\,\to\,\infty}\;
\int_{\RR^2\times \Act} \mathcal{A}^{n}f(\Hat{x},u)\,
\Phi^{Z^n}_{T} (\D{\Hat{x}}, \D{u})
\;=\;0\,.
\end{equation*}

Note that for $i=1,2$, $\cA_{i,1}^{n}$ tends to the drift of the
limiting diffusion $b_i$, while
$\cA_{i,2}^{n}$ tends to $\lambda_{i}$ as $n\to\infty$,
uniformly over  compact sets in $\RR^2 \times \Act$. 

Let $(n_{k},T_{k})$ be any sequence along which
$\Phi^{Z^n}_{T}$ converges to some $\uppi\in\cP(\RR^{2}\times\Act)$.
Let
\begin{equation*}
\Lg^{u} f(x)\;=\;
\sum_{i=1}^{2}\bigl[\lambda_i\,\partial_{ii} f(x)
+ b_i(x, u)\,\partial_i f(x)\bigr] \,. 
\end{equation*}
We have
\begin{multline}\label{L7.1X}
\int_{\RR^2 \times \Act} \Lg^{u} f(x)\,\uppi(\D{x}, \D{u}) -
\int_{\RR^2\times \Act} \mathcal{A}^{n}f(\Hat{x},u)\,
\Phi^{Z^n}_{T} (\D{\Hat{x}}, \D{u}) \\[3pt]
 \;=\;
\int_{\RR^2\times \Act}\Lg^{u} f(x)\bigl(\uppi(\D{x}, \D{u})
-\Phi^{Z^n}_{T} (\D{x}, \D{u})\bigr) \\[3pt]
+\int_{\RR^2\times \Act} \bigl(\Lg^{u} f(\Hat{x})-\mathcal{A}^{n}f(\Hat{x},u)\bigr) \,
\Phi^{Z^n}_{T} (\D{\Hat{x}}, \D{u})\,.
\end{multline}
The first term on the right hand side of \eqref{L7.1X} converges to $0$
as $n\to\infty$ by the convergence of $\Phi^{Z^n}_{T}$ to $\uppi$,
while the second term also converges to $0$ by
the uniform convergence of $\Lg^{u} f$ to $\mathcal{A}^{n}f$
on compact subsets of $\RR^2\times \Act$ and the tightness of
$\Phi^{Z^n}_{T}$.
Thus we obtain
\begin{equation*}
\int_{\RR^2 \times \Act} \Lg^{u} f(x)\,\uppi(\D{x}, \D{u})
\;=\;0 \,.
\end{equation*}
This completes the proof.  
\end{proof}

Before stating the second lemma, we first introduce a canonical
construction of scheduling policies from the
optimal control $v\in \Ussm$ for the diffusion control problems.
Recall the notation in Definition~\ref{D3.2}.

\begin{definition}\label{D7.1}
Let
$\varpi\colon \{ x \in \RR^{2}_{+}\colon e\cdot x \in \ZZ\} \to \ZZ^{2}_{+}$ be
a measurable map defined by
\begin{equation*}
\varpi(x) \;\df\; \bigl(\lfloor x_1 \rfloor, 
 e\cdot x  - \textstyle \lfloor x_1 \rfloor\bigr)\,,\qquad
x \in \RR^2\,.
\end{equation*}
For any precise  control $v \in \Ussm$, define the maps $q^{n}[v]$
and $y^{n}[v]$ by
\begin{equation*}
q^{n}[v](\Hat{x}) \;\df\;
\varpi \bigl(\bigl(e\cdot(\sqrt{n}\Hat{x} + n x^*)\bigr)^{+} v^c(\Hat{x})\bigr) \,,
\qquad
y^{n}[v](\Hat{x}) \;\df\;
\varpi \bigl(\bigl(e\cdot(\sqrt{n}\Hat{x} + n x^*)\bigr)^{-} v^s(\Hat{x})\bigr)\,.
\end{equation*}
for $\Hat{x}\in\sS^n$.
We also define define the map (Markov scheduling policy) $z^n[v]$ on $\Breve\sS^n$ by 
\begin{equation*}%\label{D7.1A}
z^{n}[v](\Hat{x}) \;\df\; \begin{bmatrix}
N_1^n - y^{n}_1[v](\Hat{x}) \quad &
x_1 - q^{n}_1[v](\Hat{x}) -  \bigl(N_1^n - y^{n}_1[v](\Hat{x})\bigr) \\[5pt]
0 \quad & x_2 - q^{n}_2[v](\Hat{x})
\end{bmatrix}\,,\quad \Hat{x}\in\Breve\sS^n\,. 
\end{equation*}
Compare this to \eqref{L2.1B}.
\end{definition}

\begin{corollary}\label{C7.1}
For any precise control $v \in \Ussm$ 
we have 
\begin{equation*}
e\cdot q^n[v]\bigl(\Hat{x}^n(x)\bigr) \wedge e \cdot y^n[v]\bigl(\Hat{x}^n(x)\bigr)
\;=\;0\,,
\quad\text{and}\quad z^n[v]\bigl(\Hat{x}^n(x)\bigr)\in\cZn(x)
\end{equation*}
for all $x\in\Breve{\sX}^{n}$, i.e., the JWC condition is
satisfied for $x\in\Breve{\sX}^{n}$.
\end{corollary}

\begin{proof}
This follows from Lemma \ref{L-JWC} and the definition of the
maps $q^n[v]$,  $y^n[v]$ and  $z^n[v]$. 
\end{proof}

The lemma which follows
asserts that if a sequence of EJWC scheduling policies is constructed using any
precise stationary stable Markov control in a way that the long-run average
moment condition in Lemma~\ref{L7.1} is satisfied, then any limit of the
mean empirical measures of the diffusion scaled processes
agrees with the ergodic occupation measure of the limiting diffusion
corresponding to that control.
This lemma is used in the proof of upper bounds in Theorem~\ref{T-ubound}.
Recall Definition~\ref{D2.1}.

\begin{lemma}\label{L7.2}
Let $v\in\Ussm$ be a  continuous precise control,
and $\bigl\{Z^{n}\,\colon\, n\in\NN\bigr\}$ be any sequence of
admissible scheduling policies such that each
$Z^n$ agrees with
the Markov scheduling policy $z^n[v]$ given in Definition~\ref{D7.1}
on $\sqrt{n}\Breve{B}$, i.e.,
$Z^n(t) = z^n[v]\bigl(\Hat{X}^n(t)\bigr)$ whenever $\Hat{X}^n(t)\in\sqrt{n}\Breve{B}$.
For $\Hat{x}\in \sqrt{n}\Breve{B}\cap\sS^n$, we define
\begin{align*}
u^{c,n}[v] (\Hat{x}) &\;\df\; \begin{cases}
\frac{q^{n}[v](\Hat{x})}{e\cdot q^{n}[v](\Hat{x})}
& \text{if~} e\cdot q^{n}[v](\Hat{x})>0\,,\\[5pt]
v^c(\Hat{x}) &\text{otherwise,}\end{cases}
\intertext{and} 
u^{s,n}[v](\Hat{x}) &\;\df\; \begin{cases}
\frac{y^{n}[v](\Hat{x})}{e\cdot y^{n}[v](\Hat{x})}
& \text{if~} e\cdot y^{n}[v](\Hat{x})>0\,,\\[5pt]
v^s(\Hat{x}) &\text{otherwise.}\end{cases}
\end{align*}
For the process $X^n$ under the scheduling policy $Z^n$,
define the mean empirical measures 
\begin{equation} \label{E-emp2}
\Tilde\Phi^{Z^n}_{T}(A\times B)\;\df\;
\frac{1}{T}\,\Exp^{Z^n}\biggl[\int_{0}^{T}\Ind_{A\times B}
\bigl(\Hat{X}^{n}(t),u^{n}[v]\bigl(\Hat{X}^{n}(t)\bigr)\bigr)\,\D{t}\biggr]
\end{equation}
for Borel sets $A\subset\sqrt{n}\Breve{B}$ and $B\subset \Act$.
Suppose that \eqref{E-moment} holds under this sequence  $\{Z^n\}$.
Then the ergodic occupation measure $\uppi_{v}$ of the controlled
diffusion in \eqref{sde} corresponding to $v$
is the unique limit point in $\cP(\Rd\times\Act)$
of $\Tilde\Phi^{Z^n}_{T}$ as $(n,T)\to\infty$.
\end{lemma}

\begin{proof}
It follows by Corollary~\ref{C7.1} that $\{Z^{n}\}\in\boldsymbol\fZ$.
Also, by the continuity of $v$, we have
\begin{equation}\label{EL7.2A}
\sup_{\Hat{x}\in\sS^n\cap K}\;
\abs{u^{n}[v](\Hat{x})-v(\Hat{x})}\;\to\; 0\qquad\text{as}~n\to\infty\,,
\end{equation}
for any compact set $K\subset\RR^2$.
Also, for any $f \in \Cc^{\infty}_c(\RR^2\times\Act)$, it holds that
\begin{equation*}
\int_{\RR^{2}\times\Act} f(\Hat{x},u) \,\Tilde\Phi^{Z^n}_{T}(\D{\Hat{x}},\D{u})
\;=\;\frac{1}{T}\;\Exp^{Z^n}\left[\int_{0}^{T} f\bigl(\Hat{X}^{n}(t),
u^{n}[v]\bigl(\Hat{X}^{n}(t)\bigr)\bigr)\,\D{t}\right]\,,
\end{equation*}
for all sufficiently large $n$ such that the support
of $f$ is contained in $\sqrt{n}\Breve{B}$.
Therefore, if $\uppi^{n}$ is any limit
point of $\Tilde\Phi^{Z^n}_{T}$ as $T\to\infty$,
and we disintegrate $\uppi^{n}$ as
\begin{equation}\label{EL7.2B}
\uppi^{n}(\D{\Hat{x}},\D{u})\;=\;\nu^{n}(\D{\Hat{x}})\,\xi^{n}(\D{u}\mid \Hat{x})\,,
\end{equation}
then we have
\begin{equation*}
\int_{\RR^2 \times \Act} \mathcal{A}^{n}f(\Hat{x},u)\, \uppi^{n} (\D{\Hat{x}}, \D{u})
\;=\;
\int_{\RR^2} \mathcal{A}^{n}f\bigl(\Hat{x},u^{n}[v](\Hat{x})\bigr)\,
\nu^{n}(\D{\Hat{x}})\,.
\end{equation*}
By Lemma~\ref{L7.1}, the sequence $\{\nu^{n}\}$ is tight.
Let $\{n\}\in\NN$ be any increasing sequence such
that $\nu^{n}\to\nu\in\cP(\RR^{2})$.
To simplify the notation, let
$\Tilde{\cA}^{n}f(\Hat{x})\df\mathcal{A}^{n}f\bigl(\Hat{x},u^{n}[v](\Hat{x}))$.
 We have
\begin{equation}\label{EL7.2C}
\int_{\RR^2} \Tilde{\cA}^{n}f\,\D\nu^{n} -
\int_{\RR^{2}}\Lg^{v}f\,\D\nu
\;=\;\int_{\RR^2} \bigl(\Tilde{\cA}^{n}f-\Lg^{v}f\bigr)\,\D\nu^{n}
+\int_{\RR^2} \Lg^{v}f\bigr(\D\nu^{n}-\D\nu\bigr)\,.
\end{equation}
It follows by \eqref{EL7.2A} that $\Tilde{\cA}^{n}f-\Lg^{v}f\to0$,
uniformly as $n\to\infty$, which
implies that the first term on the right hand side of \eqref{EL7.2C} converges
to $0$.
The second term does the same by the convergence of $\nu^{n}$
to $\nu$.
By Lemma~\ref{L7.1}, we have $\int_{\RR^2} \Tilde{\cA}^{n}f\,\D\nu^{n}\to0$
as $n\to\infty$.
Therefore, we obtain
\begin{equation*}
\int_{\RR^{2}}\Lg^{v}f(x)\,\nu(\D{x})\;=\;0\,,
\end{equation*}
and this means that $\nu$ is an invariant probability measure for
the diffusion
associated with the control $v$.
Next note that the Markov control $\xi^{n}$ in \eqref{EL7.2B}
agrees with $u^{n}[v](\Hat{x})$ when $\Hat{x}\in \sqrt{n}\Breve{B}\cap\sS^n$
by definition.
In other words,
$\xi^{n}(\D{u}\mid \Hat{x}) = \delta_{u^{n}[v](\Hat{x})}(u)$,
where $\delta$ denotes the Dirac measure.
It then follows by \eqref{EL7.2A} that $\xi^{n}$ converges
to $v$ as $n\to\infty$ in the topology of Markov controls
\cite[Section~2.4]{book}.
The ergodic occupation measure $\uppi_v\in\cP(\RR^2\times\Act)$ is given by
$\uppi_v(\D{x},\D{u}) \df \nu(\D{x},\D{u})\delta_{v(x)}(u)$.
With $g\in\Cc_c(\RR^2\times\Act)$, i.e., a continuous function with compact
support, we write
\begin{multline}\label{EL7.2D}
\babss{\int_{\RR^2\times\Act} g(x,u) \bigl(\uppi_v(\D{x},\D{u})
-\uppi^{n}(\D{x},\D{u})\bigr)} \;\le\;
\babss{\int_\Act
\biggl(\int_{\RR^2} g(x,u) \bigl(\nu(\D{x})-\nu^n(\D{x})\bigr)\biggr)
\xi^{n}(\D{u}\mid x)}\\[5pt]
+ \babss{\int_\Act
\biggl(\int_{\RR^2} g(x,u)\nu(\D{x})\biggr)
\bigr(\xi^{n}(\D{u}\mid x)-\delta_{v(x)}(u)\bigr)}\,.
\end{multline}
The first term on the right hand side of \eqref{EL7.2D} converges to $0$
as $n\to\infty$ by  the convergence of $\nu^{n}\to\nu$ in $\cP(\RR^{2})$.
Since $\nu$ has a continuous density, the second term also converges
to $0$ as $n\to\infty$ by \cite[Lemma~2.4.1]{book}.
Therefore \eqref{EL7.2D} shows that $\uppi^{n}\to\uppi_v$ in
$\cP(\RR^2\times\Act)$, and this completes the proof.
\end{proof}

%%%%%%%%%%%%%%%%%%%%%%%%%%%%%%%%%%%%%%%%%%%%%%%%%%%%%%%%%%%%%%%%%%%%%
\section{Proof of the lower bounds}\label{S8}

In this section, we prove the lower bounds in
Theorem~\ref{T-lbound}. 
The following lemma which applies to the diffusion-scaled process,
is analogous to Lemma~3.1~(c) for the diffusion limit in \citet{AP15}. 

\begin{lemma}\label{L8.1}
There  exist constants
$C_{1}$ and $C_{2}$ independent of $n$ such that
\begin{equation}\label{EL8.1A}
\limsup_{T\to\infty}\;\frac{1}{T}\,
\Exp^{Z^n}\biggl[\int_{0}^{T}\babs{\Hat{X}^{n}(s)}^m\,\D{s}\biggr]
\;\le\;
C_{1}+C_{2} J_{\mathsf{o}}
\bigl(\Hat{X}^{n}(0),\Hat{Z}^{n}\bigr)\qquad\forall\,n\in\NN\,,
\end{equation}
for any sequence $\{Z^{n}\in\fZ^{n},\;n\in\NN\}$,
where $m\ge1$ is as in \eqref{Erhat}.
\end{lemma}

\begin{proof}
Let 
$\mathscr{V}(x)\;\df\;  \mathscr{V}_{1}(x_1)+\beta \mathscr{V}_{2}(x_2)$,
$x\in \RR^2$, where $\beta$ is a positive constant
to be determined later, and
$\mathscr{V}_{i}(x)=\frac{\abs{x_{i}}^{m+1}}{\sqrt{1 +\abs{x_{i}}^{2}}}$ for $m\ge1$.
By applying It\^o's formula on $\mathscr{V}$, with
$\Exp =\Exp^{Z^n}$, we obtain from
\eqref{hatXn-1} that for $t\ge 0$,
\begin{equation}\label{EL8.1B}
\Exp\bigl[\mathscr{V}(\Hat{X}^{n}(t))\bigr] \;=\;
\Exp\bigl[\mathscr{V}(\Hat{X}^{n}(0)) \bigr] +
\Exp\biggl[\int_{0}^{t}
\mathscr{A}^n
\mathscr{V} \bigl(\Hat{X}^{n}(s), \Hat{Z}^{n}(s)\bigr)\,\D{s} \biggr] 
+\Exp \Biggl[\sum_{s\le t} \mathscr{D} \mathscr{V}(\Hat{X}^n,s) \Biggr]\,,  
\end{equation}
where $\mathscr{D} \mathscr{V}(\Hat{X}^n,s)$ is defined as in \eqref{EL7.1A}, 
\begin{equation*}
\mathscr{A}^n \mathscr{V} 
\bigl(\Hat{x}, \Hat{z} \bigr) \;\df \; \sum_{i=1}^2
\biggl(\mathscr{A}^{n}_{i,1}(\hat{x}, \Hat{z}) \partial_i \mathscr{V}(\hat{x})
+\mathscr{A}^{n}_{i,2}(\Hat{x}, \Hat{z})\partial_{ii} \mathscr{V}(\Hat{x}) \biggr)\,,
\end{equation*}
and
\begin{align*}
\mathscr{A}^{n}_{1,1}\bigl(\Hat{x}, \Hat{z}\bigr)&\;\df\;
\ell^{n}_{1} - \mu^{n}_{11} \Hat{z}_{11} - \mu^{n}_{12} \Hat{z}_{12} -\gamma^{n}_{1}
\bigl(\Hat{x}_1 - \Hat{z}_{11} - \Hat{z}_{12} \bigr)\,,\\[5pt]
\mathscr{A}^{n}_{2,1}\bigl(\Hat{x}, \Hat{z}\bigr)&\;\df\;
\ell^{n}_{2} - \mu^{n}_{22} \Hat{z}_{22}-\gamma^{n}_{2}
\bigl(\Hat{x}_2 - \Hat{z}_{22}\bigr)\,,\\[5pt]
\mathscr{A}^{n}_{1,2}\bigl(\Hat{x}, \Hat{z}\bigr)&\;\df\;
\frac{1}{2}\biggl[\frac{\lambda^{n}_{1}}{n} + \bigl( \mu_{11}^{n} z_{11}^*
+  \mu_{12}^{n}  z_{12}^* \bigr) 
+\frac{1}{ \sqrt{n}} \bigl(\mu_{11}^{n} \Hat{z}_{11} + \mu_{12}^{n} \Hat{z}_{12} \bigr) 
+ \frac{\gamma^{n}_1}{ \sqrt{n}} \bigl(  \Hat{x}_1- 
\Hat{z}_{11} -   \Hat{z}_{12}  \bigr)
\biggr]\,,  \\[5pt]
\mathscr{A}^{n}_{2,2}\bigl(\Hat{x}, \Hat{z}\bigr)&\;\df\;
\frac{1}{2}\biggl[\frac{\lambda^{n}_{2}}{n} + \mu_{22}^{n} z_{22}^*
+\frac{1}{ \sqrt{n}} \mu_{22}^{n}  \Hat{z}_{22}  
+ \frac{\gamma^{n}_2}{ \sqrt{n}} \bigl(  \Hat{x}_2 - \Hat{z}_{22}   \bigr)
\biggr]\,,  
\end{align*}
for $\Hat{x} \in \sS^n$,
and $\Hat{z}_{ij} \df \frac{1}{\sqrt{n}}(z_{ij} - nz^*_{ij})$
for $z_{ij} \in \ZZ_+ $ and $z^*$ defined in \eqref{Ez*}.
We also use the nonnegative variables $\Hat{q}_i$ and $\Hat{y}_i$, $i=1,2$,
which are defined as functions of $\Hat{x}$ and $\Hat{z}$ via
the balance equations \eqref{baleq-hat}, keeping in mind that
the work conservation condition holds for these.

Define
\begin{align*}
\Bar{\mathscr{A}}_{1,1}\bigl(x, z \bigr) & \;\df\;
\ell_{1} - \mu_{11} z_{11} - \mu_{12} z_{12} -\gamma_{1}
\bigl(x_1- z_{11} - z_{12}\bigr)\,, \\[5pt]
\Bar{\mathscr{A}}_{2,1}\bigl(x, z\bigr) &\;\df\;
\ell_{2} - \mu_{22} z_{22}-\gamma_{2}
\bigl(x_2-  z_{22}\bigr)\,,
\end{align*}
for $x\in \RR^2$ and $z \in \RR^{2\times 2}$. 

By the convergence of the parameters in Assumption~\ref{as-para},
we have that for $i=1,2$, 
\begin{equation}\label{EL8.1C}
\babs{\Bar{\mathscr{A}}_{i,1}\bigl(\hat{x}, \Hat{z} \bigr)
-\mathscr{A}^{n}_{i,1}(\Hat{x}, \Hat{z})}
\;\le\; \kappa_{1}(n)(\norm{\Hat{x}}+ \norm{\Hat{z}})
\end{equation}
for some constant $\kappa_{1}(n)\searrow 0$ as $n\to\infty$.

Let $\Hat\xi\df (e\cdot \Hat{q})\wedge (e\cdot\Hat{y})$.
We claim that if $\Hat{\xi}^{n}>0$ then
$\Hat{q}_1=0$,  $\Hat{y}_2=0$ by the work conservation condition.
Indeed since $\Hat{q}_i\wedge\Hat{y}_2=0$ for $i=1,2$,
then $\Hat{\xi}^{n}>0$ implies that $\Hat{y}_2=0$, which
in turn implies that $\Hat{y}_1>0$.
This of course implies that $\Hat{q}_1=0$.

If $\Hat\xi=\Hat{y}_1$, then by the balance equations we have
$\Hat{z}_{11} = -\Hat\xi$, $\Hat{z}_{12}=\Hat{x}_1+\Hat\xi$, and
$\Hat{z}_{22}=\Hat{x}_2-\Hat{q}_2$.
On the other hand, if $\Hat\xi=\Hat{q}_2$, then we obtain
$\Hat{z}_{11} = \Hat{x}_1+\Hat{x}_2-\Hat\xi$,
$\Hat{z}_{12}=\Hat\xi-\Hat{x}_2$, and
$\Hat{z}_{22}=\Hat{x}_2-\Hat\xi$.
Hence when $\Hat\xi>0$ we have
\begin{equation}\label{EL8.1D}
\begin{split}
\begin{alignedat}[t]{100}
&\begin{aligned}
\Bar{\mathscr{A}}_{1,1}(\Hat{x},\Hat{z}) &\;=\;  - \mu_{12} \Hat{x}_1
+ (\mu_{11} - \mu_{12})\,\Hat\xi + \ell_1\\[5pt]
\Bar{\mathscr{A}}_{2,1}(\Hat{x},\Hat{z}) &\;=\; -  \mu_{22} (\Hat{x}_2 - \Hat{q}_2)
- \gamma_2 \Hat{q}_2 + \ell_2
\end{aligned}
&&\qquad \text{if~} \Hat{y}_1<\Hat{q}_2\,,   
\\[7pt]
&\begin{aligned}
\Bar{\mathscr{A}}_{1,1}(\Hat{x},\Hat{z}) &\;=\;  - \mu_{11} \Hat{x}_1
+ (\mu_{11} - \mu_{12})\, (\Hat\xi- \Hat{x}_2) + \ell_1\\[5pt]
\Bar{\mathscr{A}}_{2,1}(\Hat{x},\Hat{z}) &\;=\; -  \mu_{22} (\Hat{x}_2 -  \Hat\xi)
- \gamma_2 \Hat\xi + \ell_2
\end{aligned}
&&\qquad \text{if~}  \Hat{y}_1\ge\Hat{q}_2\,.  
\end{alignedat}
\end{split}
\end{equation}
and when $\Hat\xi=0$, we can use the parameterization
$\Hat{q}=(e\cdot \Hat{x})^+ u^c$ and $\Hat{y}=(e\cdot \Hat{x})^- u^s$
and \eqref{EL7.1C} to obtain
\begin{equation}\label{EL8.1E}
\begin{split}
\begin{alignedat}[t]{100}
&\begin{aligned}
\Bar{\mathscr{A}}_{1,1}(\Hat{x},\Hat{z}) &\;=\;  - \mu_{12} \Hat{x}_1
+ (\mu_{12} - \gamma_1)\,\Hat{q}_1 + \ell_1\\[5pt]
\Bar{\mathscr{A}}_{2,1}(\Hat{x},\Hat{z}) &\;=\; -  \mu_{22} (\Hat{x}_2 - \Hat{q}_2)
- \gamma_2 \Hat{q}_2 + \ell_2
\end{aligned}
&&\qquad \text{if~} (e\cdot \Hat{x})^+>0\,,
\\[7pt]
&\begin{aligned}
\Bar{\mathscr{A}}_{1,1}(\Hat{x},\Hat{z}) &\;=\;   
- (\mu_{12}(1- u_1^s)+\mu_{11} u_1^s) \Hat{x}_1
- (\mu_{11} - \mu_{12})\, \Hat{x}_2\,u^s_1 + \ell_1 \\[5pt]
\Bar{\mathscr{A}}_{2,1}(\Hat{x},\Hat{z}) &\;=\; -  \mu_{22} \Hat{x}_2 + \ell_2
\end{aligned}
&&\qquad \text{if~}  (e\cdot \Hat{x})^- \ge 0\,.
\end{alignedat}
\end{split}
\end{equation}

It follows by the above analysis that
\begin{equation}\label{EL8.1F}
\abs{z_{ij}}\,\in\,\order(\abs{x}+\abs{q})\,,\qquad i,j\in\{1,2\}\,.
\end{equation}
Hence we have
\begin{equation}\label{EL8.1G}
\mathscr{A}^{n}_{i,2}(\Hat{x},\Hat{z})
\;\in\; \order(1+n^{-\nicefrac{1}{2}}\abs{\Hat{x}})\,.
\end{equation}
Following the steps in the proof of Lemma~\ref{L4.1},
and also using the fact that $\Hat\xi\le e\cdot\Hat{q}$
and Young's inequality, it follows by
\eqref{EL8.1D}--\eqref{EL8.1E} that we can choose $\beta>0$ and positive constants
$c_1$ and $c_2$ such that
\begin{equation}\label{EL8.1H}
\sum_{i=1}^2 \Bar{\mathscr{A}}^{n}_{i,1}(\hat{x}, \Hat{z})
\partial_i \mathscr{V}(\hat{x})
\;\le\; - c_{1} \mathscr{V}(\Hat{x}) + c_2 \bigl(1+ \abs{\Hat{q}}^m\bigr)\,.
\end{equation}
Thus, by \eqref{EL8.1C}, \eqref{EL8.1G}, and \eqref{EL8.1H}
we obtain
\begin{equation} \label{EL8.1I}
\mathscr{A}^n \mathscr{V}(\Hat{x},\Hat{z})
\;\le\; - c'_{1} \mathscr{V}(\Hat{x}) + c'_2 \bigl(1+ \abs{\Hat{q}}^m\bigr)
\end{equation}
for some positive constants $c'_1$ and $c'_2$.

For the jumps in \eqref{EL8.1B}, we first note that by the definition of
$\mathscr{V}_{i}$,
since  there exists a
positive constant $c_3$ such that 
\begin{equation*}
\sup_{\abs{x_i'-x_i} \le 1}\, \babs{\mathscr{V}_{i}''(x_i')} \;\le\;
c_3 \bigl(1+\abs{x_i}^{m-2}\bigr)
\qquad\forall\,x_i\in\RR\,.
\end{equation*}
Since also the jump size is of order $\frac{1}{\sqrt{n}}$,
then by Taylor's expansion we obtain 
\begin{equation*}
\Delta\mathscr{V}_{i}\bigl(\Hat{X}^{n}(s)\bigr)
-\mathscr{V}_{i}'\bigl(\Hat{X}^{n}(s-)\bigr)\cdot\Delta\Hat{X}^{n}_{i}(s)
\;\le\; \frac{1}{2} \sup_{\abs{x_i' - \Hat{X}^{n}_i(s-)} \le 1} \;
\abs{\mathscr{V}_{i}''(x_i')}\, (\Delta\Hat{X}^{n}_{i}(s))^2\,.
\end{equation*}
for $i=1,2$.
Recall the definitions of $\Bar{\mathcal{X}}^n_1$ and
$\Bar{\mathcal{X}}^n_2$ in \eqref{BarcXn}. 
Thus, for $i=1,2$, using also \eqref{EL8.1F}, we obtain
\begin{align}\label{EL8.1J}
\Exp \Biggl[\sum_{s\le t} \mathscr{D} \mathscr{V}_{i}(\Hat{X}^n,s) \Biggr]
 &\le\; \Exp \Biggl[ \sum_{s\le t} c_3
\Bigl(1+\abs{\Hat{X}^{n}_{i}(s-)}^{m-1} \Bigr)
\bigl(\Delta\Hat{X}^{n}_{i}(s)\bigr)^2 \Biggr]
\nonumber\\[5pt]
&\le\; c_3 \Exp \biggl[ \int_0^t \Bigl(1+ \abs{\Hat{X}^{n}_{i}(s)}^{m-1} \Bigr)
\Bar{\mathcal{X}}^n_i(s) \,\D{s} \biggr]
\nonumber\\[5pt]
&\le\; c_4\Exp \biggl[ \int_0^t \Bigl(1+ \abs{\Hat{X}^{n}_{i}(s)}^{m-1} \Bigr)
\Bigl(1+n^{-\nicefrac{1}{2}}\bigl(\abs{\Hat{X}^{n}(s)}+\abs{\Hat{Q}^{n}(s)}
\bigr)\Bigr)\,\D{s}
\biggr]\,, 
\end{align}
for some positive constant $c_4$.
Therefore, by \eqref{EL8.1B}, \eqref{EL8.1I},
and \eqref{EL8.1J}, we can choose positive constants $c_5$ and $c_6$
such that
\begin{equation*}
\Exp\bigl[\mathscr{V}(\Hat{X}^{n}(t))\bigr] \;\le\;
\Exp\bigl[\mathscr{V}(\Hat{X}^{n}(0))\bigr]
+ c_6 t - c_5 \Exp
\biggl[ \int_0^t\abs{\Hat{X}^{n}(s)}^{m}\,\D{s}  \biggr]
+ c_6 \Exp \biggl[ \int_0^t \abs{\Hat{Q}^{n}}^{m}\,\D{s}\biggr]\,.
\end{equation*}
Dividing by $t$ and taking limits as $t\to\infty$, establishes
\eqref{EL8.1A}.
\end{proof}

We are now ready to prove Theorem~\ref{T-lbound}.

\begin{proof}[Proof of Theorem~\ref{T-lbound}]

Let $Z^{n}\in\mathfrak{Z}^{n}$, $n\in\NN$, be an arbitrary sequence of
scheduling policies in $\boldsymbol\fZ$, and
let $\Phi^{n}\df\Phi^{Z^{n}}$ as defined in \eqref{E-emp}.
Without loss of generality we assume that along some increasing
sequence $\{n_{k}\}\subset\NN$, we have
$\sup_{k}\,J\bigl(\Hat{X}^{n_{k}}(0),Z^{n_{k}}\bigr)<\infty$;
otherwise there is nothing to prove.
By Lemmas~\ref{L7.1} and \ref{L8.1},  the sequence of mean empirical measures
$\{\Phi^{n_{k}}_{T}\,\colon\, T>0, \;k \ge 1 \}$ is tight
and any subsequential limit as $(n_{k},T)\to\infty$ is in $\eom$.
Select any subsequence $\{T_{k},n'_{k}\}\subset\RR_{+}\times\{n_{k}\}$,
with $T_{k}\to\infty$, as $k\to\infty$, and such that
\begin{equation*}
J\bigl(\Hat{X}^{n'_{k}}(0),Z^{n'_{k}}\bigr) \;\le\; \frac{1}{k}
+\liminf_{\ell\to\infty}\;J\bigl(\Hat{X}^{n_{\ell}}(0),Z^{n_{\ell}}\bigr)\,,
\end{equation*}
and
\begin{equation*}
\int_{\RR^{2}\times\Act} r(x,u)\,\Phi^{n'_{k}}_{T_{k}} (\D{x}, \D{u})\;\le\;
J\bigl(\Hat{X}^{n'_{k}}(0),Z^{n'_{k}}\bigr)+ \frac{1}{k}\,,
\end{equation*}
for all $k\in\NN$,
and extract any further subsequence, also denoted as $\{T_{k},n'_{k}\}$,
along which $\Phi^{n'_{k}}_{T_{k}}\to\Hat\uppi\in\eom$.
Since $r$ is nonnegative, taking limits as $k\to\infty$ we obtain
\begin{equation*}
\liminf_{k\to\infty}\;J\bigl(\Hat{X}^{n_{k}}(0),Z^{n_{k}}\bigr)
\;\ge\;  \Hat\uppi (r)\;\ge\;\varrho^{*}\,.
\end{equation*} 
This proves part~(i).

We next show the lower bound (ii) for the constrained problem.
Repeating the same argument as in part~(i),
suppose that 
$\sup_{k}\,J_{\mathsf{o}}\bigl(\Hat{X}^{n_{k}}(0),Z^{n_{k}}\bigr)<\infty$
along some increasing
sequence $\{n_{k}\}\subset\NN$.
As in the proof of part~(i),
let $\Hat\uppi\in\cP(\RR^2 \times \Act)$ be a limit of
$\Phi^{n}_{T}$ as $(n,T)\to\infty$. 
Recall the definition of $r_{j}$ in \eqref{E-rj}.
Since $r_{j}$ is bounded below, taking limits, we obtain
$\Hat\uppi(r_{j})\;\le\; \updelta_j$, $j=1,2$.
Therefore $\Hat\uppi\in\sH(\updelta)$, and by optimality we must have
$\Hat\uppi(r_{\mathsf{o}})\ge\varrho^{*}_{\mathsf{c}}$.
Similarly, we obtain,
\begin{equation*}
\liminf_{k\to\infty}\;J_{\mathsf{o}}\bigl(\Hat{X}^{n_{k}}(0),Z^{n_{k}}\bigr)
\;\ge\;\Hat\uppi(r_{\mathsf{o}})\;\ge\;\varrho^{*}_{\mathsf{c}} \,.
\end{equation*}
This proves part~(ii). 

The result in part~(iii) for the fairness problem
follows along the same lines as part~(ii).
With $\Hat{\uppi}$ as in part~(ii), we have
\begin{equation}\label{PT5.1A}
\liminf_{k\to\infty}\;J_{\mathsf{o}}\bigl(\Hat{X}^{n_{k}}(0),Z^{n_{k}}\bigr)
\;\ge\;\Hat\uppi(r_{\mathsf{o}})\,.
\end{equation}
The uniform integrability of
\begin{equation*}
\frac{1}{T}\;
\Exp^{Z^n} \left[\int_{0}^{T}
\bigl(\Hat{Y}^{n}_j(s)\bigr)^{\Tilde{m}}\,\D{s}\right]\,,\qquad j=1,2\,,
\end{equation*}
which follows by \eqref{E-apriori} and the assumption that $\Tilde{m}<m$,
together with \eqref{ET5.1A}, imply that
\begin{equation*}
(\uptheta-\epsilon) \Hat\uppi(r_2)\;\le\; \Hat\uppi(r_1)
\;\le\; (\uptheta+\epsilon) \Hat\uppi(r_2)\,.
\end{equation*}
Therefore, $\Hat\uppi(r_1)=\Tilde\uptheta(\epsilon)\Hat\uppi(r_2)$
for some $\Tilde\uptheta(\epsilon)$ satisfying
$\abs{\Tilde\uptheta(\epsilon)-\uptheta}\le\epsilon$.
Let
\begin{equation*}
\Tilde\varrho\;\df\;  \inf_{\uppi\,\in\,\sH_{\mathsf{f}}(\Tilde\uptheta(\epsilon))}\;
\uppi(r_{\mathsf{o}})\,,
\end{equation*}
and $\uplambda^*$ denote the Lagrange multiplier for the problem
in Theorem~\ref{T4.3}.
It is clear that $\Hat\uppi(r_{\mathsf{o}})\ge \Tilde\varrho$.
Writing $\Hat\uppi(r_1)=\Tilde\uptheta(\epsilon)\Hat\uppi(r_2)$
as $\Hat\uppi(r_1)-\uptheta\Hat\uppi(r_2)=
\bigl(\Tilde\uptheta(\epsilon)-\uptheta)\Hat\uppi(r_2)$,
we obtain by \cite[Theorem~1, p.~222]{Luenberger} that
\begin{align}\label{PT5.1B}
\varrho_{\mathsf{f}}^{*}-\Tilde\varrho
&\;\le\; \babs{\uplambda^*\bigl(\Tilde\uptheta(\epsilon)-\uptheta)\Hat\uppi(r_2)}
\nonumber\\[5pt]
&\;\le\; \varepsilon\,\babs{\uplambda^*\Hat\uppi(r_2)} \,.
\end{align}
Without loss of generality, we may assume that
$\Hat\uppi(r_{\mathsf{o}})\le \varrho_{\mathsf{f}}^{*}$; otherwise
\eqref{ET5.1B} trivially follows by \eqref{PT5.1A}.
By \eqref{E-apriori} and Jensen's inequality we have
\begin{align}\label{PT5.1C}
\Hat\uppi(r_2)&\;\le\;
\Hat\kappa \bigl(1 + \uppi(r_{\mathsf{o}})^{\nicefrac{\Tilde{m}}{m}}\bigr)
\nonumber\\[5pt]
&\;\le\;\Hat\kappa \bigl(1 + (\varrho_{\mathsf{f}}^{*})^{\nicefrac{\Tilde{m}}{m}}\bigr)
\end{align}
for some constant $\Hat\kappa$.
Therefore combining \eqref{PT5.1B}--\eqref{PT5.1C}, we obtain
\begin{align*}
\Hat\uppi(r_{\mathsf{o}}) &\;\ge\; \Tilde\varrho\\[5pt]
&\;\ge\; \varrho_{\mathsf{f}}^{*}-\epsilon\,\abs{\uplambda^*}\Hat\kappa
\bigl(1 + (\varrho_{\mathsf{f}}^{*})^{\nicefrac{\Tilde{m}}{m}}\bigr)\,,
\end{align*}
and \eqref{ET5.1B} follows by this estimate and \eqref{PT5.1A}.
This completes the proof.
\end{proof}

\section{Proof of the upper bounds}\label{S9}

In this section, we prove the upper bounds in
Theorem~\ref{T-ubound}. We need the following lemma. 

\begin{lemma}\label{L9.1}
Let $\Lyap_{k,\beta}$ be as in \eqref{Lyapk}.
Suppose $v\in\Ussm$ is such that for some
positive constants $C_1$, $C_2$, $\beta$ and $k\ge2$, it holds that
\begin{equation*} \label{EL6.4a}
\Lg^{v} \Lyap_{k,\beta}(x) \;\le\; C_1  - C_2\, \Lyap_{k,\beta}(x)\qquad
\forall x\in\RR^2\,.
\end{equation*} 
Let $\Hat{X}^{n}$ denote the diffusion-scaled state process under the
scheduling policy $z^{n}[v]$ in Definition~\ref{D7.1}, and
$\widehat\cL_n$ be its generator.
Then, there exists $n_{0}\in\NN$ such that
\begin{equation*}
\widehat\cL_n \Lyap_{k,\beta}(\Hat{x}) \;\le\;
C'_1  - C'_2 \,\Lyap_{k,\beta}(\Hat{x})\qquad
\forall \Hat{x}\in\Breve\sS^n\,,
\end{equation*} 
for some positive constants $C'_1$ and $C'_2$, and for all $n\ge n_0$.
\end{lemma}

\begin{proof}
See Appendix~\ref{App1}.
\end{proof}

\begin{proof}[Proof of Theorem~\ref{T-ubound}]
We first prove part~(i) for the unconstrained problem.
Recall the definition in \eqref{Lyapk}.
Let $k=m+1$.
By Theorems~5.5 in \citet{AP15} and Lemma~\ref{L4.1},
there exists a continuous precise
control $v_\epsilon\in\Ussm$ which is $\epsilon$-optimal for (P1$'$) 
and satisfies
\begin{equation}\label{PT5.2A}
\Lg^{v_\epsilon} \Lyap_{k,\beta}(x) \;\le\;
c_1 - c_2 \,\Lyap_{k,\beta}(x)\qquad \forall x\in \RR^2\,,
\end{equation}
for any $\beta\ge\Bar\beta$ defined in \eqref{E-barbeta}, and
for some positive constants $c_1$, $c_2$ which depend on $\beta$.
Recall Definition~\ref{D2.1}.
The scheduling policy that we apply to the $n^\text{th}$ system is as follows:
Inside the ball $n \Breve{B}$ we apply the Markov
policy in Definition~\ref{D7.1} $z^n[v_\epsilon]$,
while outside this ball we apply the Markov policy $\Check{z}^n$ in
Definition~\ref{D3.1}.
Let $Z^n$ denote this concatenated policy.
By Proposition~\ref{P3.1} and Lemma~\ref{L9.1}
there exist positive constants $C_1$, $C_2$, $\beta$, and $n_0\in\NN$, such that
\begin{equation} \label{PT5.2B}
\widehat\cL_n^{Z^n} \Lyap_{k,\beta}(\Hat{x}) \;\le\;
C_1  - C_2\, \Lyap_{k,\beta}(\Hat{x})\qquad
\forall \Hat{x}\in\sS^n\,,\quad\forall\, n\ge n_0\,.
\end{equation} 
Let $\Tilde\Phi^n_T\equiv \Tilde\Phi^{Z^n}_{T}$ as defined in \eqref{E-emp2}.
We define
\begin{align*}
\Hat{q}(\Hat{x}) &\;\df\; \bigl(\Hat{x}_1 - Z^n_{11}(\Hat{x})- Z^n_{12}(\Hat{x}),
\Hat{x}_2 - Z^n_{22}(\Hat{x})\bigr)\,,\\[3pt]
\Hat{y}(\Hat{x}) &\;\df\; \bigl( - Z^n_{11}(\Hat{x})- Z^n_{12}(\Hat{x}),
 - Z^n_{22}(\Hat{x})\bigr)\,,
\end{align*}
By \eqref{PT5.2B}
we have $\sup_{n\ge n_0}\,J(\Hat{X}^n(0),Z^n)<\infty$,
and by  Birkhoff's ergodic theorem for each $n\ge n_0$ there exists $T_n\in\RR_+$,
such that
\begin{equation}\label{PT5.2C}
\babss{\int_{\RR^2\times\Act}
\Hat{r}\bigl( (e\cdot\Hat{q}(\Hat{x})\bigr)^+\,u^c,\,
(e\cdot\Hat{y}(\Hat{x})\bigr)^+\,u^s\bigr)\,
\Tilde\Phi^n_T(\D{\Hat{x}},\D{u})-J(\Hat{X}^n(0),Z^n)}\;\le\; \frac{1}{n}\,,
\end{equation}
for all $ T\ge T_{n}$ and $n\ge n_{0}$.
By \eqref{PT5.2B} the sequence $\{T_n\}$ can be selected so as to
also satisfy
\begin{equation}\label{PT5.2D}
\sup_{n\ge n_0}\;\sup_{T\ge T_n}\;
\int_{\RR^2\times\Act} \Lyap_{k,\beta}
(\Hat{x})\,\Tilde\Phi^n_T(\D{\Hat{x}},\D{u})\;<\;\infty\,.
\end{equation}
Without loss of generality we assume that $T_n\to\infty$.
Hence, by uniform integrability which is implied by \eqref{PT5.2D},
together with \eqref{PT5.2C}
for any $\eta>0$ there exists a ball $B_{\eta}$  such that
\begin{equation}\label{PT5.2E}
\babss{\int_{B_{\eta}\times\Act}
\Hat{r}\bigl( (e\cdot\Hat{q}(\Hat{x})\bigr)^+\,u^c,\,
(e\cdot\Hat{y}(\Hat{x})\bigr)^+\,u^s\bigr)\,
\Tilde\Phi^n_T(\D{\Hat{x}},\D{u})-J(\Hat{X}^n(0),Z^n)}\;\le\; \frac{1}{n}+\eta \,,
\end{equation}
for all $ T\ge T_{n}$ and  $n\ge n_{0}$.

By JWC on $\{\Hat{x}\in\sqrt{n}\Breve{B}\}$, we have
$(e\cdot\Hat{q}(\Hat{x})\bigr)^+ = (e\cdot\Hat{x})^+$ and
$(e\cdot\Hat{y}(\Hat{x})\bigr)^+ = (e\cdot\Hat{x})^-$
for all $\Hat{x}\in B_{\eta}$, and for all large
enough $n$ by Corollary~\ref{C7.1}.
On the other hand,  $\Tilde\Phi^n_T$ converges,
as $(n,T)\to\infty$, to $\uppi_{v_\epsilon}$ in
$\cP(\RR^2\times\Act)$ by Lemma~\ref{L7.2}.
Therefore
\begin{equation}\label{PT5.2F}
\int_{B_{\eta}\times\Act}
\Hat{r}\bigl( (e\cdot\Hat{q}(\Hat{x})\bigr)^+\,u^c,\,
(e\cdot\Hat{y}(\Hat{x})\bigr)^+\,u^s\bigr)\,\Tilde\Phi^n_{T_n}(\D{\Hat{x}},\D{u})
\;\xrightarrow[n\to\infty]{}\;
\int_{B_{\eta}\times\Act} r(x,u)\,\uppi_{v_\epsilon}(\D{x},\D{u})\,.
\end{equation}
By \eqref{PT5.2E}--\eqref{PT5.2F} we obtain
\begin{equation*}
\limsup_{n\to\infty}\;J(\Hat{X}^n(0),Z^n) \;\le\; \varrho^{*} + \epsilon +\eta\,.
\end{equation*}
Since $\eta$ and $\epsilon$ are arbitrary, this completes the proof
of part (i).

We next show the upper bound for the constrained problem.
Let 
$\epsilon>0$ be given.
By Theorem~5.7 in \cite{AP15} and Lemma~\ref{L4.1}, there exists a continuous
precise control $v_{\epsilon} \in \Ussm$ and constants
$\updelta^{\epsilon}_j<\updelta_j$, $j=1,2$, satisfying
$\uppi_{v_{\epsilon}}(r_{\mathsf{o}})\le \varrho^*_{\mathsf{c}}+\epsilon$,
and $\uppi_{v_{\epsilon}}(r_j)\le \updelta^{\epsilon}_j$, $j=1,2$,
and \eqref{PT5.2A} holds.
Let $Z^n$ be the Markov policy constructed in part (i) by concatenating
$z^n[v_\epsilon]$ and $\Check{z}^n$.
Following the proof of part (i) and choosing $\eta$ small enough,
i.e., $\eta< \epsilon\wedge\frac{1}{2}\min(\updelta_j-\updelta^{\epsilon}_j,\;j=1,2)$,
we obtain
\begin{align*}
\limsup_{n\to\infty}\;J_{\mathsf{o}}(\Hat{X}^n(0),Z^n) &\;\le\;
\varrho^{*}_{\mathsf{c}} + 2\epsilon\,,\\[5pt]
\limsup_{n\to\infty}\;J_{\mathsf{c},j}\bigl(\Hat{X}^{n}(0), Z^{n}\bigr)
&\;\le\; \frac{1}{2}(\updelta_j+\updelta^{\epsilon}_j) \,, \quad j =1,2\,.
\end{align*}
This completes the proof of part (ii).

The proof of the upper bound for the fairness problem
is analogous to part (ii).
By Theorem~5.7 and Remark~5.1 in \cite{AP15},
for any $\epsilon>0$,
there exists a continuous
precise control $v_{\epsilon} \in \Ussm$ for (P3$'$)
satisfying
\begin{equation}\label{PT5.2G}
\uppi_{v_{\epsilon}}(r_{\mathsf{o}})\;\le\; \varrho^*_{\mathsf{f}} +\epsilon\,,
\qquad\text{and}\quad
\uppi_{v_{\epsilon}}(r_{1})\;=\;\uptheta\,\uppi_{v_{\epsilon}}(r_{2})\,.
\end{equation}
Since $\{\uppi_{v_{\epsilon}}\,,\;\epsilon\in(0,1)\}$ is tight,
and $(e\cdot x)^-$ is strictly positive on an open subset of $B_1$, it
follows by the Harnack inequality for the density of the invariant probability
measure of the diffusion that
\begin{equation}\label{PT5.2H} % used
\inf_{\epsilon\in(0,1)}\; \uppi_{v_{\epsilon}}(r_{2})>0\,.
\end{equation}
Arguing as in part (ii), we obtain
\begin{equation}\label{PT5.2I}
\begin{split}
\limsup_{n\to\infty}\;J_{\mathsf{o}}(\Hat{X}^n(0),Z^n) &\;\le\;
\varrho^{*}_{\mathsf{f}} + \epsilon\,,\\[5pt]
\lim_{n\to\infty}\;
J_{\mathsf{c},j}\bigl(\Hat{X}^{n}(0), Z^{n}\bigr)
&\;=\; \uppi_{v_{\epsilon}}(r_{j})\,,\qquad j \;=\;1,2\,.
\end{split}
\end{equation}
The result then follows by \eqref{PT5.2G}--\eqref{PT5.2I}, thus
completing the proof.
\end{proof}

\smallskip
%%%%%%%%%%%%%%%%%%%%%%%%%%%%%%%%%%%%%%%%%%%%%%%%%%%%%%%%%%%%%%%%%%%%%
\section{Conclusion}

We have proved asymptotic optimality for the N-network in the Halfin-Whitt regime.
The analysis results in a good understanding of the stability of
the diffusion-scaled state processes under certain scheduling policies and
the convergence properties of the associated mean empirical measures. 
The state-dependent priority scheduling policy constructed
not only gives us a better understanding of the N-network,
but also plays a key role in proving the upper bound.
In addition we have identified some important properties
of the diffusion-scaled state processes that concern
existence of moments,
and the convergence of the mean empirical measures.
The methodology we followed should help to establish
asymptotic optimality for more general
multiclass multi-pool networks in the Halfin-Whitt regime.
If this is done, it will nicely complement the results on ergodic control
of the limiting controlled diffusion in \citet{AP15}. 
%For more general networks, the main difficulty lies in discovering simple
%(state-dependent) scheduling policies under which the diffusion-scaled state process
%is geometrically stable.

\smallskip
%%%%%%%%%%%%%%%%%%%%%%%%%%%%%%%%%%%%%%%%%%%%%%%%%%%%%%%%%%%%%%%%%%%%%
\appendix
\section{Proofs of Proposition~\ref{P3.1}
and Lemma~\ref{L9.1}}\label{App1}

In these proofs we use the fact that
the quantities
\begin{equation*}
\lambda_1^{n} - \mu_{11}^{n} N_1^n - \mu_{12}^n N_{12}^n\,,~
nx_1^* - N_1^n - N_{22}^n\,,~
\lambda_2^{n} -\mu_{22}^{n}  N_{22}^n\,,~
n x_2^*  -  N_{22}^n\,,~\text{and}~ \lambda_2^{n}  - \mu_{22}^{n} n x_2^*\,, 
\end{equation*}
are in $\order(\sqrt n)$.
This is straightforward to verify using Assumption~\ref{as-para}.

\begin{proof}[Proof of Proposition~\ref{P3.1}]
Simplifying the notation in Definition~\ref{D3.1}
we let $z^n=\Check{z}^n$, and analogously for $\Check{y}^n$ and
$\Check{q}^n$.  Fix $k > 2$.

Under the scheduling policy in Definition~\ref{D3.1}, 
the resulting process $X^{n}$ is Markov with generator
\begin{multline} \label{PP3.1a} 
\cL_n^{\Check{z}^{n}} f(x) \;\df\; \sum_{i=1}^2 \lambda^{n}_i \bigl(f(x+e_i) - f(x)\bigr)
+ ( \mu_{11}^{n} z^{n}_{11} + \mu_{12}^{n} z^{n}_{12}) \bigl(f(x-e_1)
- f(x)\bigr)  \\
+  \mu_{22}^{n} z^{n}_{22} \bigl(f(x-e_2) - f(x)\bigr)
+ \sum_{i=1}^2 \gamma_i^{n} q^{n}_i
\bigl(f(x-e_i)- f(x)\bigr)\,, \qquad   x \in \ZZ^{2}_{+}\,,
\end{multline}

Recall the definition of $\Hat{x}$ in \eqref{Td-xn}.
Define 
\begin{equation*}
f_n(x)\;\df\;  \abs{x_1 - n x_{1}^{*}}^k
+\beta \abs{x_2 - n x_{2}^{*}}^k
\;=\;  n^{\nicefrac{k}{2}} \bigl(\abs{\Hat{x}_1}^k+\beta \abs{\Hat{x}_2}^k\bigr) \,, 
\end{equation*}
for some positive constant $\beta$, to be determined later. 
If we show that 
\begin{equation} \label{PP3.1b}
\cL_n^{\Check{z}^{n}} f_n(x) \;\le\; C_1 n^{\nicefrac{k}{2}} - C_2 f_n(x)\,,
\quad x \in \ZZ^{2}_{+}\,,
\end{equation}
for some positive constants $C_1$ and $C_2$, and for all $n\ge n_0$, then
by using \eqref{E-gen2} we obtain  \eqref{Lya-GE}.

Given \eqref{PP3.1b}, we easily obtain that 
\begin{align*}
\Exp \left[f_n(X^{n}(T))\right] - f_n(X^{n}(0))  &\;=\;
\Exp \left[ \int_{0}^{T} \cL_n f_n(X^{n}(s))\,\D{s} \right] \\[5pt]
&\;\le\; C_1 n^{\nicefrac{k}{2}} T
- C_2 \Exp \left[ \int_{0}^{T} f_n(X^{n}(s))\,\D{s} \right] \,,
\end{align*}
which implies that 
\begin{equation*}
\frac{1}{T}\;\Exp \left[ \int_{0}^{T}
\Lyap_{k,\beta}\bigl(\Hat{X}^{n}(s)\bigr)\,\D{s}\right] \;\le\;
C_1 + \frac{1}{T} \Lyap_{k,\beta}\bigl(\Hat{X}^{n}(0)\bigr)
- \frac{1}{T}\;\Exp \bigl[ \Lyap_{k,\beta}\bigl(\Hat{X}^{n}(T)\bigr) \bigr] \,. 
\end{equation*}
By letting $T\to \infty$,
this implies that \eqref{upper-p1} holds.

We now focus on proving \eqref{PP3.1b}. Note that 
\begin{equation*}
(a\pm 1)^k - a^k\;=\; k a^{k-1} + \order(a^{k-2})\,, \quad a \in \RR \,.  
\end{equation*}
Recall $\Tilde{x}$ in \eqref{Td-xn}.  
Then by \eqref{PP3.1a}, we have
\begin{align*} 
\cL_n^{\Check{z}^{n}} f_n(x) &\;=\;   \lambda^{n}_1
\bigl(k \Tilde{x}_1 \,\abs{\Tilde{x}_1}^{k-2}
+ \order \bigl(\abs{\Tilde{x}_1}^{k-2}\bigr)\bigr)
+\beta  \lambda^{n}_2
\bigl(k \Tilde{x}_2 \,\abs{\Tilde{x}_2}^{k-2}
+ \order \bigl(\abs{\Tilde{x}_2}^{k-2}\bigr)\bigr)
\nonumber \\[5pt]
&\mspace{50mu} +
(\mu_{11}^{n} z^{n}_{11} + \mu_{12}^{n} z^{n}_{12}) \,\bigl( -k \Tilde{x}_1 \,
\abs{\Tilde{x}_1}^{k-2} 
+ \order \bigl(\abs{\Tilde{x}_1}^{k-2}\bigr) \bigr) \nonumber\\[5pt]
&\mspace{100mu}
+  \beta \mu_{22}^{n}z^{n}_{22}\,
\bigl( -k \Tilde{x}_2 \,\abs{\Tilde{x}_2}^{k-2} 
+ \order \bigl(\abs{\Tilde{x}_2}^{k-2}\bigr) \bigr) \nonumber\\[5pt]
&\mspace{150mu}
+\gamma_1^{n} q^{n}_1\,
\bigl( - k \Tilde{x}_1 \,\abs{\Tilde{x}_1}^{k-2}
+ \order\bigl(\abs{\Tilde{x}_1}^{k-2}\bigr)\bigr)\nonumber\\[5pt]
&\mspace{200mu}
+\beta \gamma_2^{n} q^{n}_2\,
\bigl( - k \Tilde{x}_2 \,\abs{\Tilde{x}_2}^{k-2}
+ \order\bigl(\abs{\Tilde{x}_2}^{k-2}\bigr)\bigr)\,. 
\end{align*}

Let
\begin{multline} \label{E-Fn1}
F_n^{(1)}(x) \;\df\;  \bigl(\lambda_1^{n}
+ \gamma^{n}_1 q^{n}_1 \bigr) \order\bigl(\abs{\Tilde{x}_1}^{k-2} \bigr)
+\beta\bigl(\lambda_2^{n}
+ \gamma^{n}_2 q^{n}_2 \bigr) \order\bigl(\abs{\Tilde{x}_2}^{k-2} \bigr) \\
 +  \bigl(\mu_{11}^{n} z^{n}_{11}  + \mu_{12}^{n} z^{n}_{12} \bigr)
\order\bigl(\abs{\Tilde{x}_1}^{k-2} \bigr) 
+  \beta\mu_{22}^{n}z^{n}_{22} 
\order\bigl(\abs{\Tilde{x}_2}^{k-2} \bigr) \,, 
\end{multline}
and
\begin{multline} \label{E-Fn2}
F_n^{(2)}(x) \;\df\;    k \Tilde{x}_1 \,
\abs{\Tilde{x}_1}^{k-2}\bigl(\lambda_1^{n} - \gamma^{n}_1 q^{n}_1 \bigr)
+ \beta k \Tilde{x}_2 \,
\abs{\Tilde{x}_2}^{k-2}\bigl(\lambda_2^{n} - \gamma^{n}_2 q^{n}_2 \bigr)\\
- k \Tilde{x}_1 \,\abs{\Tilde{x}_1}^{k-2}
( \mu_{11}^{n} z^{n}_{11} + \mu_{12}^{n} z^{n}_{12}) 
-  \beta k \Tilde{x}_2 \,
\abs{\Tilde{x}_2}^{k-2} \mu_{22}^{n} z^{n}_{22}  \,. 
\end{multline}
Then 
$$
\cL_n^{\Check{z}^{n}} f_n(x)  \;=\; F_n^{(1)}(x) + F_n^{(2)}(x)\,. 
$$

We first study $F_n^{(1)}(x)$. 
It is easy to
observe that for each $i =1,2$ and $j =1,2$, 
\begin{equation} \label{PP3.1c}
z^{n}_{ij}  \;\le\; x_i\,,
\quad \text{and} \quad  q^{n}_i \;\le\; x_i\,.
\end{equation}
Thus, we obtain
\begin{align} \label{PP3.1d}
F_n^{(1)}(x) & \le\;  \bigl(\lambda_1^{n}
+ \gamma^{n}_1 x_1 \bigr) \order\bigl(\abs{\Tilde{x}_1}^{k-2} \bigr)
+\beta \bigl(\lambda_2^{n}
+ \gamma^{n}_2 x_2 \bigr) \order\bigl(\abs{\Tilde{x}_2}^{k-2} \bigr)
\nonumber \\
& \mspace{50mu} +  \bigl(\mu_{11}^{n}   + \mu_{12}^{n}  \bigr)x_{1}
\order\bigl(\abs{\Tilde{x}_1}^{k-2} \bigr)
+  \beta \mu_{22}^{n}x_{2} 
\order\bigl(\abs{\Tilde{x}_2}^{k-2} \bigr)  \nonumber \\[5pt]
&= \;  \bigl(\lambda_1^{n} + \gamma^{n}_1 (n x_{1}^{*} + \Tilde{x}_1) \bigr)
\order\bigl(\abs{\Tilde{x}_1}^{k-2} \bigr)
+ \beta \bigl(\lambda_2^{n} + \gamma^{n}_2 (n x_{2}^{*} + \Tilde{x}_2) \bigr)
\order\bigl(\abs{\Tilde{x}_2}^{k-2} \bigr)  \nonumber \\
& \mspace{50mu}
+\bigl(\mu_{11}^{n} + \mu_{12}^{n}  \bigr)(n x_{1}^{*} + \Tilde{x}_1)
\order\bigl(\abs{\Tilde{x}_1}^{k-2} \bigr)
+  \beta \mu_{22}^{n} (n x_{2}^{*} + \Tilde{x}_2)
\order\bigl(\abs{\Tilde{x}_2}^{k-2} \bigr)  \nonumber \\[5pt]
&\le\;  \sum_{i=1}^2  \Bigl(\order(n) \order\bigl(\abs{\Tilde{x}_i}^{k-2}\bigr)
+ \order\bigl(\abs{\Tilde{x}_i}^{k-1} \bigr) \Bigr)\,,
\end{align}
where the last inequality follows from Assumption~\ref{as-para}.

We next focus on $F_n^{(2)}(x)$. 
We consider four cases:

\smallskip\noindent
\emph{Case~1\/}: $x_1\ge N_1^n+ N_{12}^n$ and  $x_2\ge N_{22}^n$.
Then
\begin{equation*}
z^n_{11} =  N_1^n\,,  \quad z^n_{12} =N_{12}^n\,,
\quad z^n_{22} = N_{22}^n\,, \quad q_1^n=x_1-N_1^n-N_{12}^n\,,\quad
 q_2^n=x_2 - N_{22}^n\,.
\end{equation*}

We obtain
\begin{align}\label{EC1}
F_n^{(2)}(x) 
& \;=\;k  \Tilde{x}_1 \,\abs{\Tilde{x}_1}^{k-2}
\bigl[\lambda_1^{n} - \mu_{11}^{n}  N_1^n - \mu_{12}^{n} N_{12}^n
- \gamma^{n}_1 (nx_1^*-N_1^n- N_{12}^n) \bigr]  \nonumber\\[3pt]
& \mspace{60mu} + 
\beta k  \Tilde{x}_2 \,\abs{\Tilde{x}_2}^{k-2}
\bigl[\lambda_2^{n}  -  \mu_{22}^{n}  N_{22}^n
- \gamma^{n}_2 (  nx^*_2 -   N_{22}^n)\bigr]    
- k  \gamma^{n}_1 \abs{\Tilde{x}_1}^{k} 
- \beta k  \gamma^{n}_2 \abs{\Tilde{x}_2}^{k} \nonumber \\[5pt]
& \;=\; \order(\sqrt n)\,
\bigl(\abs{\Tilde{x}_1}^{k-1}+\beta\abs{\Tilde{x}_2}^{k-1}\bigr)
- k \gamma^{n}_1 \abs{\Tilde{x}_1}^{k} 
- \beta k  \gamma^{n}_2 \abs{\Tilde{x}_2}^{k}\,.
\end{align}

\smallskip\noindent
\emph{Case~2\/}: $x_1< N_1^n+ N_{12}^n$ and  $x_2< N_{22}^n$.
Consider two subcases:

\smallskip
\emph{Case~2.1\/}: $x_1 >N_1^n$.
Then
\begin{equation*}
z^n_{11} =  N_1^n\,,  \quad z^n_{12} =x_1-N_1^n\,,
\quad z^n_{22} = x_2\,, \quad q_1^n=q_2^n=0\,.
\end{equation*}

We have 
\begin{align}\label{EC2.1}
F_n^{(2)}(x) 
& \;=\;k  \Tilde{x}_1 \,\abs{\Tilde{x}_1}^{k-2}
\bigl[\lambda_1^{n} - \mu_{11}^{n} N_1^n
- \mu_{12}^{n} (nx_1^*-N_1^n) \bigr]
\nonumber\\[3pt]
& \mspace{60mu} + 
   \beta k  \Tilde{x}_2 \,\abs{\Tilde{x}_2}^{k-2}
\bigl[\lambda_2^{n}  -  \mu_{22}^{n} n x^*_2  \bigr]
- k  \mu^{n}_{12} \abs{\Tilde{x}_1}^{k} 
- \beta k  \mu^{n}_{22} \abs{\Tilde{x}_2}^{k}\nonumber \\[5pt]
& \;=\; \order(\sqrt n)\,
\bigl(\abs{\Tilde{x}_1}^{k-1}+\beta\abs{\Tilde{x}_2}^{k-1}\bigr)
- k  \mu^{n}_{12} \abs{\Tilde{x}_1}^{k} 
- \beta k  \mu^{n}_{22} \abs{\Tilde{x}_2}^{k}\,.
\end{align}

\smallskip
\emph{Case~2.2\/}:  $x_1 < N_1^n$. 
Then
\begin{equation*}
z^n_{11} =  x_1\,,  \quad z^n_{12} =0\,,
\quad z^n_{22} = x_2\,, \quad q_1^n=q_2^n=0\,.
\end{equation*}

We obtain
\begin{align}\label{EC2.2-aux}
F_n^{(2)}(x) &\;=\;
k  \Tilde{x}_1 \,\abs{\Tilde{x}_1}^{k-2}
\bigl[\lambda_1^{n} - \mu_{11}^{n}x_1\bigr]
 + \beta k  \Tilde{x}_2 \,\abs{\Tilde{x}_2}^{k-2}
\bigl[\lambda_2^{n} - \mu_{22}^{n} (\Tilde{x}_2+n x_2^*) \bigr] 
\nonumber \\[5pt]
& \;\le\;k  \Tilde{x}_1 \,\abs{\Tilde{x}_1}^{k-2}
\bigl[\lambda_1^{n}- \mu_{11}^{n} N_1^n- \mu_{12}^{n} N_{12}^n\bigr]
+\beta k  \Tilde{x}_2 \,\abs{\Tilde{x}_2}^{k-2}
\bigl[\lambda_2^{n}  - \mu_{22}^{n} n x_2^*\bigr]
\nonumber \\[3pt]
& \mspace{60mu}
+k \mu_{12}^{n}  N_{12}^n\,\Tilde{x}_1 \,\abs{\Tilde{x}_1}^{k-2}
-\beta k \mu_{22}^{n}\,\abs{\Tilde{x}_2}^{k}\,.  
\end{align}
Since $x_1\le N_1^n$ we have
\begin{align*}
\mu_{12}^{n}N_2^n\,\Tilde{x}_1 &\;\le\; -\frac{\mu_{12}^{n}N_2^n}{N_1^n}
\abs{\Tilde{x}_1}^2\\[3pt]
&\;=\; -\frac{\mu_{12}\nu_2}{\nu_1}\abs{\Tilde{x}_1}^2
+\order(\sqrt n) \abs{\Tilde{x}_1}\,.
\end{align*}
Thus, \eqref{EC2.2-aux} takes the form
\begin{equation}\label{EC2.2}
F_n^{(2)}(x) \;\le\; \order(\sqrt n)\,
\bigl(\abs{\Tilde{x}_1}^{k-1}+\beta\abs{\Tilde{x}_2}^{k-1}\bigr)
-k  \xi^*_{12} \frac{\mu_{12}\nu_2}{\nu_1}\,\abs{\Tilde{x}_1}^{k}
-\beta k \mu_{22}^{n}\,\abs{\Tilde{x}_2}^{k}\,.
\end{equation}

\smallskip\noindent
\emph{Case~3\/}: $x_1\ge N_1^n+  N_{12}^n$ and  $x_2< N_{22}^n$.
We distinguish two subcases.

\smallskip
\emph{Case~3.1\/}: $x_1+x_2\ge  N_1^n+N_2^n$. 

Then
\begin{equation*}
z^n_{11} =  N_1^n\,,  \quad z^n_{12} = N_2^n - x_2\,,
\quad z^n_{22} = x_2\,, \quad q_1^n= x_1+x_2-N_1^n-N_2^n\,,
\quad q_2^n=0\,.
\end{equation*}

We have
\begin{align}\label{EC3.1}
F_n^{(2)}(x) &\;=\;
k  \Tilde{x}_1 \,\abs{\Tilde{x}_1}^{k-2}
\bigl[\lambda_1^{n}  
- \mu_{11}^{n} N_1^n - \mu_{12}^{n} (N_2^n -n x_2^* -\Tilde{x}_2)
- \gamma^{n}_1 (x_1+x_2-N_1^n-N_2^n)\bigr]
\nonumber \\[3pt]
& \mspace{60mu} + \beta k  \Tilde{x}_2 \,\abs{\Tilde{x}_2}^{k-2}
\bigl[\lambda_2^{n}  - \mu_{22}^{n} (n x_2^* +\Tilde{x}_2) \bigr]
\nonumber \\[5pt]
& \;=\;k  \Tilde{x}_1 \,\abs{\Tilde{x}_1}^{k-2}
\bigl[\lambda_1^{n}- \mu_{11}^{n} N_1^n- \mu_{12}^{n}  N_{12}^n
+ \mu_{12}^{n} (n x_2^*+  N_{22}^n)\bigr]
\nonumber \\[3pt]
& \mspace{60mu}
- k  \Tilde{x}_1 \,\abs{\Tilde{x}_1}^{k-2}
\gamma^{n}_1 (nx_1^*+nx_2^*-N_1^n-N_2^n)
+\beta k  \Tilde{x}_2 \,\abs{\Tilde{x}_2}^{k-2}
\bigl[\lambda_2^{n}  - \mu_{22}^{n} n x_2^*\bigr]
\nonumber \\[3pt]
& \mspace{120mu}
+k  (\mu_{12}^{n}-\gamma^{n}_1)
\Tilde{x}_1\,\Tilde{x}_2 \,\abs{\Tilde{x}_1}^{k-2}
-k \gamma^{n}_1\,\abs{\Tilde{x}_1}^{k}
-\beta k \mu_{22}^{n}\,\abs{\Tilde{x}_2}^{k}
\nonumber \\[5pt]
&\;=\;
\order(\sqrt n)\,
\bigl(\abs{\Tilde{x}_1}^{k-1}+\beta\abs{\Tilde{x}_2}^{k-1}\bigr)
+k  (\mu_{12}^{n}-\gamma^{n}_1)
\Tilde{x}_1\,\Tilde{x}_2 \,\abs{\Tilde{x}_1}^{k-2}  \nonumber\\[3pt]
& \mspace{320mu}
 -k \gamma^{n}_1\,\abs{\Tilde{x}_1}^{k} 
-\beta k \mu_{22}^{n}\,\abs{\Tilde{x}_2}^{k}\,. 
\end{align}

\smallskip
\emph{Case~3.2\/}: $x_1+x_2<  N_1^n+N_2^n$.
Then
\begin{equation*}
z^n_{11} =  N_1^n\,,
\quad z^n_{12} = \Tilde{x}_1+ n x_1^*-N_1^n\,,
\quad z^n_{22} = x_2\,, \quad q_1^n= q_2^n=0\,.
\end{equation*}

We have
\begin{align}\label{EC3.2}
F_n^{(2)}(x) &\;=\;
k  \Tilde{x}_1 \,\abs{\Tilde{x}_1}^{k-2}
\bigl[\lambda_1^{n} - \mu_{11}^{n} N_1^n
- \mu_{12}^{n} (\Tilde{x}_1+ n x_1^*-N_1^n)\bigr]
\nonumber \\[3pt]
& \mspace{60mu} + \beta k  \Tilde{x}_2 \,\abs{\Tilde{x}_2}^{k-2}
\bigl[\lambda_2^{n} - \mu_{22}^{n} (n x_2^* +\Tilde{x}_2) \bigr] 
\nonumber \\[5pt]
& \;=\;k  \Tilde{x}_1 \,\abs{\Tilde{x}_1}^{k-2}
\bigl[\lambda_1^{n}- \mu_{11}^{n} N_1^n- \mu_{12}^{n}  N_{12}^n
- \mu_{12}^{n} (n x_1^*-N_1^n- N_{12}^n)\bigr]
\nonumber \\[3pt]
& \mspace{60mu}
+\beta k  \Tilde{x}_2 \,\abs{\Tilde{x}_2}^{k-2}
\bigl[\lambda_2^{n}  - \mu_{22}^{n} n x_2^*\bigr]
-k \mu_{12}^{n}\,\abs{\Tilde{x}_1}^{k}
-\beta k \mu_{22}^{n}\,\abs{\Tilde{x}_2}^{k} \nonumber \\
& \;=\;  \order(\sqrt n)\,
\bigl(\abs{\Tilde{x}_1}^{k-1}+\beta\abs{\Tilde{x}_2}^{k-1}\bigr)
-k \mu_{12}^{n}\,\abs{\Tilde{x}_1}^{k}
-\beta k \mu_{22}^{n}\,\abs{\Tilde{x}_2}^{k} \,. 
\end{align}

\smallskip\noindent
\emph{Case~4\/}: $x_1< N_1^n+  N_{12}^n$ and  $x_2\ge N_{22}^n$.
Here we distinguish four subcases.

\smallskip
\emph{Case~4.1\/}:
$x_1\le N_1^n$ and $x_2\le N_2^n$.
Using the argument used in Case~2.2,
we obtain the same estimate as \eqref{EC2.2}.

\smallskip
\emph{Case~4.2\/}:
$x_1\le N_1^n$ and $x_2> N_2^n$.
Then
\begin{equation*}
z^n_{11} =  x_1\,,  \quad z^n_{12} =0\,,
\quad z^n_{22} = N_2^n\,, \quad q_1^n=0\,,\quad q_2^n=x_2-N_2^n\,.
\end{equation*}
We use the inequality
\begin{equation*}
\mu_{22}^n  N_{12}^n
+ \gamma_2^n (\Tilde{x}_2+ n x_{2}^* - N_{2}^n)
\;\ge\; (\mu_{22}^n\wedge \gamma_2^n) \, \Tilde{x}_2+\order(\sqrt n)\,,
\qquad x_2 > N_2^n
\end{equation*}
to write
\begin{align*}
\lambda_2^{n} - \mu_{22}^{n}  N_2^n - \gamma^n_2 (x_2 - N_2^n)
&\;=\;\lambda_2^{n} - \mu_{22}^{n}  N_2^n
- \gamma^n_2 (\Tilde{x}_2+n x^*_2 - N_2^n)\\[3pt]
&\;\le\; \lambda_2^{n} - \mu_{22}^{n}  N_{22}^n
+(\mu_{22}^n\wedge \gamma_2^n) \, \Tilde{x}_2+\order(\sqrt n)\,.
\end{align*}
Therefore, as in Case~2.2,  we obtain
\begin{align}\label{EC4.2} % used
F_n^{(2)}(x) &\;\le\; k  \order(\sqrt n)\,\abs{\Tilde{x}_1}^{k-1}
+\beta k  \Tilde{x}_2 \,\abs{\Tilde{x}_2}^{k-2}
\bigl[\lambda_2^{n} - \mu_{22}^{n}  N_{22}^n
+\order(\sqrt n)\bigr]\nonumber\\[3pt]
&\mspace{60mu}
-k  \xi^*_{12} \frac{\mu_{12}\nu_2}{\nu_1}\,\abs{\Tilde{x}_1}^{k}
-\beta k  (\mu_{22}^n\wedge \gamma_2^n)\,\abs{\Tilde{x}_2}^{k}
\nonumber\\[5pt]
&\;\le\; \order(\sqrt n)\,
\bigl(\abs{\Tilde{x}_1}^{k-1}+\beta\abs{\Tilde{x}_2}^{k-1}\bigr)
-k   \xi^*_{12} \frac{\mu_{12}\nu_2}{\nu_1}\,\abs{\Tilde{x}_1}^{k}
-\beta k  (\mu_{22}^n\wedge \gamma_2^n)\,\abs{\Tilde{x}_2}^{k}\,.
\end{align}

\smallskip
\emph{Case~4.3\/}:
$x_1> N_1^n$ and $x_1+x_2<  N_1^n+N_2^n$.
Then
\begin{equation*}
z^n_{11} =  N_1^n\,,  \quad z^n_{12} =x_1-N_1^n\,,
\quad z^n_{22} = x_2\,, \quad q_1^n=0\,,\quad q_2^n=0\,.
\end{equation*}

We obtain
\begin{align}\label{EC4.3} % used
F_n^{(2)}(x) &\;=\;
k  \Tilde{x}_1 \,\abs{\Tilde{x}_1}^{k-2}
\bigl[\lambda_1^{n} - \mu_{11}^{n}N_1^n - \mu_{12}^n(\Tilde{x}_1
+nx_1^*-N_1^n)\bigr]\nonumber \\[3pt]
&\mspace{60mu} + \beta k  \Tilde{x}_2 \,\abs{\Tilde{x}_2}^{k-2}
\bigl[\lambda_2^{n} - \mu_{22}^{n} (\Tilde{x}_2+n x_2^*) \bigr] 
\nonumber \\[5pt]
&\;=\;
k  \Tilde{x}_1 \,\abs{\Tilde{x}_1}^{k-2}
\bigl[\lambda_1^{n} - \mu_{11}^{n}N_1^n - \mu_{12}^n N_{12}^n
-\mu_{12}^n(nx_1^*-N_1^n- N_{12}^n)\bigr]\nonumber \\[3pt]
&\mspace{60mu} + \beta k  \Tilde{x}_2 \,\abs{\Tilde{x}_2}^{k-2}
\bigl[\lambda_2^{n} -\mu_{22}^{n} N_{22}^n
- \mu_{22}^{n} (n x_2^*- N_{22}^n) \bigr]\nonumber \\[3pt]
&\mspace{120mu}
-k \mu_{12}^{n}\abs{\Tilde{x}_1}^{k}
-\beta k \mu_{22}^{n}\abs{\Tilde{x}_2}^{k}
\nonumber \\[5pt]
& \;=\;\order(\sqrt n)\,
\bigl(\abs{\Tilde{x}_1}^{k-1}+\beta\abs{\Tilde{x}_2}^{k-1}\bigr)
-k \mu_{12}^{n}\abs{\Tilde{x}_1}^{k}
-\beta k \mu_{22}^{n}\abs{\Tilde{x}_2}^{k}\,.  
\end{align}

\smallskip
\emph{Case~4.4\/}.
$x_1> N_1^n$ and $x_1+x_2\ge N_1^n+N_2^n$.
Then
\begin{equation*}
Z^n_{11} =  N_1^n\,,  \quad Z^n_{12} =x_1-N_1^n\,,
\quad Z^n_{22} = N_2^n+N_1^n-x_1\,, \quad Q_1^n=0\,,\quad Q_2^n=
x_1+x_2-N_1^n-N_2^n\,.
\end{equation*}
Therefore, we obtain
\begin{align}\label{EC4.4}
F_n^{(2)}(x) &\;=\;
q \beta_1 \Tilde{x}_1 \,\abs{\Tilde{x}_1}^{q-2}
\bigl[\lambda_1^{n} - \mu_{11}^{n}N_1^n - \mu_{12}^n(\Tilde{x}_1
+nx_1^*-N_1^n)\bigr]\nonumber \\[3pt]
&\qquad + q \beta_2 \Tilde{x}_2 \,\abs{\Tilde{x}_2}^{q-2}
\bigl[\lambda_2^{n} - \mu_{22}^{n} (N_2^n+N_1^n-x_1)
-\gamma_2^n \bigl(x_2 - (N_2^n+N_1^n-x_1) \bigr)\bigr] 
\nonumber \\[5pt]
& \;\le \;  q \beta_1 \Tilde{x}_1 \,\abs{\Tilde{x}_1}^{q-2}
\bigl[\lambda_1^{n} - \mu_{11}^{n}N_1^n - \mu_{12}^n N_{12}^n
-\mu_{12}^n(\Tilde{x}_1 + nx_1^*-N_1^n-N_{12}^n)\bigr]\nonumber \\[3pt]
&\qquad \qquad \qquad  + q \beta_2 \Tilde{x}_2 \,\abs{\Tilde{x}_2}^{q-2}
\bigl[\lambda_2^{n} - (\mu_{22}^{n} \wedge \gamma_2^n) (\Tilde{x}_2 + nx_2^*)  \bigr] 
\nonumber \\[5pt]
& \;\le\;\order(\sqrt n)\,
\bigl(\beta_1\abs{\Tilde{x}_1}^{q-1}+\beta_2\abs{\Tilde{x}_2}^{q-1}\bigr)
-q \beta_1\mu_{12}^{n}\abs{\Tilde{x}_1}^{q}
-q \beta_2 (\mu_{22}^{n} \wedge \gamma_2^n)  \abs{\Tilde{x}_2}^{q}\,,
\end{align}
where the first inequality follows by observing that 
\begin{equation*}
\mu_{22}^{n} (N_2^n+N_1^n - x_1) + \gamma_2^n \bigl(x_2 - (N_2^n+N_1^n-x_1) \bigr)
\;\ge\; (\mu_{22}^{n} \wedge \gamma_2^n) x_2\,,
\end{equation*}
since $x_2 \ge N_2^n+N_1^n-x_1$ and $N_2^n+N_1^n - x_1> N_{22}^n+N_1^n - x_1 > 0$.  

\smallskip
By Young's inequality, we have 
\begin{equation*}
\abs{\Tilde{x}_1}^{k-1} \abs{\Tilde{x}_2} \;\le\;
\epsilon\abs{\Tilde{x}_1}^k
+ \frac{1}{\epsilon^{k-1}} \abs{\Tilde{x}_2}^k \,,
\end{equation*}
\begin{equation*}
\abs{\Tilde{x}_2}^{k-1} \abs{\Tilde{x}_1} \;\le\;
\epsilon\abs{\Tilde{x}_1}^k +
\frac{1}{\epsilon^{\frac{1}{k-1}}} \abs{\Tilde{x}_2}^k 
\end{equation*}
for any $\epsilon>0$.
Using this in \eqref{EC3.1} in combination with
\eqref{EC1}--\eqref{EC2.1},
\eqref{EC2.2}
and \eqref{EC3.2}--\eqref{EC4.4},
we can choose the constant $\beta$ 
properly so that  
\begin{align} \label{st-p-p6}
\cL_n^{\Check{z}^{n}} f_n(x) &\;\le\; \sum_{i=1}^2  \biggl( \order(n)
\order\bigl(\abs{\Tilde{x}_i}^{k-2}\bigr)
+ \order(\sqrt{n}) \order\bigl(\abs{\Tilde{x}_i}^{k-1}\bigr)  \biggr) 
-  \Tilde{C}_2\sum_{i=1}^2  \abs{\Tilde{x}_i}^{k} \,,
\end{align}
for some positive constant $\Tilde{C}_2$.
Now applying Young's inequality again to the first two terms on the right hand
side of \eqref{st-p-p6}, we obtain
\begin{align*}
\order(\sqrt{n}) \order(\abs{\Tilde{x}_i}^{k-1}) &\;\le\;
\epsilon\bigl( \order(\abs{\Tilde{x}_i}^{k-1}) \bigr)^{\nicefrac{k}{(k-1)}}
+ \epsilon^{1-k} \bigl(\order(\sqrt{n})\bigr)^k \,,\\[5pt]
\order(n) \order(\abs{\Tilde{x}_i}^{k-2}) &\;\le\;
\epsilon\bigl( \order(\abs{\Tilde{x}_i}^{k-2}) \bigr)^{\nicefrac{k}{(k-2)}}
+ \epsilon^{1-\nicefrac{k}{2}} \bigl(\order(n)\bigr)^{\nicefrac{k}{2}}
\end{align*}
for any $\epsilon>0$.
This shows that can choose $\beta$, $C_1$ and $C_2$ appropriately to obtain the claim
in \eqref{PP3.1b}.

Recall $\Hat{x}^n$ in \eqref{Td-xn} and let $\Hat{q}^n_i := q^n_i/\sqrt{n}$
for $i=1,2$. 
Concerning the claim in \eqref{upper-p1} with $\Hat{X}^n$ replaced by $\Hat{Q}^n$
we observe that 
in Case~1, $\Hat{q}^n_1 = \Hat{x}_1^n + \order(1)$, and
$\Hat{q}^n_2= \Hat{x}_2^n + \order(1)$,
in Case~3.1, $\Hat{q}^n_1 =  \Hat{x}_1^n +  \Hat{x}_2^n +\order(1)$,
and $\Hat{q}^n_2 =0$,
in Case~4.2,  $\Hat{q}^n_1=0$, and  $\Hat{q}^n_2  \le  \Hat{x}_2^n + \order(1)$,
in Case~4.4, $\Hat{q}^n_1=0$, and 
$\Hat{q}^n_2  =   \Hat{x}_1^n + \Hat{x}_2^n + \order(1)$,
and in all the other cases,  $\Hat{q}^n_1=\Hat{q}^n_2=0$. 
The same claim for $\Hat{Y}^n$ then follows from the balance equation \eqref{E-sumbal}.
The proof of the proposition is complete. 
\end{proof}

\begin{proof}[Proof of Lemma~\ref{L9.1}] 
We need to show \eqref{PP3.1b} holds for $\cL_n f_n(x)$ under the
scheduling policy $z^{n}[v]$ in Definition~\ref{D7.1}. 
We can write  $\cL_n f_n(x)= F^{(1)}_n (x)+ F^{(2)}_n(x)$ with $F^{(1)}_n (x)$ 
and $F^{(2)}_n(x)$ given by \eqref{E-Fn1} and \eqref{E-Fn2} respectively.
We obtain \eqref{PP3.1d} for  $F^{(1)}_n (x)$ since \eqref{PP3.1c}
also holds under the policy $z^{n}[v]$. 
For $F^{(2)}_n (x)$, by \eqref{E-Fn2} and Definition \ref{D7.1}, since
the control $v$ satisfies \eqref{EL6.4a} and $x\in\sX^n$
(JWC being satisfied), we easily obtain
\begin{equation*}
F^{(2)}_n(x) \;\le \; \order(\sqrt n)\,
\bigl(\abs{\Tilde{x}_1}^{k-1}+\beta\abs{\Tilde{x}_2}^{k-1}\bigr)
-  \Tilde{C}_3\sum_{i=1}^2  \abs{\Tilde{x}_i}^{k} \,,
\end{equation*}
for some positive constant $\Tilde{C}_3$. Thus, following the argument
in the proof of Proposition \ref{P3.1}, we obtain the claim in
\eqref{PP3.1b} and hence the result follows by scaling. 
\end{proof}

\smallskip
%%%%%%%%%%%%%%%%%%%%%%%%%%%%%%%%%%%%%%%%%%%%%%%%%%%%%%%%%%%%%%%%%%%%%
\section{Proof of Theorem~\ref{T4.3}}\label{App2}

Recall $\sH_{\mathsf{f}}(\uptheta)$ defined in \eqref{E-sH-f-theta}. 
As in Theorem~\ref{T4.2} there exists 
$\uplambda^{*}\in\RR$ such that
\begin{equation*}
\inf_{\uppi\,\in\,\sH_{\mathsf{f}}(\uptheta)}\;\uppi(r_{\mathsf{o}}) \;=\;
\inf_{\uppi\,\in\,\eom}\;\uppi(h_{\uptheta,\uplambda^{*}})
\;=\;\varrho^*_{\mathsf{f}}\,,
\end{equation*}
and the property in \eqref{E-apriori} implies that the infimum is attained
in some $\uppi^*\in\eom$.
Therefore, the conclusions
analogous to parts (a) and (b) of Theorem~\ref{T4.2} hold.
Part (e) is also standard.
It remains to derive the HJB equation and the characterization of optimality
corresponding to Theorem~\ref{T4.2}\,(c)--(d).
This is broken in a series of lemmas.

We need to introduce some notation.
We denote by $\tc_\delta$, $\delta>0$, the \emph{first exit time} of a process
from $B^c_\delta$, i.e.,
\begin{equation*}
\tc_\delta \;\df\;\inf\;\{t>0\,\colon\,X_{t}\not\in B^c_\delta\}\,.
\end{equation*}
We denote by $\Usm^\star$ the class of Markov controls $v$ satisfying
$\uppi_v(r_{\mathsf{o}})<\infty$, and by $\eom^\star$ the corresponding
class of ergodic occupation measures.

By  the method of proof of \eqref{E-apriori}
there exists inf-compact $\sV\in\Cc^2(\RR^2)$ and positive
constants $\kappa_1$ and $\kappa_2$ satisfying
\begin{equation}\label{ET4.3a}
\Lg^u\sV(x) \;\le\; \kappa_1 -\kappa_2 \abs{x}^m
+ r_{\mathsf{o}} (x,u)\qquad \forall (x,u)\in\RR^2\times\Act\,.
\end{equation}
Moreover, since $(e\cdot x)^-\in\sorder(\abs{x}^m)$, there exists a constant
$\kappa_0$ such that
\begin{equation}\label{ET4.3b}
(1+\theta)\lambda^* r_j(x,u)\;\le\; \kappa_0 + \frac{\kappa_2}{2}\, \abs{x}^m \qquad
\forall (x,u)\in\RR^2\times\Act\,,\quad j=1,2\,.
\end{equation}
For $\epsilon>0$ we define
\begin{equation*}
h^{\epsilon}(x,u) \;\df\; h_{\uptheta,\uplambda^*} (x,u)
+ \epsilon\kappa_2 \abs{x}^m\,.
\end{equation*}

\begin{lemma}\label{LA2.1}
The following hold:
\begin{align*}
\uppi (h_{\uptheta,\uplambda^{*}}) &\;\le\;
\kappa_0 + \tfrac{\kappa_1}{2} + \tfrac{3}{2}
\uppi (r_{\mathsf{o}})\qquad\forall\,\uppi\in\eom^\star\,,\\[5pt]
\uppi (r_{\mathsf{o}}) &\;\le\; \kappa_0 + \tfrac{\kappa_2}{2} +
\uppi( h_{\uptheta,\uplambda^*})\qquad \forall\, \uppi\in\eom\,,\\[5pt]
\uppi (h^\epsilon) &\;\le\;
\epsilon\bigl( \kappa_0 + \tfrac{\kappa_1}{2}+\tfrac{\kappa_2}{2}\bigr)
+(1+\epsilon) \uppi (h_{\uptheta,\uplambda^{*}})\,.
\end{align*}
\end{lemma}

\begin{proof}
This is an easy calculation using \eqref{ET4.3a}--\eqref{ET4.3b}.
\end{proof}

\begin{lemma}\label{LA2.2}
There exists a unique function $V^{\varepsilon}\in\Cc^{2}(\RR^{2})$
with $V^{\varepsilon}(0)=0$,
which is bounded below in $\RR^{2}$,
and solves the HJB
\begin{equation}\label{ET4.3c}
\min_{u\in\Act}\;\bigl[\Lg^{u}V^{\varepsilon}(x)
+ h^\epsilon(x,u)\bigr]
\;=\; \varrho_\epsilon\,,\quad x\in\RR^{2}\,.
\end{equation}
where $\varrho_{\varepsilon}\df\inf_{\uppi\in\eom}\;\uppi(h^\epsilon)$,
and the usual characterization of optimality holds.
Moreover,
\begin{itemize}
\item[\upshape{(}a\upshape{)}]
for every $R>0$, there exists a constant $k_{R}>0$ such that
\begin{equation*}
\sup_{\epsilon\in(0,1)}\;\osc_{B_{R}}\;V^{\varepsilon}\;\le\;k_{R} \,;
\end{equation*}

\item[\upshape{(}b\upshape{)}]
if $v_{\varepsilon}$ is a measurable a.e.\ selector
from the minimizer of the Hamiltonian in \eqref{ET4.3c}, then for any $\delta>0$,
we have
\begin{equation*}
V^{\varepsilon}(x)\;\ge\;\Exp^{v_{\varepsilon}}_{x}
\biggl[\int_{0}^{\tc_{\delta}}\bigl(h^\epsilon(X_{s},v_{\varepsilon}(X_{s}))
-\varrho_{\varepsilon}\bigr)\,\D{s}\biggr]
+\inf_{B_{\delta}}\;V^{\varepsilon}\,;
\end{equation*}

\item[\upshape{(}c\upshape{)}]
for any stationary control $v\in\Usm^\star$ and for any $\delta>0$,
it holds that
\begin{equation*}
V^{\varepsilon}(x)\;\le\;
\Exp^{v}_{x}\biggl[\int_{0}^{\tc_{\delta}}
\bigl(h^\epsilon\bigl(X_{s},v(X_{s})\bigr)
-\varrho_{\varepsilon}\bigr)\,\D{s} + V^{\varepsilon}(X_{\tc_{\delta}})\biggr]\,.
\end{equation*}
\end{itemize}
\end{lemma}

\begin{proof}
The proof follows along the lines of Theorem~3.3 in \citet{ABP14},
using the fact that $h^\epsilon$ is inf-compact, for each $\epsilon>0$,
and $\inf_{\uppi\in\eom}\;\uppi(h^\epsilon)<\infty$ by
Lemmas~\ref{L4.1} and \ref{LA2.1}.
There is one important difference though:  the running cost
$h^\epsilon$ is not bounded below uniformly in $\epsilon>0$,
and the estimate in part (a) needs special attention.
By \eqref{ET4.3a}--\eqref{ET4.3b}, using It\^o's formula, we obtain
\begin{equation*}
\Exp^{U}_{x}\biggl[\int_{0}^{\infty} \E^{-\alpha s}\,
(1+\theta)\lambda^* r_i(X_{s},U_{s})\,\D{s}\biggr]
\;\le\;\sV(x)+ \Exp^{U}_{x}\biggl[\int_{0}^{\infty}\E^{-\alpha s}\,
\bigl(\kappa_0 + \tfrac{\kappa_1}{2}+
\tfrac{1}{2}r_{\mathsf{o}}(X_{s},U_{s})\bigr)\,\D{s}\biggr]
\end{equation*}
for all $U\in\Uadm$, $\alpha>0$.
It follows that, given any ball $B_R$, the discounted value function
\begin{equation*}
\Tilde{V}^\epsilon_{\alpha}(x)\;\df\;
\inf_{U\in\Uadm}\;
\Exp^{U}_{x}\biggl[\int_{0}^{\infty}\E^{-\alpha s}\,
\bigl(2\kappa_0 + \kappa_1+
h^\epsilon(X_{s},U_{s})\bigr)\,\D{s}\biggr]
\end{equation*}
is strictly positive on $B_R$ for all sufficiently small $\alpha>0$.
Therefore, by adding the constant $2\kappa_0 + \kappa_1$ to the running
cost, we obtain estimates on the oscillation of $\Tilde{V}^\epsilon$
that are uniform over $\epsilon>0$
by Lemmas~3.5 and 3.6 of \cite{ABP14}.
\end{proof}

The next lemma completes the proof of Theorem~\ref{T4.3}.

\begin{lemma}%\label{LA2.3}
Let $V^{\varepsilon}$ and $\varrho_{\varepsilon}$,
for $\varepsilon>0$, be as in Lemma~\ref{LA2.2}.
The following hold:
\begin{itemize}
\item[\upshape{(}i\upshape{)}]
The function $V^{\varepsilon}$ converges
to some $V_{\mathsf{f}}\in \Cc^{2}(\RR^{2})$, uniformly on compact sets,
and $\varrho_{\varepsilon}\to\varrho^*_{\mathsf{f}}$, as $\varepsilon\searrow0$,
and $V_{\mathsf{f}}$ satisfies
\begin{equation}\label{ET4.3d}
\min_{u\in\Act}\;\bigl[\Lg^{u}V_{\mathsf{f}}(x)+h_{\uptheta,\uplambda^{*}}(x,u)\bigr]
\;=\;\varrho^*_{\mathsf{f}}\;=\; \uppi^{*}(h_{\uptheta,\uplambda^{*}})\,.
\end{equation}
Also, any limit point $v^{*}$ $($in the topology of Markov controls$)$
as $\varepsilon\searrow0$ of measurable selectors
$\{v_{\varepsilon}\}$ from the minimizer of \eqref{ET4.3c} satisfies
\begin{equation*}
\Lg^{v^{*}}V_{\mathsf{f}}(x)+h_{\uptheta,\uplambda^{*}}(x,v^{*}(x))\;=\;
\varrho^{*}_{\mathsf{f}}\quad\text{a.e.\ in~} \RR^{2}\,.
\end{equation*}

\item[\upshape{(}ii\upshape{)}]
A stationary Markov control $v\in\Usm$
is optimal if and only if it satisfies 
\begin{equation}\label{ET4.3e}
H_{h_{\uptheta,\uplambda^{*}}}\bigl(x,\nabla V_{\mathsf{f}}(x)\bigr) \;=\;
b\bigl(x, v(x)\bigr)\cdot \nabla V_{\mathsf{f}}(x)
+ h_{\uptheta,\uplambda^{*}}\bigl(x,v(x)\bigr)
\quad\text{a.e.~in~} \RR^{2}\,,
\end{equation}
where $H_{h_{\uptheta,\uplambda^{*}}}$ is defined in \eqref{E-H} with $r$
replaced by $h_{\uptheta,\uplambda^{*}}$.

\item[\upshape{(}iii\upshape{)}]
The function $V_{*}$ has the stochastic representation
\begin{align*}
V_{\mathsf{f}}(x)&\;=\;
\lim_{\delta\searrow0}\;\inf_{v\,\in\,\Usm^{\star}}\;
\Exp^{v}_{x}\biggl[\int_{0}^{\tc_{\delta}}
\bigl(h_{\uptheta,\uplambda^{*}}
\bigl(X_{s},v(X_{s})\bigr)-\varrho_{*}\bigr)\,\D{s}\biggr]\\[5pt]
&\;=\;
\lim_{\delta\searrow0}\;\Exp^{\Bar{v}}_{x}\biggl[\int_{0}^{\tc_{\delta}}
\bigl(h_{\uptheta,\uplambda^{*}}
\bigl(X_{s},\Bar{v}(X_{s})\bigr)-\varrho_{*}\bigr)\,\D{s}\biggr]\nonumber
\end{align*}
for any $\Bar{v}\in\Usm$ that satisfies \eqref{ET4.3e}.
\end{itemize}
\end{lemma}

\begin{proof}
We follow the method in the proof of Theorem~3.4 in \cite{ABP14}.
Since $\varrho_\epsilon$ is non-increasing and bounded below,
it converges to some value which is clearly
$\uppi^{*}(h_{\uptheta,\uplambda^{*}})$ by Lemma~\ref{LA2.1}.
Parts (i) and (iii) then follow as in the proof of Lemma~3.9 in \cite{ABP14},
and we can follow the method in the proof of Lemma~3.10 in the same paper
to establish that $V_{\mathsf{f}}^- \in\sorder(\sV)$.

Now let $\Hat{v}\in\Usm$ be any control satisfying \eqref{ET4.3e}.
We modify the estimate in \eqref{ET4.3b} and write it as
$(1+\theta)\lambda^* r_i(x,u)\;\le\; \kappa_0 + \frac{\kappa_2}{4}\,
\abs{x}^m$
for some constant $\kappa'_0$.
An easy calculation using \eqref{ET4.3a}
then shows that
\begin{equation*}
\Lg^{\Hat{v}}(\sV+2V_{\mathsf{f}}) \;\le\;
\kappa_0+\kappa_1 +2 \kappa'_0 - \frac{\kappa_2}{2} \abs{x}^m
-h_{\uptheta,\uplambda^{*}}\bigl(x,\Hat{v}(x)\bigr)\,.
\end{equation*}
Therefore, since $\sV+2V_{\mathsf{f}}$ is inf-compact, 
we must have $\Hat{v}\in\Usm^\star$.
Using this and the fact that $V_{\mathsf{f}}^- \in\sorder(\sV)$,
we deduce that $\frac{1}{T} \Exp^{\Hat{v}}_{x}\bigl[V^-_{\mathsf{f}}(X_T)] \to 0$
as $T\to\infty$.
Hence, by It\^o's formula and \eqref{ET4.3d} we obtain
$\uppi_{\Hat{v}}(h_{\uptheta,\uplambda^{*}})\le\varrho^*_{\mathsf{f}}$.
Thus we must have equality
$\uppi_{\Hat{v}}(h_{\uptheta,\uplambda^{*}})=
\uppi^*(h_{\uptheta,\uplambda^{*}})$,
i.e., $\Hat{v}$ is optimal.
This completes the proof.
\end{proof}

\smallskip
%%%%%%%%%%%%%%%%%%%%%%%%%%%%%%%%%%%%%%%%%%%%%%%%%%%%%%%%%%%%%%%%%%%%%
\section*{Acknowledgements}
The authors thank the referees for the helpful comments that have improved the paper.  
This research was supported in part by the Army Research Office under
grant W911NF-17-1-0019.
The work of Ari Arapostathis was also supported in part by the Office of Naval
Research through grant N00014-14-1-0196.
The work of Guodong Pang is also supported in part by the Marcus Endowment Grant at the 
Harold and Inge Marcus Department of Industrial and Manufacturing Engineering 
at Penn State.


\begin{thebibliography}{30}
\providecommand{\natexlab}[1]{#1}
\providecommand{\url}[1]{\texttt{#1}}
\expandafter\ifx\csname urlstyle\endcsname\relax
  \providecommand{\doi}[1]{doi: #1}\else
  \providecommand{\doi}{doi: \begingroup \urlstyle{rm}\Url}\fi

\bibitem[Arapostathis and Pang(2016)]{AP15}
A.~Arapostathis and G.~Pang.
\newblock Ergodic diffusion control of multiclass multi-pool networks in the
  {H}alfin-{W}hitt regime.
\newblock \emph{Ann. Appl. Probab.}, 26\penalty0 (5):\penalty0 3110--3153,
  2016.

\bibitem[Arapostathis et~al.(2012)Arapostathis, Borkar, and Ghosh]{book}
A.~Arapostathis, V.~S. Borkar, and M.~K. Ghosh.
\newblock \emph{Ergodic control of diffusion processes}, volume 143 of
  \emph{Encyclopedia of Mathematics and its Applications}.
\newblock Cambridge University Press, Cambridge, 2012.

\bibitem[Arapostathis et~al.(2015)Arapostathis, Biswas, and Pang]{ABP14}
A.~Arapostathis, A.~Biswas, and G.~Pang.
\newblock Ergodic control of multi-class ${M/M/N+M}$ queues in the
  {H}alfin-{W}hitt regime.
\newblock \emph{Ann. Appl. Probab.}, 25\penalty0 (6):\penalty0 3511--3570,
  2015.

\bibitem[Armony(2005)]{Armony-05}
M.~Armony.
\newblock Dynamic routing in large-scale service systems with heterogeneous
  servers.
\newblock \emph{Queueing Syst.}, 51\penalty0 (3-4):\penalty0 287--329, 2005.

\bibitem[Armony and Ward(2010)]{AW-10}
M.~Armony and A.~R. Ward.
\newblock Fair dynamic routing in large-scale heterogeneous-server systems.
\newblock \emph{Oper. Res.}, 58\penalty0 (3):\penalty0 624--637, 2010.

\bibitem[Atar(2005{\natexlab{a}})]{Atar-05a}
R.~Atar.
\newblock A diffusion model of scheduling control in queueing systems with many
  servers.
\newblock \emph{Ann. Appl. Probab.}, 15\penalty0 (1B):\penalty0 820--852,
  2005{\natexlab{a}}.

\bibitem[Atar(2005{\natexlab{b}})]{Atar-05b}
R.~Atar.
\newblock Scheduling control for queueing systems with many servers: asymptotic
  optimality in heavy traffic.
\newblock \emph{Ann. Appl. Probab.}, 15\penalty0 (4):\penalty0 2606--2650,
  2005{\natexlab{b}}.

\bibitem[Atar et~al.(2004)Atar, Mandelbaum, and Reiman]{atar-mandel-rei}
R.~Atar, A.~Mandelbaum, and M.~I. Reiman.
\newblock Scheduling a multi class queue with many exponential servers:
  asymptotic optimality in heavy traffic.
\newblock \emph{Ann. Appl. Probab.}, 14\penalty0 (3):\penalty0 1084--1134,
  2004.

\bibitem[Atar et~al.(2009)Atar, Mandelbaum, and Shaikhet]{Atar-09}
R.~Atar, A.~Mandelbaum, and G.~Shaikhet.
\newblock Simplified control problems for multiclass many-server queueing
  systems.
\newblock \emph{Math. Oper. Res.}, 34\penalty0 (4):\penalty0 795--812, 2009.

\bibitem[Bell and Williams(2001)]{bell-williams-2001}
S.~L. Bell and R.~J. Williams.
\newblock Dynamic scheduling of a system with two parallel servers in heavy
  traffic with resource pooling: asymptotic optimality of a threshold policy.
\newblock \emph{Ann. Appl. Probab.}, 11\penalty0 (3):\penalty0 608--649, 2001.

\bibitem[Biswas(2015)]{Biswas-15}
A.~Biswas.
\newblock An ergodic control problem for many-sever multi-class queueing
  systems with help.
\newblock \emph{arXiv}, 1502.02779v2, 2015.

\bibitem[Dai and Tezcan(2008)]{dai-tezcan-08}
J.~G. Dai and T.~Tezcan.
\newblock Optimal control of parallel server systems with many servers in heavy
  traffic.
\newblock \emph{Queueing Syst.}, 59\penalty0 (2):\penalty0 95--134, 2008.

\bibitem[Dai and Tezcan(2011)]{dai-tezcan-11}
J.~G. Dai and T.~Tezcan.
\newblock State space collapse in many-server diffusion limits of parallel
  server systems.
\newblock \emph{Math. Oper. Res.}, 36\penalty0 (2):\penalty0 271--320, 2011.

\bibitem[Gamarnik and Stolyar(2012)]{gamarnik-stolyar}
D.~Gamarnik and A.~L. Stolyar.
\newblock Multiclass multiserver queueing system in the {H}alfin-{W}hitt heavy
  traffic regime: asymptotics of the stationary distribution.
\newblock \emph{Queueing Syst.}, 71\penalty0 (1-2):\penalty0 25--51, 2012.

\bibitem[Ghamami and Ward(2013)]{GW13}
S.~Ghamami and A.~R. Ward.
\newblock Dynamic scheduling of a two-server parallel server system with
  complete resource pooling and reneging in heavy traffic: {A}symptotic
  optimality of a two-threshold policy.
\newblock \emph{Math. Oper. Res.}, 38\penalty0 (4):\penalty0 761--824, 2013.

\bibitem[Gurvich and Whitt(2009)]{GW09}
I.~Gurvich and W.~Whitt.
\newblock Queue-and-idleness-ratio controls in many-server service systems.
\newblock \emph{Math. Oper. Res.}, 34\penalty0 (2):\penalty0 363--396, 2009.

\bibitem[Gurvich and Whitt(2010)]{GW10}
I.~Gurvich and W.~Whitt.
\newblock Service-level differentiation in many-server service system via
  queue-ratio routing.
\newblock \emph{Oper. Res.}, 58\penalty0 (2):\penalty0 316--328, 2010.

\bibitem[Harrison(1988)]{harrison-1988}
J.~M. Harrison.
\newblock Brownian models of queueing networks with heterogeneous customer
  populations.
\newblock In \emph{Stochastic differential systems, stochastic control theory
  and applications ({M}inneapolis, {M}inn., 1986)}, volume~10 of \emph{IMA Vol.
  Math. Appl.}, pages 147--186. Springer, New York, 1988.

\bibitem[Harrison(1998)]{harrison-1998}
J.~M. Harrison.
\newblock Heavy traffic analysis of a system with parallel servers: Asymptotic
  optimality of discrete-review policies.
\newblock \emph{Ann. Appl. Probab.}, 8:\penalty0 822--848, 1998.

\bibitem[Harrison(2000)]{harrison-2000}
J.~M. Harrison.
\newblock Brownian models of open processing networks: Canonical representation
  of workload.
\newblock \emph{Ann. Appl. Probab.}, 10\penalty0 (1):\penalty0 75--103, 2000.

\bibitem[Kallenberg(2002)]{kallenberg}
O.~Kallenberg.
\newblock \emph{Foundations of modern probability}.
\newblock Probability and its Applications (New York). Springer-Verlag, New
  York, second edition, 2002.

\bibitem[Luenberger(1967)]{Luenberger}
D.~G. Luenberger.
\newblock \emph{Optimization by vector space methods}.
\newblock John Wiley \& Sons Inc., New York, 1967.

\bibitem[Stolyar(2015)]{stolyar-15}
A.~L. Stolyar.
\newblock Diffusion-scale tightness of invariant distributions of a large-scale
  flexible service system.
\newblock \emph{Adv. in Appl. Probab.}, 47\penalty0 (1):\penalty0 251--269,
  2015.

\bibitem[Stolyar and Yudovina(2012{\natexlab{a}})]{stolyar-yudovina-12a}
A.~L. Stolyar and E.~Yudovina.
\newblock Tightness of invariant distributions of a large-scale flexible
  service system under a priority discipline.
\newblock \emph{Stoch. Syst.}, 2\penalty0 (2):\penalty0 381--408,
  2012{\natexlab{a}}.

\bibitem[Stolyar and Yudovina(2012{\natexlab{b}})]{stolyar-yudovina-12b}
A.~L. Stolyar and E.~Yudovina.
\newblock Systems with large flexible server pools: instability of ``natural"
  load balancing.
\newblock \emph{Ann. Appl. Probab.}, 23\penalty0 (5):\penalty0 2099--2183,
  2012{\natexlab{b}}.

\bibitem[Tezcan and Dai(2010)]{TD-10}
T.~Tezcan and J.~G. Dai.
\newblock Dynamic control of {N}-systems with many servers: asymptotic
  optimality of a static priority policy in heavy traffic.
\newblock \emph{Oper. Res.}, 58\penalty0 (1):\penalty0 94--110, 2010.

\bibitem[Ward and Armony(2013)]{WA-13}
A.~R. Ward and M.~Armony.
\newblock Blind fair routing in large-scale service systems with heterogeneous
  customers and servers.
\newblock \emph{Oper. Res.}, 61\penalty0 (1):\penalty0 228--243, 2013.

\bibitem[Williams(2000)]{williams-2000}
R.~J. Williams.
\newblock On dynamic scheduling of a parallel server system with complete
  resource pooling.
\newblock \emph{Analysis of Communication Networks: Call Centres, Traffic and
  Performance. Fields Inst. Commun. Amer. Math. Soc., Providence, RI.},
  28:\penalty0 49--71, 2000.

\bibitem[Williams(2016)]{williams-2016}
R.~J. Williams.
\newblock Stochastic processing networks.
\newblock \emph{Annual Review of Statistics and Its Application}, 3, 2016.

\bibitem[Xu et~al.(1992)Xu, Righter, and Shanthikumar]{XRS92}
S.~H. Xu, R.~Righter, and J.~G. Shanthikumar.
\newblock Optimal dynamic assignment of customers to heterogeneous servers in
  parallel.
\newblock \emph{Oper. Res.}, 40(6):\penalty0 1126--1138, 1992.

\end{thebibliography}
\end{document}